\documentclass[twoside,centertags                 
]{amsart}

\def\Picardgroupoid{Picard groupoid}
\def\Picardgroupoids{Picard groupoids}

\def\exact{distinguished}
\def\Exact{Distinguished}

\def\Picardcategory{\Picardgroupoid}\def\Picardcategories{\Picardgroupoids}

\usepackage[cmtip,color,all,
]{xy}

\entrymodifiers={+!!<0pt,\fontdimen22\textfont2>}

\usepackage{bbm}
\usepackage{amssymb}
\usepackage{eucal}
\usepackage{stmaryrd}
\usepackage{graphicx}
\usepackage{mathrsfs}
\usepackage{wasysym}

\newtheorem{thm}[equation]{Theorem}
\newtheorem{cor}[equation]{Corollary}
\newtheorem{lem}[equation]{Lemma}
\newtheorem{prop}[equation]{Proposition}
\theoremstyle{definition}
\newtheorem{defn}[equation]{Definition}
\theoremstyle{remark}
\newtheorem{rem}[equation]{Remark}
\newtheorem{exm}[equation]{Example}

\newcommand{\abs}[1]{\left\vert#1\right\vert}

\newcommand{\eps}{\varepsilon}
\newcommand{\To }{\longrightarrow}

\def\filledsquare{{\scriptscriptstyle 
\blacklozenge}}

\newcommand{\C}[1]{\mathscr{#1}}

\newcommand{\D}[1]{\mathcal{#1}}
\newcommand{\DD}[1]{\mathcal{#1}^{\mathrm{der}}}
\newcommand{\der}{\mathbb{D}}

\newcommand{\EZDIAG}[5]{\xymatrix@C+=1.5cm{*+[r]{#1}
\ar@(u,l)_(0.62){\scriptstyle #5}[]
\ar@<.5ex>^-{#3}[r]&\ar@<.5ex>^-{#4}[l]#2}}

\def\op{\mathrm{op}}
\def\tr{\mathrm{tr}}
\def\abb{{ab}}

\def\r{\rightarrow} 
\def\l{\leftarrow} 
\def\rr{\Rightarrow} 

\def\into{\rightarrowtail}
\def\onto{\twoheadrightarrow}

\def\hom{\operatorname{Hom}}

\def\ho{\operatorname{Ho}}
\newcommand{\iso}[1]{\operatorname{iso}(#1)}

\def\cof{\operatorname{cof}}
\def\we{\operatorname{we}}

\def\ext{\operatorname{Ext}}

\def\aut{\operatorname{Aut}}

\def\rank{\operatorname{rk}}

\def\X{X}
\def\Y{Y}
\def\Z{Z}
\newcommand{\cone}[1]{C^{#1}}

\renewcommand{\det}{\operatorname{det}}

\def\st{\stackrel} 

\def\coker{\operatorname{Coker}}
\renewcommand{\ker}{\operatorname{Ker}}
\newcommand{\im}{\operatorname{Im}}

\def\ZZ{\mathbb{Z}}

\def\abg#1{\langle #1\rangle^\abb}
\def\nig#1{\langle #1\rangle^{nil}}

\def\we{\mathrm{we}}

\newcommand{\grupo}[1]{\langle #1\rangle}

\numberwithin{equation}{subsection}

\newcommand{\ass}{\text{ass.}}
\newcommand{\com}{\text{comm.}}
\newcommand{\comm}[3]{\text{comm}_{#1,#2}}

\newcommand{\hp}[3]{\text{mult}_{#1,#2}}
\newcommand{\hpp}{\text{mult.}}
\newcommand{\up}[1]{\text{unit.}}

\renewcommand{\sp}{{\_\,}}

\newdir{ >}{{}*!/-7pt/@{>}}
\newdir{> }{{}*!/10pt/@{>}}
\newdir{>> }{{}*!/10pt/@{>>}}

\def\tetra#1#2#3#4#5{{\entrymodifiers={+0}\xymatrix@!R=2pc@!C=4pc{
    #5\\#5\\#5\\#5
    \ar@{-}"#1";"#2"|*+{\scriptscriptstyle A}="a"
    \ar@{-}"#1";"#4"|*+{\scriptscriptstyle C}="c"
    \ar@{-}"#2";"#3"|*+{\scriptscriptstyle B/A}="ba"
    \ar@{-}"#2";"#4"|*+{\scriptscriptstyle C/A}="ca"
    \ar@{-}"#3";"#4"|*+{\scriptscriptstyle C/B}="cb"
    \ar@{--}"#1";"#3"|!{"a";"c"}\hole|*+{\scriptscriptstyle B}="b"
                      |!{"#2";"#4"}\hole|!{"ba";"ca"}\hole
    \ar@{ -->> }"b";"ba"|!{"#2";"#4"}\hole
    \ar@{ -->> }"c";"cb"|!{"#2";"#4"}\hole
    \ar@{ >--> }"a";"b"
    \ar@{ >-> } "a";"c"
    \ar@{ >--> }"b";"c"
    \ar@{ >-> } "ba";"ca"
    \ar@{ ->> } "c";"ca"
    \ar@{ ->> } "ca";"cb"}}}

\def\tetri#1#2#3#4#5{{\entrymodifiers={+0}\xymatrix@!R=1pc@!C=2pc{
    #5\\#5\\#5\\#5
    \ar@{}"#1";"#2"|*+{A}="a"
    \ar@{}"#1";"#4"|*+{C}="c"
    \ar@{}"#2";"#3"|*+{B/A}="ba"
    \ar@{}"#2";"#4"|*+{C/A}="ca"
    \ar@{}"#3";"#4"|*+{C/B}="cb"
    \ar@{}"#1";"#3"|!{"a";"c"}\hole|*+{B}="b"
                      |!{"#2";"#4"}\hole|!{"ba";"ca"}\hole
    \ar@{->>}"b";"ba" 
    \ar@{->>}"c";"cb" 
    \ar@{>->}"a";"b"_-f
    \ar@{>->} "a";"c"^-{gf}
    \ar@{>->}"b";"c"^-g
    \ar@{>->} "ba";"ca"
    \ar@{->>} "c";"ca"
    \ar@{->>} "ca";"cb"}}}

\def\ob{\operatorname{Ob}}

\def\nattrans{\Rightarrow}

\def\arDEF{\def\arAB{}\def\arAC{}\def\arBC{}\def\arBD{}\def\arCF{}\def\arDE{}\def\arEF{}\def\arCE{}\def\arEA{}\def\arFD{}\def\arEF{}\def\arFB{}\def\arDA{}\def\arSIZE{\scriptstyle}}

\usepackage{keyval}\makeatletter
\define@key{tetravar}{arSIZE}{\def\arSIZE{#1}}
\define@key{tetravar}{objA}{\def\objA{#1}}
\define@key{tetravar}{objB}{\def\objB{#1}}
\define@key{tetravar}{objC}{\def\objC{#1}}
\define@key{tetravar}{objD}{\def\objD{#1}}
\define@key{tetravar}{objE}{\def\objE{#1}}
\define@key{tetravar}{objF}{\def\objF{#1}}
\define@key{tetravar}{arAB}{\def\arAB{\arSIZE #1}}
\define@key{tetravar}{arBC}{\def\arBC{\arSIZE #1}}
\define@key{tetravar}{arAC}{\def\arAC{\arSIZE #1}}
\define@key{tetravar}{arCE}{\def\arCE{\arSIZE #1}}
\define@key{tetravar}{arCF}{\def\arCF{\arSIZE #1}}
\define@key{tetravar}{arDE}{\def\arDE{\arSIZE #1}}
\define@key{tetravar}{arEF}{\def\arEF{\arSIZE #1}}
\define@key{tetravar}{arFD}{\def\arFD{\arSIZE #1}}
\define@key{tetravar}{arEA}{\def\arEA{\arSIZE #1}}
\define@key{tetravar}{arDA}{\def\arDA{\arSIZE #1}}
\define@key{tetravar}{arFB}{\def\arFB{\arSIZE #1}}
\define@key{tetravar}{arBD}{\def\arBD{\arSIZE #1}}
\makeatother

\def\minusone{\scriptscriptstyle+1}

\newcommand\octahedron[1]{\mbox{\arDEF\setkeys{tetravar}{#1}
\hspace*{-33mm}$\begin{array}[m]{c}\\[-22mm]
\entrymodifiers={+0}\vcenter{\xymatrix@R=5.8pc@C=4.6pc{
    &&&&&\\&&&&&\\&&&&&\\&&&&&
    \ar@{}"3,1";"4,4"|*+{\objA}="a"
    \ar@{}"3,1";"1,4"|*+{\objC}="c"
    \ar@{}"4,4";"3,6"|*+{\objD}="ba"
    \ar@{}"4,4";"1,4"|*+{\objE}="ca"
    \ar@{}"3,6";"1,4"|*+{\objF}="cb"
    \ar@{}"3,1";"3,6"|!{"a";"c"}\hole|<(.6)*+{\objB}="b"
                      |!{"4,4";"1,4"}\hole|!{"ba";"ca"}\hole
    \ar@{.>}"b";"ba"_-{\arBD} 
    \ar"c";"cb"^-{\arCF} 
    \ar@{.>}"a";"b"_(.6){\!\!\arAB}
    \ar"a";"c"^-{\arAC}
    \ar"ba";"ca"_(.4){\!\!\!\arDE}
    \ar"c";"ca"^(0.65){\!\!\!\!\arCE}
    \ar"ca";"cb"^(.4){\arEF\!\!\!\!}
    \ar"cb";"ba"^<(.6){\arFD}|{\minusone}
    \ar"ba";"a"|{\minusone}^(.6){\arDA}
    \ar"ca";"a"_<(.6){\arEA}|{\minusone}
    \ar@{.>}"b";"c"^-{\arBC}|!{"ca";"a"}{\hole}
    \ar@{.>}"cb";"b"^(.55){\!\!\!\arFB}|(.4){\minusone}|!{"ba";"ca"}{\hole}
}}
\\[-7mm]\end{array}$\hspace*{-20mm}}}

\newcommand\octahedronles[1]{\mbox{\arDEF\setkeys{tetravar}{#1}
\hspace*{-33mm}$\begin{array}[m]{c}\\[-22mm]
\entrymodifiers={+0}\vcenter{\xymatrix@R=5.8pc@C=4.6pc{
    &&&&&\\&&&&&\\&&&&&\\&&&&&
    \ar@{}"3,1";"4,4"|*+{\objA}="a"
    \ar@{}"3,1";"1,4"|*+{\objC}="c"
    \ar@{}"4,4";"3,6"|*+{\objD}="ba"
    \ar@{}"4,4";"1,4"|*+{\objE}="ca"
    \ar@{}"3,6";"1,4"|*+{\objF}="cb"
    \ar@{}"3,1";"3,6"|!{"a";"c"}\hole|<(.6)*+{\objB}="b"
                      |!{"4,4";"1,4"}\hole|!{"ba";"ca"}\hole
    \ar@{.>}"b";"ba"_-{\arBD} 
    \ar"c";"cb"^-{\arCF} 
    \ar@{.>}"a";"b"_(.6){\!\!\arAB}
    \ar"a";"c"^-{\arAC}
    \ar"ba";"ca"_(.4){\!\!\!\arDE}
    \ar"c";"ca"^(0.65){\!\!\!\!\arCE}
    \ar"ca";"cb"^(.4){\arEF\!\!\!\!}
    \ar"cb";"ba"^<(.6){\arFD}|{\minusone}
    \ar"ba";"a"|{\minusone}^(.6){\arDA}
    \ar"ca";"a"_<(.6){\arEA}|{\minusone}
    \ar@{.>}"b";"c"^-{\arBC}|!{"ca";"a"}{\hole}
    \ar@{.>}"cb";"b"^(.55){\!\!\!\arFB}|(.4){\minusone}|!{"ba";"ca"}{\hole}
}}
\\[-7mm]\end{array}$\hspace*{-20mm}}}

\newcommand\octahedronada[1]{\mbox{\arDEF\setkeys{tetravar}{#1}
\hspace*{-33mm}$\begin{array}[m]{c}\\[-22mm]
\entrymodifiers={+0}\vcenter{\xymatrix@R=5.8pc@C=4.6pc{
    &&&&&\\&&&&&\\&&&&&\\&&&&&
    \ar@{}"3,1";"4,4"|*+{\objA}="a"
    \ar@{}"3,1";"1,4"|*+{\objC}="c"
    \ar@{}"4,4";"3,6"|*+{\objD}="ba"
    \ar@{}"4,4";"1,4"|*+{\objE}="ca"
    \ar@{}"3,6";"1,4"|*+{\objF}="cb"
    \ar@{}"3,1";"3,6"|!{"a";"c"}\hole|<(.6)*+{\objB}="b"
                      |!{"4,4";"1,4"}\hole|!{"ba";"ca"}\hole
    \ar@{.>}"b";"ba"_-{\arBD} 
    \ar"c";"cb"^-{\arCF} 
    \ar@{.>}"a";"b"_(.6){\!\!\arAB}
    \ar"a";"c"^-{\arAC}
    \ar"ba";"ca"_(.4){\!\!\!\arDE}
    \ar"c";"ca"^(0.65){\!\!\!\!\arCE}
    \ar"ca";"cb"^(.4){\arEF\!\!\!\!}
    \ar"cb";"ba"^<(.6){\arFD}
    \ar"ba";"a"^(.6){\arDA}
    \ar"ca";"a"_<(.6){\arEA}
    \ar@{.>}"b";"c"^-{\arBC}|!{"ca";"a"}{\hole}
    \ar@{.>}"cb";"b"^(.55){\!\!\!\arFB}|!{"ba";"ca"}{\hole}
}}
\\[-7mm]\end{array}$\hspace*{-20mm}}}

\begin{document}



\title
{On determinant functors and $K$-theory}%
\author{Fernando Muro}%
\address{Universidad de Sevilla,
Facultad de Matem\'aticas,
Departamento de \'Algebra,
Avda. Reina Mercedes s/n,
41012 Sevilla, Spain}
\email{fmuro@us.es}
\author{Andrew Tonks}%
\address{STORM, School of Computing, London Metropolitan University, 166--220
Holloway Road, London N7 8DB, United Kingdom}
\email{a.tonks@londonmet.ac.uk}
\author{Malte Witte}%
\address{Ruprecht-Karls-Universit\"at Heidelberg,
Mathematisches Institut,
Im Neuenheimer Feld 288,
69120 Heidelberg, Germany}
\email{witte@mathi.uni-heidelberg.de}

\thanks{The first and second authors were partially supported
by the Spanish Ministry of Education and
Science under the MEC-FEDER grants  MTM2007-63277 and MTM2010-15831, and by the Government of Catalonia under the grant SGR-119-2009. The first author was also partially supported by the Spanish Ministry of Science and Innovation under a Ram\'on y Cajal research contract and by the Andalusian Ministry of Economy, Innovation and Science under the grant FQM-5713.}

\def\g#1{\save
[].[ddrr]!C="g#1"*+[F.]\frm{}\restore}

\def\gf#1{\save
[].[ddrr]!C="g#1"*+++[F.]\frm{}\restore}

\def\gfh#1{\save
[].[ddrr]!C="g#1"*+<.7cm>[F.]\frm{}\restore}

\begin{abstract}
We extend Deligne's notion of determinant functor to Waldhausen categories and (strongly) triangulated categories. We construct explicit universal determinant functors in each case, whose target is an algebraic model for the $1$-type of the corresponding $K$-theory spectrum. As applications, we answer open questions by Maltsiniotis and Neeman on the $K$-theory of (strongly) triangulated categories and a question of Grothendieck to Knudsen on determinant functors. We also prove additivity theorems for low-dimensional $K$-theory of (strongly) triangulated categories and obtain generators and (some) relations for various $K_{1}$-groups. This is achieved via a unified theory of determinant functors which can be applied in further contexts, such as derivators.
\end{abstract}

\maketitle
\tableofcontents



\section*{Introduction}

Determinant functors, considered first by Knudsen and Mumford \cite{pmssc}, categorify the usual notion of determinant of
invertible matrices. The most elementary instance of such a functor
sends a finite-dimensional vector space $V$ to the pair
\begin{align*}
\det V&=(\dim V,\wedge^{\dim V}V).
\end{align*}
The highest exterior power of an automorphism $f\colon V\cong V$ with matrix $A$ with respect to some basis is multiplication by the determinant, $\wedge^{\dim V}f=\det A$.


Deligne \cite{dc} axiomatised the properties of this functor in his definition of determinant functor $\det\colon \C E\r \C P$ on an exact category $\C{E}$ with values 
in  a \Picardgroupoid\  $\C{P}$. 
As a functor, $\det$ is only defined on isomorphisms, $\det\colon \iso{\C E}\r \C P$, but short exact sequences  $\X\into \Y\onto C$ induce natural isomorphisms
$$\det(C)\otimes\det(\X){\cong} \det(\Y),$$
called \emph{additivity data}, 
which must satisfy some coherence laws.

Deligne  constructed  a \Picardgroupoid\ of \emph{`virtual objects'}
$V(\C{E})$ which is the target of a {\em universal} determinant functor $\det\colon \C E\r V(\C E)$, in
the sense that any other determinant functor factors through this one in an essentially unique way.
The group of isomorphism classes of objects in 
$V(\C{E})$ is Quillen's 
$K_0(\C{E})$ and the automorphism group of the tensor unit is 
$K_1(\C{E})$. This shows that any interesting exact category has
highly non-trivial determinant functors.

Knudsen \cite{dfec} showed  by elementary
methods that determinant functors on an exact category $\C{E}$ extend
to the category of bounded complexes 
$C^b(\C{E})$ in an essentially unique way, generalising  results with Mumford~\cite{pmssc}. 
This extension is not only defined on isomorphisms, but on quasi-isomorphisms in $C^b(\C{E})$. 

Quasi-isomorphisms are the weak equivalences of a Waldhausen category structure on $C^b(\C{E})$. Therefore, Knudsen's extension theorem hints at  the existence of a theory of determinant functors  $\det\colon \C W\r \C P$ for Waldhausen categories $\C W$. This theory is developed in this paper. In particular, we construct a  universal determinant functor $\det\colon \C W\r V(\C W)$ where Waldhausen's $K_0(\C{W})$ is the group of isomorphism classes of objects in 
$V(\C{W})$ and 
$K_1(\C{W})$ is the automorphism group of the tensor unit. Knudsen's results follow then from the Gillet--Waldhausen isomorphism $K_{*}(\C E)\cong K_{*}(C^{b}(\C E))$. A theory of determinant functors for Waldhausen categories with values in \emph{strict} categorical groups has also been developed in \cite{malte}.

It is a common practice to pass from $C^{b}(\C E)$ to the derived category $D^{b}(\C E)$, inverting quasi-isomorphisms. The underlying functor $\det\colon \we(C^{b}(\C E))\r \C P$ of a determinant functor $\det\colon C^{b}(\C E)\r \C P$ factors uniquely through $\det\colon \iso{D^{b}(\C E)}\r \C P$. It would be desirable to enrich this functor with additivity data associated to \exact\ triangles  $\X\r \Y\r C\r\Sigma\X$ fitting into an appropriate notion of determinant functor for triangulated categories. 

Breuning defined recently a notion of determinant functor for triangulated categories. 
He showed that any triangulated category $\C{T}$ possesses a universal
determinant functor $\det\colon \C T\r V({}^{b}\C T)$ \cite{dftc}, but he did not find
a connection between $V({}^{b}\C T)$ and  Neeman's $K$-theories of triangulated categories \cite{kttcneeman}. Notwithstanding, he proved that if $\C{T}$ has a
bounded non-degenerate $t$-structure with heart $\C{A}$, e.g.~$\C T=D^{b}(\C A)$, 
determinant functors on $\C{T}$ essentially coincide  with those on the abelian
category $\C A$ in the sense of Deligne \cite[Theorem 5.2]{dftc}. In general, for an exact category $\C E$, not all determinant functors $\det\colon \C E\r\C P$ extend to a Breuning determinant functor $\det\colon D^{b}(\C E)\r\C P$.

In connection with this problem, Grothendieck in a 1973 letter to Kundsen \cite[Appendix B]{dfec} had suggested considering $D^{b}(\C E)$ as a triangulated category enhanced with a category of `true triangles', to develop a theory of determinant functors for such enhanced  triangulated categories, and to show that any determinant functor $\det\colon \C E\r\C P$ extends to $\det\colon D^{b}(\C E)\r\C P$ in an essentially unique way for $\C E$ additive or abelian. Of course the problem  makes sense more generally for $\C E$ exact. 

We regard the bounded derived category $D^{b}(S_{2}\C E)$ of short exact sequences in $\C E$ as the category of true triangles of $D^{b}(\C E)$. More generally, we work with the homotopy category $\ho(\C W)$ of a Waldhausen category $\C W$ and the homotopy category $\ho(S_{2}\C W)$ of cofiber sequences in $\C W$. We define  derived determinant functors on $\C W$ by using only $\ho(\C W)$ and $\ho(S_{2}\C W)$ and we construct a universal derived determinant functor with target $V^{\operatorname{der}}(\C W)$. The group of isomorphism classes of objects in $V^{\operatorname{der}}(\C W)$ and the automorphism group of the tensor unit are Garkusha's derived $K$-theory groups $DK_{0}(\C W)$ and $DK_{1}(\C W)$ \cite{sdckt2}, respectively. We then deduce from \cite{malt} that derived and non-derived determinant functors on a Walhausen category $\C W$ are essentially the same thing, provided $\C W$  has cylinders and a saturated class of weak equivalences. This answers Grothendieck's question in the positive.

Returning to ordinary triangulated categories $\C T$, we define new notions of determinant functors whose universal examples compute Neeman's $K$-theories $K_{*}({}^{d}\C T)$ and $K_{*}({}^{v}\C T)$ in degrees $0$ and $1$ \cite{kttcneeman}. These are functorial $K$-theories, therefore they cannot simultaneously satisfy some desirable properties such as additivity, localization and agreement with Quillen's $K_{1}$ of exact categories \cite{kttc}. Little is known about these $K$-theories. Neeman asked, for $\C T$ a triangulated category with a bounded non-degenerate $t$-structure with heart $\C A$, whether $K_{1}(\C A)=K_{1}({}^{d}\C T)=K_{1}({}^{v}\C T)$ \cite[Problem 56]{kttcneeman}. He did this ``in order to show how embarrassingly little we know''  (sic) about the $K$-theory of triangulated categories. We here answer affirmatively this question. We also show with an example that $K_{1}({}^{d}\C T)\neq K_{1}({}^{v}\C T)$ in general. In addition, we prove that  $K_{*}({}^{d}\C T)$ and $K_{*}({}^{v}\C T)$ 
satisfy additivity in degrees $0$ and $1$. 

We should mention that Neeman has yet another $K$-theory $K_*(^{w}\C{T})$ which is not functorial, and moreover it is only defined for triangulated categories $\C T$ with a special kind of algebraic model. This condition is not satisfied by most triangulated categories arising in topology, e.g. the stable homotopy category. Neeman managed to show in a series of  papers that if $\C{T}$ has the required models and also a $t$-structure as above then $K_*(\C{A})=K_*(^{w}\C{T})$ in all dimensions, see \cite{kttcneeman} and the references therein. Our similarly flavoured theorems in dimension $1$ are the first results of this kind for functorial $K$-theories of triangulated categories beyond dimension $0$. Our techniques are completely different to Neeman's.

Another way of enhancing a triangulated category is by considering higher triangles, as suggested by Be{\u\i}linson--Bernstein--Deligne \cite{fp}. A 2-triangle would be an octahedron, 
there should be a class of \exact\ octahedra satisfying some axioms, generalizing the axioms for \exact\ triangles in triangulated categories, and so on.  Maltsiniotis \cite{cts} worked out an explicit definition and gave these categories the names of strongly triangulated categories or $\infty$-triangulated categories $\C T_{\infty}$. He defined a $K$-theory for them, that we denote $K_{*}(^{s}\C T_{\infty})$, and made some conjectures. We consider determinant functors for strongly triangulated categories and show the existence of a universal determinant functor  $\det\colon \C T_{\infty}\r V({}^{s}\C T_{\infty})$ such that $V({}^{s}\C T_{\infty})$ computes $K_{0}(^{s}\C T_{\infty})$ and $K_{1}(^{s}\C T_{\infty})$, as in previous cases. We use this to give an example of an exact category $\C E$ for which $K_{1}(\C E)\neq K_{1}(^{s}D^{b}(\C E))$. This disproves two conjectures due to Maltsiniotis \cite{cts,ktdt}. 

Another application included in this paper is to provide generators and a (possibly incomplete) set of relations for  $K_{1}({}^{d}\C T)$,  $K_{1}({}^{s}\C T_{\infty})$ and the automorphism group of the tensor unit in $V({}^{b}\C T)$. This extends results of \cite{K1gr,k1tc,k1wc}.

Our methods are fairly general, and they are suitable for application in other contexts not included in this paper. We consider determinant functors on certain kind of simplicial categories $\C C_{\bullet}$ with extra structure, that we call \emph{$S_{\bullet}$-categories}. Here $S_{\bullet}$ stands for Waldhausen's construction $S_{\bullet}\C W$  \cite{akttsI}, that he used to define $K_{*}(\C W)$. 

Most $K$-theories in the literature can be defined from a certain $S_{\bullet}$-category. We  construct a  universal determinant functor $\det\colon \C C_{\bullet}\r V(\C C_{\bullet})$ and show that $V(\C C_{\bullet})$ computes $\pi_{0}$ and $\pi_{1}$ of a connective spectrum $K(\C C_{\bullet})$ defined from $\C C_{\bullet}$ by using Segal's delooping machine \cite{cct}, in the same way as Waldhausen defined a spectrum $K(\C W)$ out of  $S_{\bullet}\C W$ whose homotopy groups are the $K$-theory groups $K_{*}(\C W)$.

We can apply the general theory to Garkusha's $S_{\bullet}\der$ of a right pointed derivator $\der$ \cite{sdckt1}. This yields a definition of determinant functor for derivators $\det\colon \der\r\C P$ that we have not worked out explicitly. Nevertheless, our results show the existence of a universal determinant functor  $\det\colon \der\r V(\der)$ such that $V(\der)$ is a model of the $1$-type of Garkusha--Maltsiniotis's $K$-theory spectrum $K(\der)$ \cite{sdckt1,ktdt}.

Figure \ref{figurin} illustrates different types of categories, interpolating between exact and triangulated categories, to which the theory of this paper applies.

\begin{figure}
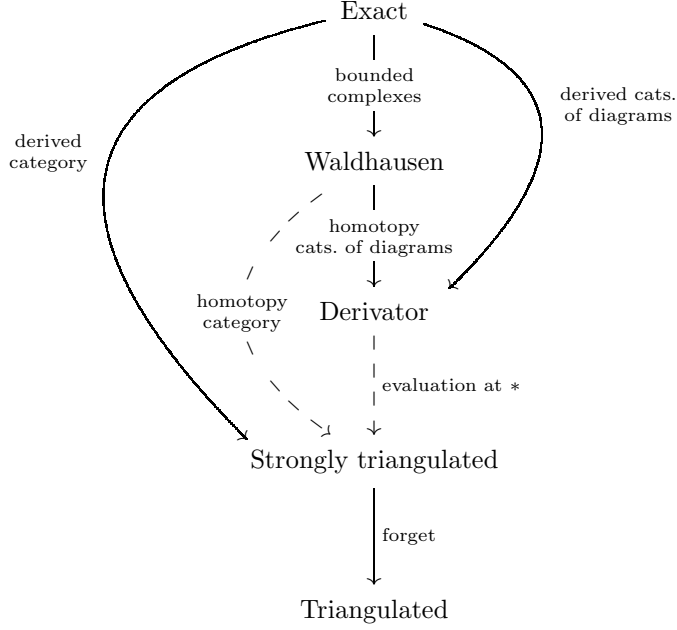

$$\xy
(0,0)*++{\text{Exact}}="e",
(0,-20)*++{\text{Waldhausen}}="w",
(0,-40)*++{\text{Derivator}}="d",
(0,-60)*++{\text{Strongly triangulated}}="s",
(0,-80)*++{\text{Triangulated}}="t",
\ar "e";"w"|{\begin{array}{c}\scriptstyle\text{bounded}\vspace{-4pt}\\\scriptstyle\text{complexes}\end{array}}
\ar
"w";"d"|{\begin{array}{c}\scriptstyle\text{homotopy}\vspace{-4pt}\\\scriptstyle\text{cats.\
        of diagrams}\end{array}} 
\ar @{-->} "d";"s"^-{\text{evaluation at }{*}}
\ar "s";"t"^-{\text{forget}}
\ar @/^50pt/ "e";(10,-37)^{\begin{array}{c}\scriptstyle\text{derived cats.}\vspace{-4pt}\\\scriptstyle\text{of diagrams}\end{array}}
\ar @{-->} @/_50pt/ "w";"s"|{\begin{array}{c}\scriptstyle\text{homotopy}\vspace{-4pt}\\\scriptstyle\text{category}\end{array}} 
\ar @/_80pt/ "e";(-17,-57)_{\begin{array}{c}\scriptstyle\text{derived}\vspace{-4pt}\\\scriptstyle\text{category}\end{array}}
\endxy$$
\caption{The hierarchy between exact and triangulated categories. The
  dashed arrows indicate that well-known stability properties are
  required.}\label{figurin}
\end{figure}


\medskip

Picard groupoids are algebraic models for spectra with homotopy groups concentrated in dimensions $0$ and $1$, and we prove that $V(\C C_{\bullet})$ is a specific algebraic model of the $1$-type of  $K(\C C_{\bullet})$. This strengthens our previous comments on how we obtain low-dimensional $K$-theory groups out of categories of virtual objects.
Our explicit  models for categories of virtual objects are as strict and small as they can be. They arise from \emph{stable quadratic modules}, an uncomplicated algebraic structure defined by Baues \cite{ch4c} to model stable homotopy types with homotopy groups concentrated in two consecutive degrees. Stable quadratic modules form a $2$-category that we prove to be $2$-equivalent, in a weak sense, to the $2$-category of Picard groupoids. This is a crucial step in the proof of our main results.

Note that determinant functors on  Waldhausen categories  have already been successfully applied in non-commutative Iwasawa theory  \cite{malte, ncimcvff}, and in $\mathbb{A}^{1}$-homotopy theory \cite{adrrii}. They have  also been discussed in the University of Chicago's Geometric Langlands Seminar  \cite{pcs}, see Remark \ref{chicago}. Fukaya and Kato give in \cite{fukayakato} an alternative construction of the category of virtual objects for $\C{E}$ the exact category of projective modules of finite type over a ring $R$.

\section{Determinant functors}\label{sec1}

In this section we state our main results: the definition of determinant functors for many kinds of categories, together with the construction of universal determinant functors whose codomains calculate the $1$-type of the respective $K$-theory spectra (see Theorems \ref{mainconcreto} and \ref{main2}).  Furthermore, the $1$-types of several known comparison maps between these spectra are calculated by straightforward algebraic functors between these codomains. The proofs will be given in Section \ref{universaldet}, where we develop a unified approach to these different notions.

\subsection{\Picardcategories\ and categorical groups}

Recall that a \emph{\Picardcategory} $\C P$ is a symmetric monoidal category \cite[VII.1, 7]{cwm} such that all morphisms are invertible and tensoring with any object $x$ in $\C{P}$ yields an equivalence of categories
$$x\otimes\sp\,\colon\C{P}\st{\sim}\To \C{P}.$$
Some examples are:

\begin{itemize}
\item The groupoid $\mathbf{Pic}(X)$ of line bundles over a scheme or
  manifold $X$ with the tensor product over the structure sheaf
  $\otimes_{\C{O}_X}$. If  $X=\operatorname{Spec} R$ for a commutative ring $R$ then $\mathbf{Pic}(R)=\mathbf{Pic}(X)$ 
is the groupoid of invertible $R$-modules.

\item The category $\mathbf{Pic}^\ZZ(X)$ of graded line bundles over a
  scheme or manifold $X$. Objects are pairs $(n,L)$ with $L\r X$ a line
  bundle  and $n\colon X\r \ZZ$ a locally constant map, called \emph{degree}. Morphisms are only allowed between objects with the same degree. They are simply line bundle isomorphisms. The symmetric monoidal
  structure is $(n,L)\otimes (n',L')= (n+n',L\otimes_{\C{O}_X}L')$ with
  the usual associativity and unit constraints. The graded symmetry
  constraint is the usual one twisted by a sign depending on the degrees, 
  $$(n,L)\otimes (n',L')\To(n',L')\otimes (n,L)\colon a\otimes
  b\mapsto(-1)^{n\cdot n'}b\otimes a.$$ 
\end{itemize}

\Picardcategories\ are also called \emph{symmetric categorical groups}. A \emph{categorical group} is a monoidal groupoid such that $x\otimes\sp$ is an equivalence of categories for any object $x$. 

Associativity and commutativity (symmetry) isomorphisms in these  monoidal groupoids will simply be denoted by
$$\ass\colon x\otimes(y\otimes z)\To  (x\otimes y)\otimes z,\qquad
\com\colon x\otimes y\To  y\otimes x.$$
Recall also that a tensor functor $F$ is equipped with natural multiplication isomorphisms
$$\hpp\colon F(x)\otimes F(y)\To  F(x\otimes y).$$
See Section \ref{victim} for further details.


\renewcommand{\X}{X}
\renewcommand{\Y}{Y}
\renewcommand{\Z}{Z}

\subsection{For Waldhausen categories}\label{forwald}

A \emph{Waldhausen category} $\C{W}$ is a category together with a distinguished zero object $0$ and two subcategories $\cof(\C W)$ and $\we(\C W)$ containing $\iso{\C{W}}$, whose morphisms are called \emph{cofibrations} $\into$ and \emph{weak equivalences} $\st{\sim}\r$, respectively. Some axioms must be satisfied, in particular the pushout of any map and a cofibration $\Y\l \X\into \Z$ 
exists in $\C{W}$, and is denoted $\Y\cup_\X\Z$. These categories were introduced by Waldhausen \cite[Section 1.2]{akts}, under the name of \emph{categories with cofibrations and weak equivalences}, as a general setting where a reasonable $K$-theory can be defined extending Quillen's  \cite{hakt}.

\begin{exm}
The following are three simple examples of Waldhausen categories:
\begin{itemize}
\item 
An \emph{exact category} $\C{E}$ is a full additive subcategory of an
abelian category closed under extensions. A short exact sequence in
$\C{E}$ is a short exact sequence in the ambient abelian category
between objects in $\C{E}$. The first arrow of a short exact sequence
in $\C E$ is called an \emph{admissible monomorphism}. Admissible
monomorphisms are the cofibrations of a Waldhausen category structure
on $\C{E}$ where weak equivalences are isomorphisms $\we(\C E)=\iso{\C{E}}$. One must also choose a zero
object $0$ in $\C E$. Examples of exact categories are abelian
categories, the category $\mathbf{proj}(R)$ of finitely generated
projective modules over a ring~$R$, and the category $\mathbf{vect}(X)$ of vector bundles of finite rank over a scheme or a manifold $X$.

\item The category $C^b(\C{E})$ of bounded complexes in an exact category $\C E$. Cofibrations are levelwise split monomorphisms and weak equivalences are quasi-isomorphisms, i.e.\ chain morphisms inducing isomorphisms in homology computed in the ambient abelian category. The distinguished zero object is the complex with $0$ everywhere.

\item The category $C^b(\C{E})$ with the same weak equivalences and distinguished zero object as above, but levelwise admissible monomorphisms as cofibrations. This Waldhausen category has the same $K$-theory as the previous one. We will always assume that $C^b(\C{E})$  is endowed with this Waldhausen category structure so that the inclusion of complexes concentrated in degree~$0$, $\C{E}\subset C^b(\C{E})$, preserves all the structure.
\end{itemize}
\end{exm}

Coproducts $\X\sqcup \Y=\X\cup_0\Y$ exist in $\C{W}$. Also, 
for any cofibration $f\colon \X\into \Y$ we have a \emph{cofiber sequence}
\begin{equation}\label{cofibersequence}
\Delta\colon\;\X\st{f}\into \Y\onto C^f. 
\end{equation}
Here $C^f=0\cup_\X\Y$ is the cofiber, sometimes denoted by $Y/X$, and the second morphism is the canonical map to  the push-out. Cofiber sequences in exact categories are short exact sequences.

The following definition of determinant functor generalizes Deligne's definition for the special case of exact categories {\cite[4.2]{dc}}.  

\begin{defn}\label{detwald}
Let $\C W$ be a Waldhausen category and $\C P$ a \Picardgroupoid.
A \emph{determinant functor} $\det\colon \C W\r \C P$  consists of a functor from the subcategory of weak equivalences,
$$\det\colon\we(\C{W})\To \C{P},$$ together with
\emph{additivity data}: for any cofiber sequence $\Delta$ as above, 
a morphism
$$\det(\Delta)\colon \det(C^{f})\otimes\det(\X) \To \det(\Y),$$
natural with respect to weak equivalences of cofiber sequences, given by commutative diagrams in $\C W$,
\begin{equation}\label{wecofibersequence}
\begin{array}{c}
\xymatrix{\Delta\ar[d]^\sim_\Phi&
\X\ar@{>->}[r]^-f\ar[d]^\sim&\Y\ar@{->>}[r]\ar[d]^\sim&C^f\ar[d]^\sim\\
\Delta'&
\X'\ar@{>->}[r]^-{f'}&\Y'\ar@{->>}[r]&C^{f'}}
\end{array}
\end{equation}
The following two axioms must be satisfied.

\begin{enumerate}

\item \emph{Associativity:} given a \emph{staircase} commutative diagram
\begin{eqnarray}\label{stair}
\Theta\colon\;\begin{array}{c}
\xymatrix{&&C^g\\
&C^f\ar@{>->}[r]&C^{gf}
\ar@{->>}[u]\\
\X\ar@{>->}[r]^-f&\Y\ar@{>->}[r]^-g\ar@{->>}[u]&\Z\ar@{->>}[u]}
              \end{array}
\end{eqnarray}
containing four cofiber sequences in $\C W$,
\begin{align*}
\Delta_{f}\colon\quad&\X\st{f}\into \Y\onto C^{f},&\Delta_{g}\colon\quad&\Y\st{g}\into \Z\onto C^{g},\\
\Delta_{gf}\colon\quad&\X\st{gf}\into \Z\onto C^{gf},& \widetilde{\Delta}\colon\quad&C^{f}\into C^{gf}\onto C^{g},
\end{align*}
the following diagram in $\C P$ commutes:
$$\xymatrix@C=5pt{&\det(\Z)&\\
\det(C^{g})\otimes\det(\Y)\ar[ru]^{\det(\Delta_g)}&&\det(C^{gf})\otimes\det(\X)\ar[lu]_{\det(\Delta_{gf})}\\
\det(C^{g})\otimes(\det(C^{f})\otimes\det(\X))\ar[u]^{1\otimes\det(\Delta_f)}\ar[rr]_{\ass}&&(\det(C^{g})\otimes\det(C^{f}))\otimes\det(\X)\ar[u]_{\det(\widetilde{\Delta})\otimes1}.}$$

\item \emph{Commutativity:} given 
two objects $\X$ and $\Y$ in $\C W$,  there are
  two cofiber sequences associated to the inclusions
  and projections of the two factors of 
their coproduct, 
\begin{equation}\label{inprowald}
\Delta_1\colon\quad\X\into \X\sqcup \Y\onto \Y,\qquad \Delta_2\colon\quad\Y\into \X\sqcup \Y\onto \X,
\end{equation}
and the following triangle commutes:
$$\xymatrix@C=0pt{&\det(\X\sqcup \Y)\ar@{<-}[rd]^-{\det(\Delta_2)}\ar@{<-}[ld]_-{\det(\Delta_1)}&\\\det(\Y)\otimes\det(\X)\ar[rr]_-{\com}&&\det(\X)\otimes\det(\Y).}$$
\end{enumerate}

If $\C P$ is just a categorical group, we define \emph{non-commutative determinant functors} $\det\colon \C W\r \C P$ as above, but omitting the commutativity axiom.
\end{defn}

The notation $\det\colon \C W\r \C P$ may seem to be misleading at a first glance. It suggests that $\det$ is defined on all morphisms of $\C W$. As a functor, it is only defined on weak equivalences. Nevertheless, it also takes values on cofiber sequences in the form of additivity data. This justifies the usual notation.

\begin{exm}\label{exexact}
The prototypical example of determinant functor on an exact category is the following. 
 Suppose $X$ is a scheme or manifold.
Then the rank of a vector bundle $E$ over $X$ is a locally constant function 
$\rank E\colon X\r \ZZ$, and we can define a determinant functor $\det\colon\mathbf{vect}(X)\r\mathbf{Pic}^\ZZ(X)$ as follows:
\begin{align*}\det(E)&=(\rank E,\wedge^{\rank
    E}_{\C{O}_X}E).
    \end{align*}
As a particular case, 
 we get a determinant functor  $\det\colon\mathbf{proj}(R)\r\mathbf{Pic}^\ZZ(R)$.
 
Knudsen--Mumford  \cite{pmssc} showed that this example can be extended to bounded complexes  $\det\colon C^b(\mathbf{vect}(X))\r\mathbf{Pic}^\ZZ(X)$ in an essentially unique way. Knudsen \cite{dfec} generalized this result to arbitrary determinant functors on an exact category. These results are proved by a lengthy direct computation. 
We here derive them from the existence of universal determinant functors with values in a \Picardcategory\ computing the first two $K$-theory groups and from the Gillet--Waldhausen theorem.
\end{exm}

\begin{rem}\label{chicago}
In the seminar notes \cite{pcs} a tentative definition of determinant functor for Waldhausen categories is given. Drinfeld wonders whether this notion is such that a universal determinant functor exists and whether the target is associated to Waldhausen's $K$-theory \cite[Endnote 7)]{pcs}. The results on non-commutative determinant functors in \cite[Section 2.3]{malte} and in Section \ref{ncdf} of this paper show that the answer is yes provided we introduce a slight correction in \cite[(ii) in Section 2]{pcs}: 
we must require the induced map 
$A\cup_{A'} B'\r B$ 
to be a cofibration, compare \cite[Proposition 1.6]{k1wc}.
The same correction must be made in \cite[Definition 2.2.1 (c)]{adrrii}.
\end{rem}

\begin{defn}\label{uniwald}
A determinant functor  $\det\colon\C{W}\r V(\C{W})$ is \emph{universal} if any determinant  functor $\det'\colon\C{W}\r\C{P}$ factors through $\det$ in an essentially unique way. More precisely, there exists a diagram
\begin{equation}\label{factorwald}
\xymatrix{\C{W}\ar[r]^-{\det}\ar[rd]_-{\det'}^{}="b"&V(\C{W})\ar[d]^-{f}_<(0){\quad}="a"\\&\C{P}\ar@{=>}"a";"b"_\alpha}
\end{equation}
where $f$ is a symmetric tensor functor and $\alpha$ is a natural transformation compatible with the additivity data, i.e.~for any cofiber sequence $\Delta$ as above, the following square commutes:
$$\xymatrix@C=45pt{
f(\det(C^{f}))\otimes f(\det(\X))\ar[r]^-{\hpp}
\ar[d]_{\alpha(C^{f})\otimes \alpha(\X)}
&f(\det(C^{f})\otimes\det(\X)) \ar[r]^-{f(\det(\Delta))}&f(\det(\Y))\ar[d]^{\alpha(\Y)}\\
\det'(C^{f})\otimes\det'(\X) \ar[rr]^-{\det'(\Delta)}&&\det'(\Y)
}$$
Moreover, if
$$\xymatrix{\C{W}\ar[r]^-{\det}\ar[rd]_-{\det'}^{}="b"&V(\C{W})\ar[d]^-{f'}_<(0){\quad}="a"\\&\C{P}\ar@{=>}"a";"b"_{\alpha'}}$$
is another such factorization, then there exists a unique tensor natural transformation $\beta\colon f\rr f'$ such that \eqref{factorwald} coincides with the pasting of
$$\xymatrix{\C{W}\ar[r]^-{\det}\ar[rd]_-{\det'}^{}="b"&V(\C{W})\ar[d]|-{f'}_<(0){\quad}="a"^{\,}="c"
\ar@/^25pt/[d]^f_{\,}="d"\\&\C{P}\ar@{=>}"a";"b"_{\alpha'}\ar@{<=}"c";"d"_{\beta}}$$

We call $V(\C W)$ the \emph{category of virtual objects} of $\C W$, following Deligne's terminology for  exact categories. This \Picardgroupoid\ is well defined up to equivalence. We later show its existence, producing a very explicit model.

\emph{Universal non-commutative determinant functors} are defined in the obvious way, dropping the symmetry condition from the tensor functors $f$ and $f'$.
\end{defn}


Any Waldhausen category $\C W$ has an associated homotopy category
$\ho (\C W)$ obtained by formally inverting weak equivalences in $\C
W$. 
We can also consider the Waldhausen category $S_2\C W$ of cofiber sequences in~$\C W$ \cite{akts}. 

\begin{defn}\label{derdet}
Let $\C P$ be a \Picardcategory. 
A \emph{derived determinant functor} $\det\colon\C W\r\C P$  consists of a functor from the category of isomorphisms in the homotopy category,
$$\det\colon\iso{\ho(\C{W})}\To \C{P},$$ together with
\emph{additivity data}: for any cofiber sequence $\Delta\colon\;\X\into \Y\onto C^{f}$ in~$\C{W}$, a morphism in $\C P$
$$\det(\Delta)\colon \det(C^{f})\otimes\det(\X)\To \det(\Y),$$
natural in $\ho(S_2\C W)$.
Axioms (1) and (2) in Definition \ref{detwald} must be satisfied.

\emph{Non-commutative derived determinant functors} with target a categorical group $\C P$ are defined by removing the commutativity axiom (2). Moreover, \emph{universal (non-commutative) derived determinant functors} are defined as in Definition \ref{uniwald}.
\end{defn}

Derived determinant functors are related to Grothendieck's  question to Knudsen that we answer positively in Section \ref{dnd}.


\renewcommand{\X}{X}
\renewcommand{\Y}{Y}
\renewcommand{\Z}{Z}

\subsection{For triangulated categories}

A \emph{triangulated category} $\C T$ is an additive category together with an equivalence $\Sigma\colon\C{T}\st{\sim}\r\C{T}$ and a class of diagrams called \emph{\exact\ triangles}
$$\Delta\colon\; X\st{f}\To  Y\st{i^f}\To  \cone{f}\st{q^f}\To  \Sigma X,$$
also depicted as
\begin{equation}\label{triangle}
\Delta\colon\;\begin{array}{c}
\xymatrix@!=15pt{X\ar[rr]^-f&&Y\ar[ld]^-{i^f}\\&\cone{f}\ar[lu]|{\minusone}^-<(.15){q^f}&
}\end{array}
\end{equation}
where $\X\st{\minusone}\r \Y$ denotes a morphism $\X\r\Sigma \Y$. 
Any diagram like \eqref{triangle} where two consecutive morphisms compose to $0$ will be called a \emph{triangle}. We say that $f$ is the \emph{base} of the triangle. 

Distinguished triangles must satisfy a set of well-known axioms, see \cite{triang}. Verdier's octahedral axiom says that given composable morphisms 
$$X\st{f}\To  Y\st{g}\To  Z,$$ 
and three \exact\ triangles $\Delta_f$, $\Delta_g$ and $\Delta_{gf}$ with bases $f$, $g$ and $gf$, respectively, then there exists a diagram with the shape of an octahedron 
\begin{equation}\label{octa}
\Theta\colon\qquad\octahedron{objA=X,objB=Y,objC=Z,objD=\cone{f},objE=\cone{gf},objF=\cone{g},
            arAB=f,arBC=g,arAC=gf,arDE=\bar g,arEF=\bar f, arCE=i^{gf}, arEA=q^{gf}, arFB=q^g, arCF=i^g, arBD=i^f, arDA=q^f, arFD=(\Sigma i^f)q^g}
\end{equation}
in which three faces are $\Delta_f$, $\Delta_g$ and $\Delta_{gf}$,
four faces are commutative triangles, and the remaining face
\begin{equation*}
\widetilde{\Delta}\colon\qquad\begin{array}{c}
\xymatrix@!=5pt{\cone{f}\ar[rr]^-{\bar g}
&&\cone{gf}\ar[ld]^-{\bar f}
\\&\cone{g}\ar[lu]|{\minusone}}
\end{array}
\end{equation*}
is also a \exact\ triangle. Moreover, three  planes divide the octahedron into two square pyramids. The squares 
perpendicular to the page
must be commutative.

Verdier's axiom is about the existence of appropriate $\bar{f}$ and $\bar{g}$; the rest is given.
Any diagram with the properties of \eqref{octa} will be called an \emph{octahedron}.



A \emph{special octahedron} is an octahedron \eqref{octa} such that the two commutative squares are  homotopy cartesian in the sense of \cite[Definition 1.4.1]{triang}, i.e.~the following triangles are \exact:
\begin{equation}\label{push}
\begin{array}{c}
Y\st{\binom{g}{-i^f}}\To  Z\oplus \cone{f}\st{(i^{gf},\bar{g})}\To  \cone{gf}\st{q^g\bar{f}}\To \Sigma Y,\\
\cone{gf}\st{\binom{q^{gf}}{-\bar{f}}}\To  \Sigma X\oplus \cone{g}\st{(\Sigma f,q^{g})}\To  \Sigma Y\st{\Sigma (\bar{g}q^f)}\To \Sigma \cone{gf}.
\end{array}
\end{equation}

\begin{rem}\label{speziale}
Special octahedra where first introduced by Be{\u\i}linson--Bernstein--Deligne \cite[Remarque 1.1.13]{fp}. If $\C{T}$ is a derived category, or more generally a stable homotopy category, then it is well known that the standard octahedral completion of two composable morphisms $X\r Y\r Z$ is special in this sense. In general, the octahedral axiom completion can be chosen so that one of the two triangles in \eqref{push} is \exact, compare \cite[Proposition 1.4.6]{triang}. In particular, if the completion happens to be unique then the resulting octahedron is special. This observation will be useful in applications concerning $t$-structures.
\end{rem}

\begin{defn}\label{dettriang}
Let $\C P$ be a \Picardcategory.
A \emph{Breuning determinant functor} $\det\colon\C T\r \C P$ consists of a functor $$\det\colon\iso{\C{T}}\To \C{P},$$ together with
\emph{additivity data}: for any \exact\ triangle $\Delta$ as in \eqref{triangle}, a morphism in $\C P$
$$\det(\Delta)\colon \det(\cone{f})\otimes\det(X)\To \det(Y),$$
natural with respect to \exact\ triangle isomorphisms
\begin{equation}\label{isotriangle}
\begin{array}{c}
\xymatrix{\Delta\ar[d]^\cong_\Phi&
\X\ar[r]^-f\ar[d]^\cong_{h}&\Y\ar[r]^-{i^f}\ar[d]^\cong&C^f\ar[d]^\cong \ar[r]^-{q^f}&\Sigma X\ar[d]^\cong_{\Sigma h}\\
\Delta'&
\X'\ar[r]^-{f'}&\Y'\ar[r]^-{i^{f'}}&C^{f'}\ar[r]^-{q^{f'}}&\Sigma X'}
\end{array}
\end{equation}
The following two axioms must be satisfied, see \cite[Definition 3.1]{dftc}.

\begin{enumerate}

\item \emph{Associativity}: for any octahedron $\Theta$ as in \eqref{octa}  the following diagram in $\C P$ commutes:
$$\xymatrix@C=5pt{&\det(Z)&\\
\det(\cone{g})\otimes\det(Y)\ar[ru]^{\det(\Delta_g)}&&\det(\cone{gf})\otimes\det(X)\ar[lu]_{\det(\Delta_{gf})}\\
\det(\cone{g})\otimes(\det(\cone{f})\otimes\det(X))\ar[u]^{1\otimes\det(\Delta_f)}\ar[rr]_{\ass}&&(\det(\cone{g})\otimes\det(\cone{f}))\otimes\det(X)\ar[u]_{\det(\widetilde{\Delta})\otimes1}}$$

\item \emph{Commutativity}: given 
two objects $X$, $Y$ in $\C T$, if we consider the two \exact\ triangles  associated to the inclusions and projections of the two factors of their coproduct
\begin{equation}\label{inprotriang}
\Delta_1\colon\quad X\r X\oplus Y\r Y\st{0}\r\Sigma X,\qquad \Delta_2\colon\quad Y\r X\oplus Y\r X\st{0}\r\Sigma Y,
\end{equation}
then the following diagram commutes:
$$\xymatrix@C=0pt{&\det(X\oplus Y)\ar@{<-}[rd]^-{\det(\Delta_2)}\ar@{<-}[ld]_-{\det(\Delta_1)}&\\\det(Y)\otimes\det(X)\ar[rr]_-{\text{comm.}}&&\det(X)\otimes\det(Y)}$$
\end{enumerate}

A \emph{special determinant functor} is defined in the same way, but we only require associativity with respect to special octahedra.

We can define \emph{non-commutative Breuning or special determinant functors}  allowing $\C P$ to be any categorical group and dropping the commutativity axiom. We can also define \emph{universal (non-commutative) Breuning or special determinant functors} as in Definition \ref{uniwald}. The only difference is that $\Delta$ must be a \exact\ triangle instead of a cofiber sequence.

\end{defn}




We now recall Vaknin's notion of virtual triangle \cite{vt}. 
A \emph{contractible triangle} is a direct sum of triangles of the form
\begin{align*}
A\st{1}\r A\r 0\r\Sigma A,&& 0\r B\st{1}\r B\r 0,&& C\r 0\r \Sigma C\st{1}\r\Sigma C,
\end{align*}
i.e.
$$\xymatrix{A\oplus C\ar[r]^{\binom{0\;0}{1\;0}}& B\oplus A\ar[r]^{\binom{0\;0}{1\;0}}&\Sigma C\oplus B \ar[r]^{\binom{0\;0}{1\;0}}&\Sigma A\oplus\Sigma C.}$$
Contractible triangles are always \exact. 

The definition of virtual triangle is a little bit involved. As a special case we have the triangles 
$$X'\st{f'}\To  Y'\st{i'} \To  \cone{f'}\st{q'} \To \Sigma X'$$
such that there exist \exact\ triangles as follows:
\begin{align*}
X'\st{f''}\r Y'\st{i'}\r \cone{f'}\st{q'}\r\Sigma X',&& X'\st{f'}\r Y'\st{i''}\r \cone{f'}\st{q'}\r\Sigma X',&& X'\st{f'}\r Y'\st{i'}\r \cone{f'}\st{q''}\r\Sigma X',
\end{align*}
i.e.~we can replace each arrow in $X'\st{f'}\to Y'\st{i'} \to \cone{f'}\st{q'} \to\Sigma X'$ so as to obtain a \exact\ triangle. 

In general, a triangle $X\st{f}\to Y\st{i} \to \cone{f}\st{q} \to\Sigma X$ is a \emph{virtual triangle} 
if the direct sum with some contractible triangle gives the special case of virtual triangle  above, 
$$\xymatrix@C=35pt{X'\ar[r]^-{f'}\ar[d]^{\cong}_{h}&Y'\ar[r]^-{i'}\ar[d]^{\cong}&\cone{f'}\ar[r]^-{q'}\ar[d]^{\cong}&\Sigma X'\ar[d]^{\cong}_{\Sigma h}\\
X\oplus A\oplus C\ar[r]^-{f\oplus\binom{0\;0}{1\;0}}& Y\oplus B\oplus A\ar[r]^-{i\oplus\binom{0\;0}{1\;0}}&\cone{f}\oplus \Sigma C\oplus B\ar[r]^-{q\oplus\binom{0\;0}{1\;0}}&\Sigma X\oplus \Sigma A\oplus\Sigma C.}$$

A \emph{virtual octahedron} is a diagram $\Theta$ as in \eqref{octa} where four faces $\Delta_f$, $\Delta_g$, $\Delta_{gf}$, $\widetilde{\Delta}$, are virtual triangles, the remaining four 
faces are commutative triangles, and we have two commutative squares as in classical octahedra.

\begin{rem}
In a virtual octahedron, the triangles  \eqref{push} are always virtual by Vaknin's two-out-of three property \cite[Section 1.3]{vt} applied to
$$\xymatrix{Z\ar[r]^{i^{gf}}&\cone{gf}\\
Y\ar[r]^{i^{f}}\ar[u]^{g}&\cone{f}\ar[u]_{\bar{g}}\\
X\ar[r]\ar[u]^{f}&0\ar[u]}
\qquad\qquad\qquad\qquad
\xymatrix{\cone{g}\ar[r]^{q^{g}}&\Sigma Y\\
\cone{gf}\ar[r]^{q^{gf}}\ar[u]^{\bar f}&\Sigma X\ar[u]_{\Sigma f}\\
Z\ar[r]^{}\ar[u]^{i^{gf}}&0\ar[u]}$$
\end{rem}

\begin{defn}\label{virdettriang}
A \emph{virtual determinant functor} $\det\colon\C T\r \C P$ is a functor $$\det\colon\iso{\C{T}}\To \C{P}$$ together with
\emph{additivity data}: for any virtual triangle $\Delta$ as in \eqref{triangle}, a morphism 
$$\det(\Delta)\colon \det(\cone{f})\otimes\det(X)\To \det(Y),$$
natural with respect to virtual triangle isomorphisms.
In addition we require associativity for virtual octahedra and commutativity as in Definition \ref{dettriang}.

We can also define \emph{(universal, non-commutative) virtual determinant functors}, compare Definition \ref{dettriang}.
\end{defn}


Following a remark of Be{\u\i}linson--Bernstein--Deligne \cite[1.1.14]{fp}, Maltsiniotis defined  the notion of \emph{strongly triangulated category} $\C T_{\infty}$, also termed \emph{$\infty$-triangulated category} \cite{cts}. He indicated how the bounded derived category $D^b(\C{E})$ can be endowed with such a structure. He also defined truncated versions, called \emph{$n$-pretriangulated category}. 
A \emph{$3$-pre\-triangulated category} $\C{T}_3$ is a triangulated category  together with a family of \emph{\exact\ octahedra} (\emph{$3$-triangles} in Maltsiniotis's terminology), which must satisfy some axioms
generalizing the axioms for \exact\ triangles in a triangulated
category, see \cite[1.3 and 1.4]{cts}.

\begin{defn}\label{3eme}
Let 
$\C P$ be a \Picardcategory. 
A \emph{determinant functor} $\det\colon\C{T}_3\r\C{P}$ is the same as a determinant functor on the
underlying triangulated category, except that we only 
require the associativity axiom (1) to hold for  \exact\ octahedra. A determinant functor on a strongly triangulated category is a determinant functor on the underlying $3$-pre\-triangulated category. We similarly define \emph{(universal, non-commutative) determinant functors} in this context.
\end{defn}


\subsection{Stable quadratic modules}
In this section we introduce algebraic tools which allow a very explicit construction   of universal determinant functors. The main tools are Baues's stable quadratic modules \cite{ch4c}.

\begin{defn}\label{baues}
A \emph{stable quadratic module} $C_*$ consists of 
group homomorphisms
$$C_0^\abb\otimes C_0^\abb\st{\grupo{\cdot,\cdot}}\To  C_1\st{\partial}\To  C_0,$$
satisfying the following equations, $c_{i},c_{i}'\in C_{i}$:
\begin{enumerate}
\item $\grupo{c_0,c_0'}+\grupo{c_0',c_0}=0$,
\item $\partial\grupo{c_0,c_0'}=[c_{0}',c_{0}]$,
\item $\grupo{\partial(c_1),\partial(c_1')}=[c_{1}',c_{1}]$.
\end{enumerate}
Here $[x,y]=-x-y+x+y$ denotes the \emph{commutator} of two elements $x, y$ in a group. We will denote group laws additively, although the groups may be non-abelian. 
We use the same notation for elements of $C_0$ as for
their images in $C_0^\abb$ since in context there is no ambiguity. 

A \emph{morphism} of stable quadratic modules $f\colon C_{*}\r D_{*}$ consists of group homomorphisms $f_{i}\colon C_{i}\r D_{i}$, $i=0,1$,  satisfying $\partial f_{1}=f_{0}\partial$ and $\grupo{f_{0},f_{0}}=f_{1}\grupo{\cdot,\cdot}$. 

A \emph{homotopy} $\alpha\colon f\rr g$ between two morphisms
$f,g\colon C_*\to D_*$ 
is a function $\alpha\colon C_0\r D_1$ such that
\begin{align*}
\alpha(c_0+c_0')&=\alpha(c_0)^{g_0(c_0')}+\alpha(c_0'),\\
\partial\alpha(c_0)&=-g_0(c_0)+f_0(c_0),\\
\alpha\partial(c_1)&=-g_1(c_1)+f_1(c_1).
\end{align*}
Here we follow the conventions in \cite{malte}, which are opposite to \cite{ch4c,1tk,k1wc}.

Stable quadratic modules, morphisms and homotopies form a $2$-category. Horizontal composition is given by composition of maps, and the vertical composition of two homotopies $$f\st{\alpha}\Longrightarrow g\st{\beta}\Longrightarrow h$$ is given by the map $\beta+\alpha$, compare \cite[Proposition 7.2]{2hg1}.
\end{defn}

Notice that, if we think of a stable quadratic module $C_{*}$ as a non-abelian chain complex concentrated in degrees $0$ and $1$,
$$\cdots\r 0\r C_{1}\st{\partial}\To C_{0}\r0\r\cdots,$$
enriched with the bracket operation,
the  homotopies above are analogs of classical chain homotopies.

\begin{rem}
The bracket $\grupo{\cdot,\cdot}$ behaves as a bilinear form, since its source is the tensor square of the abelianization of $C_{0}$. It follows that the groups $C_{0}$ and $C_{1}$ have nilpotency class~$2$. Groups of nilpotency class $2$ are very close to  abelian groups. Commutators need not vanish, but they are central, and the commutator bracket $[\cdot,\cdot]$ behaves as a bilinear form: it factors through the tensor square of the abelianization.  The bracket $\grupo{\cdot,\cdot}$  also maps to the center of $C_{1}$. Moreover, $\partial(C_{1})\subset C_{0}$  is normal.

The group $C_{0}$ acts on the right of $C_{1}$ by the formula
$$c_1^{c_0}=c_1+\grupo{c_0,\partial(c_1)}.$$
The exponent of $c_{1}^{c_{0}}$ is actually in the abelianization $C_0^\abb$. Moreover, we have
\begin{align*}
c_{1}+c_{1}'&=c_{1}'+c_{1}^{\partial(c_{1}')}=(c_{1}')^{-\partial(c_{1})}+c_{1}.
\end{align*}

Homotopies satisfy $\alpha(0)=\alpha(0+0)=\alpha(0)^{g_{0}(0)}+\alpha(0)$,
hence $$\alpha(0)=0.$$ Moreover, $0=\alpha(0)=\alpha(c_0-c_0)=\alpha(c_0)^{-g_0(c_0)}+\alpha(-c_0)$,
therefore $$\alpha(-c_0)=-\alpha(c_0)^{-g_0(c_0)}.$$
\end{rem}

\begin{exm}\label{sqmexamples}
We here give some easy examples of stable quadratic modules.
\begin{enumerate}
\item A stable quadratic module with trivial bracket $\langle\cdot,\cdot\rangle$ is the same as an abelian group homomorphism $f$:
$$B\otimes B\st{0}\To A\st{f}\To B.$$

\item If $G$ is a group of nilpotency class $2$ and $H\subset G$ is a subgroup containing commutators, $[G,G]\subset H$, then
$$
G^{\abb}\otimes G^{\abb}\st{\langle\cdot,\cdot\rangle}\To H\hookrightarrow G,\qquad
\grupo{x,y}=[y,x],
$$
is a stable quadratic module.

\item The following stable quadratic module consists of abelian groups but has a non-trivial bracket
$$\mathbb{Z}\otimes \mathbb{Z}\st{\grupo{\cdot,\cdot}}\To \mathbb{Z}/2\st{0}\To \mathbb{Z},
\qquad \grupo{1,1}=1.$$

\item More generally, if $R$ is a commutative ring and $R^{\times}$ is the (multiplicative) group of units we can consider the stable quadratic module
$$\mathbb{Z}\otimes \mathbb{Z}\st{\grupo{\cdot,\cdot}}\To R^{\times}\st{0}\To \mathbb{Z},
\qquad \grupo{1,1}=-1.$$
\end{enumerate}
\end{exm}

Stable quadratic modules are closely related to \Picardcategories.

\begin{defn}\label{gammatron}
The \Picardcategory\ $\Gamma C_*$ associated to a stable quadratic module $C_*$ is defined as follows.
The set of objects is 
$C_0$. The set of all morphisms is the semidirect product $C_0\ltimes C_1$, i.e.~the cartesian product with the following group structure:
\begin{align*}
(c_0,c_1)+(c'_0,c'_1)&=(c_0+c'_0,c_1^{c'_0}+c'_1).
\end{align*}
The source and target of $(c_0,c_1)\in C_0\ltimes C_1$ are
$$(c_0,c_1)\colon c_0+\partial (c_1)\To  c_0.$$
Composition of morphisms is defined as
$$(c_0,c_1)\circ(c_0+\partial (c_1),c_1')=(c_0,c_1+c_1').$$
The tensor product is simply the sum $+$ on both objects and morphisms, and the tensor unit is $I=0$. It is strictly associative and unital. Symmetry constraints are defined by the bracket
$$(c_0'+c_0,\grupo{c_0',c_0}):c_0+c_0'\To  c_0'+c_0.$$
Notice that identity morphisms are given by
$$1_{c_{0}}=(c_{0},0).$$

A morphism $f\colon C_{*}\r D_{*}$ induces a strict tensor functor $\Gamma f\colon\Gamma  C_{*}\r\Gamma  D_{*}$ given on objects by $f_{0}$ and on morphisms by $f_{0}\ltimes f_{1}$. A homotopy $\alpha\colon f\rr g$ between two morphisms
$f,g\colon C_*\to D_*$ induces a tensor natural transformation $\Gamma\alpha\colon\Gamma f\rr\Gamma g$ defined by $(\Gamma \alpha)(c_{0})=(g_{0}(c_{0}),\alpha(c_{0}))$. In this way, $\Gamma$ defines a $2$-functor from the $2$-category of  stable quadratic modules to the $2$-category of \Picardgroupoids.
\end{defn}

\begin{defn}\label{homotopygroups}
The
{\em homotopy groups} of a stable quadratic module $C_*$ are
$$
\pi_0(C_*)=C_0/\partial (C_1),\qquad\qquad \pi_1(C_*)=\ker\partial.
$$
The $k$-{\em invariant} is the natural homomorphism
\begin{align*}
\eta\colon\pi_0(C_*)\otimes\mathbb Z/2&\To  \pi_1(C_*), \\
{ [c_0]\otimes 1}&\;\mapsto\; \grupo{c_0,c_0}.
\end{align*}

Homotopy groups are functors from the category $\mathbf{squad}$ of stable quadratic modules, and the $k$-invariant is a natural transformation. A \emph{weak equivalence} or \emph{quasi-isomorphism} is a morphism which induces isomorphisms on homotopy groups. The \emph{homotopy category} $\ho\mathbf{squad}$ is obtained from $\mathbf{squad}$ by formally inverting quasi-isomorphisms. It is known that two stable quadratic modules are weakly equivalent if and only if they have isomorphic $k$-invariant.

A stable quadratic module $C_*$ is \emph{$0$-free} if $C_0=\nig{E}$ 
is the nilpotent group of class two freely
generated by a set $E$, i.e.~the quotient of the free non-abelian group $\grupo{E}$ by triple commutators. We denote by $\mathbf{squad}_{0}$ the full subcategory spanned by $0$-free objetcs
\end{defn}

\begin{rem}
Notice that $\pi_{0}(C_{*})$ is the group of isomorphism classes of objects in $\Gamma C_{*}$, and $\pi_{1}(C_{*})$ is the automorphism group of the tensor unit in $\Gamma C_{*}$.  The $k$-invariant measures the deviation of $\Gamma C_*$ from being strictly commutative. The homotopy category  $\ho\mathbf{squad}$ is equivalent to the category $\mathbf{squad}_{0}/\!\simeq$ obtained by dividing out homotopies from  $\mathbf{squad}_{0}$.
\end{rem}

\begin{exm}\label{sqmexamples2}
The homotopy groups and the $k$-invariants of the stable quadratic modules in Example \ref{sqmexamples} are:
\begin{enumerate}
\item $\pi_{0}=\coker f$, $\pi_{1}=\ker f$, and the $k$-invariant vanishes.

\item $\pi_{0}=G/H$, $\pi_{1}=0$, and the $k$-invariant vanishes for obvious reasons.

\item $\pi_{0}=\mathbb Z$, $\pi_{1}=\mathbb Z/2$, and the $k$-invariant is the natural projection $\mathbb Z\onto \mathbb Z/2$.

\item $\pi_{0}=\mathbb Z$, $\pi_{1}=R^{\times}$, and the $k$-invariant is $\mathbb Z\r R^{\times}\colon 1\mapsto -1$.
\end{enumerate}

\end{exm}

The following lemma about homotopies is very useful to deform morphisms, see \cite[Lemmas 2.1.13 and 2.1.14]{malte}.

\begin{lem}\label{abouthomo}
Let $g\colon C_{*}\r D_{*}$ be a morphism of stable quadratic modules with $C_{0}=\nig{E}$. Any map $E\r D_{1}$ extends to a map
$$\alpha\colon C_{0}\To  D_{1}$$ 
satisfying
$$\alpha(c_{0}+c_{0}')=\alpha(c_{0})^{g_0(c_{0}')}+\alpha(c_{0}'),\qquad c_{0},c_{0}'\in C_{0}.$$
Moreover, there is a unique morphism $f=g+\alpha\colon C_{*}\r D_{*}$, defined as
\begin{align*}
f(c_{0})&=g(c_{0})+\partial\alpha(c_{0}),&
f(c_{1})&=g(c_{1})+\alpha\partial(c_{1}),&
c_{i}\in C_{i},&\; i=0,1,
\end{align*}
such that $\alpha$ is a homotopy $\alpha\colon f\Rightarrow g$.
\end{lem}

\begin{defn}
A \emph{strong deformation retraction} is a special kind of homotopy equivalence between stable quadratic modules, given by a diagram
$$\EZDIAG{C_{*}}{D_{*}}{p}{j}{\alpha},$$
where $p$ and $j$ are morphisms such that $pj=1_{D_{*}}$ and $\alpha\colon jp\Rightarrow 1_{C_{*}}$ is a homotopy satisfying $\alpha j=0$ and $p\alpha=0$.
\end{defn}

%

The following lemma will help us to define strong deformation retractions.

\begin{lem}\label{cacharro}
Consider a  $0$-free stable quadratic module  $C_{*}$ with  $C_{0}=\nig{E}$. Moreover, suppose $j\colon D_{*}\r C_{*}$ is a levelwise injective morphism of stable quadratic modules  such that $D_{0}=\nig{E'}$, $E'\subset E$,   and $j_{0}$ is induced by the inclusion. Let $\alpha\colon E\r C_{1}$ be a map. Assume that the morphism $f=1_{C_{*}}+\alpha\colon C_{*}\r C_{*}$ from Lemma \ref{abouthomo} factors as $f=jp$, and that the homotopy $\alpha$ satisfies 
\begin{align*}
\alpha(e')&=0,\quad e'\in E',&
\alpha(e+\partial\alpha(e))
&=0,\quad e\in E.
\end{align*}
Then $\alpha$, $p$ and $j$ define a strong deformation retraction.
\end{lem}

\begin{proof}
The only non-obvious formulas are  $p\alpha=0$ and $pj=1_{D_{*}}$. Since $j$ is levelwise injective, it is enough to check that $f\alpha=0$ and $fj=j$. The formulas defining $f=1_{C_{*}}+\alpha$ and the first equation in the statement show that the former implies the latter. The second equation proves $f\alpha=0$, since $\alpha(e+\partial\alpha(e))=\alpha(e)+\alpha\partial\alpha(e)=f\alpha(e)$.
\end{proof}

We now consider the left adjoint of the functor sending a stable quadratic module $C_{*}$ to the pair of sets $(C_{0},C_{1})$, see \cite[Appendix A]{1tk}.  Stable quadratic modules  in the image of this left adjoint are said to be \emph{free}.

\begin{defn}
The \emph{free stable quadratic module}
$F_{*}^{s}(E_{0},E_{1})$ on a pair of sets $(E_{0},E_{1})$ can be constructed as follows: 
$F_{0}^{s}(E_{0},E_{1})=\nig{E_{0}\sqcup E_{1}}$.
Moreover, if we denote $\abg{E}$ the free abelian group on a set $E$, then
$$F_{1}^{s}(E_{0},E_{1})=\abg{E_0}\otimes\mathbb{Z}/2\times\wedge^2\abg{E_0}
\times\abg{E_0}\otimes\abg{E_1}\times\nig{E_1}.
$$
The homomorphism $\partial$ and the bracket $\grupo{\cdot,\cdot}$ in $F_{*}^{s}(E_{0},E_{1})$  are defined by the following formulas:
\begin{align*}
\partial(e_{0}\otimes 1,e_0'\wedge e_0'',e_{0}'''\otimes e_{1},e_{1}')&=[e_0'',e_0']+[e_{1},e_{0}''']+e_{1}';&
\grupo{e_0,e_0}&=(e_0\otimes1,0,0,0);\\
\text{if }e_{0}\neq e_{0}'\text{ then}\quad \grupo{e_0,e_0'}&=(0,e_0\wedge e_0',0,0);&
\grupo{e_0, e_1}&=(0,0,e_0\otimes e_1,0);\\
\grupo{e_1, e_1'}&=(0,0,0,[e_1',e_1]).
\end{align*}

Given two sets of relations $R_{i}\subset F_{i}^{s}(E_{0},E_{1})$, $i=0,1$, the \emph{stable quadratic module $C_{*}$ with generators $(E_{0},E_{1})$ and relations $(R_{0},R_{1})$} is defined as follows: $C_{0}$ is the quotient of $F_{0}^{s}(E_{0},E_{1})$ by the normal subgroup $N_{0}$ generated by $R_{0}\cup \partial R_{1}$, and $C_{1}$ is the quotient of $F_{1}^{s}(E_{0},E_{1})$ by the normal subgroup generated by 
$R_{1}$ and $\grupo{F_{0}^{s}(E_{0},E_{1}),N_{0}}$. The homomorphism $\partial$ and the bracket $\grupo{\cdot,\cdot}$ on $F_{*}^{s}(E_{0},E_{1})$ induce a
structure of a stable quadratic module on $C_*$. This is the unique
structure for which the natural projection $F_{*}^{s}(E_{0},E_{1})\onto C_{*}$ is a morphism of stable quadratic modules.
\end{defn}

Stable quadratic modules defined by a presentation satisfy the obvious universal property.

\subsection{Universal determinant functors}\label{losejemplillos}

In this section we present universal determinant functors
for all cases considered above. 
We will actually construct them by using presentations of 
stable quadratic modules.

\begin{defn}\label{univtargetwald}
Let $\C{W}$ be a Waldhausen category. We define the stable quadratic module $\D{D}_*(\C{W})$ by generators
\begin{enumerate}
\renewcommand{\labelenumi}{(G{\arabic{enumi}})}
\item $[\X]$ for any object, in dimension $0$,
\item $[\X\st{\sim}\r \X']$ for any weak equivalence, in dimension $1$,
\item $[\Delta]$ for any cofiber sequence as in \eqref{cofibersequence}, in dimension $1$,
\end{enumerate}
and relations
\begin{enumerate}
\renewcommand{\labelenumi}{(R{\arabic{enumi}})}
\item $\partial[\X\st{\sim}\r \X']=-[\X']+[\X]$,
\item $\partial[\Delta]=-[\Y]+[C^f]+[\X]$,
\item $[0]=0$ for the zero object,
\item $[\X\st{1}\r \X]=0$ for any object,
\item $[\X\st{1}\into \X\onto 0]=0=[0\into\X\st{1}\onto \X]$ for any object,
\item for any pair of composable weak equivalences $\X\st{\sim}\r \Y \st{\sim}\r \Z$,
$$[\X\st{\sim}\r \Z]\;\;=\;\;[\Y\st{\sim}\r \Z]+[\X\st{\sim}\r \Y],$$
\item for any weak equivalence of cofiber sequences $\Phi\colon\Delta \st{\sim}\r \Delta'$ as in \eqref{wecofibersequence},
$$[\X\st{\sim}\r \X']+[C^{f}\st{\sim}\r C^{f'}]^{[X]}\;\;=\;\;-[\Delta']+[\Y\st{\sim}\r \Y']+[\Delta],$$
\item for any staircase diagram $\Theta$ as in \eqref{stair},
$$[\Delta_{g}]+[\Delta_{f}]\;\;=\;\;[\Delta_{gf}]+[\widetilde{\Delta}]^{[\X]},$$
\item for any 
two objects $\X$ and $\Y$, 
$$\grupo{[\X],[\Y]}\;\;=\;\;-[\Delta_{2}]+[\Delta_{1}],$$
where $\Delta_{1}$ and $\Delta_{2}$ are the cofiber sequences in \eqref{inprowald}.
\end{enumerate}
This stable quadratic module was first considered in \cite{1tk}.

The stable quadratic module $\DD{D}_*(\C{W})$ is defined by almost the same presentation, modified in the following way. We have generators (G2) for all isomorphisms $\X{\cong} \X'$ in $\ho\C W$, (R1--5,8,9) remain the same, (R6) must hold for any pair of composable isomorphisms in $\ho\C W$, and (R7) is required for all cofiber sequence isomorphisms $\Phi\colon\Delta {\cong} \Delta'$ in $\ho(S_2\C W)$, see \cite[Definition 5.4]{malt}.

Given a triangulated category $\C T$ we define the stable quadratic module $\D{D}_*(^{b}\!\C{T})$ as follows. The definition is by generators and relations as above. Generators (G1--3) correspond to objects, isomorphisms $\X{\cong} \X'$, and \exact\ triangles $\Delta$ as in \eqref{triangle}, respectively. Relations (R1--4) are the same,  (R5) is
\begin{itemize}
\item[{(R5)}] $[\X\st{1}\r \X\r 0\r\Sigma X]=0=[0\r\X\st{1}\r \X\r 0]$ for any object,
\end{itemize}
relation (R6) is imposed for any pair of composable isomorphisms, (R7) is required for any exact triangle isomorphism $\Phi\colon\Delta {\cong} \Delta'$  as in \eqref{isotriangle}, (R8) must hold for any octahedron $\Theta$ as in \eqref{octa}, and the exact triangles $\Delta_{1}$ and $\Delta_{2}$ in (R9) are defined in \eqref{inprotriang}.

The stable quadratic module $\D{D}_*(^{d}\!\C{T})$ is presented as $\D{D}_*(^{b}\!\C{T})$, except that we only require (R8) for special octahedra. 

For the presentation of $\D{D}_*(^{v}\!\C{T})$, we allow all virtual triangles as generators (G3), (R7) must hold for any virtual triangle isomorphism, and (R8) for any virtual octahedron $\Theta$.

If $\C T_3$ is a $3$-pretriangulated category, the stable quadratic module $\D{D}_*(^{s}\!\C{T}_3)$ is again presented as $\D{D}_*(^{b}\!\C{T})$, 
but we only impose (R8) for \exact\ octahedra.
\end{defn}

\begin{rem}
All these stable quadratic modules  are $0$-free. The degree $0$ group is free of nilpotence class $2$ with basis given by the set of non-trivial objects.

The presentations above are not minimal. Relation (R3) follows from (R2) and (R5),  (R4) follows from (R6), and (R5) is equivalent to
\begin{itemize}
\item[$(\mathrm{R}5')$] $[0\into 0\onto 0]=0$ in $\D{D}_{*}(\C W)$.
\item[$(\mathrm{R}5')$] $[0\r 0\r 0\r\Sigma 0]=0$ in $\D{D}_{*}({}^{\filledsquare}\C T)$, $\filledsquare=b,d,s,v$.
\end{itemize}

In all cases, if (R9) holds for a given coproduct of two objects then it also holds for any other coproduct. This follows from the uniqueness of coproducts up to isomorphism and the rest of relations.

For $\D{D}_*(\C W)$ and $\DD{D}_*(\C W)$ it is enough to impose (R8) for only one staircase completion of each two composable cofibrations $X\into Y\into Z$, as any two completions are isomorphic. Similarly, for $\D{D}_*({}^{s}\C T_3)$ is is enough to impose (R8) for only one \exact\ octahedron completing each pair of composable morphisms $X\r Y\r Z$. This is not the case for ordinary triangulated categories, see \cite{kuenzer}.
\end{rem}

The following proposition provides smaller presentations in some cases. It follows from Proposition \ref{porquillo}.

\begin{prop}\label{menos}
Let $\C{W}$ be a Waldhausen category where weak equivalences are isomorphisms and $\C{T}$ a (strongly) triangulated category. Then $\D{D}_*(\C{W})$ and $\D{D}_*(^{\filledsquare}\!\C{T})$, $\filledsquare=b,d,s,v$, have a presentation with generators $(\mathrm{G}1)$ and $(\mathrm{G}3)$ and relations $(\mathrm{R}2)$, $(\mathrm{R}5')$, $(\mathrm{R}8)$ and $(\mathrm{R}9)$. 
\end{prop}

The following theorem is the main result of this paper.

\begin{thm}\label{mainconcreto}
Let $\C W$ be a Waldhausen category, $\C T$ a triangulated category, and $\C T_3$ a $3$-pretriangulated category. We have:
\begin{itemize}
 \item a determinant functor $\det\colon\C W\r \Gamma\D{D}_*(\C{W})$,
\item a derived determinant functor $\det\colon\C W\r \Gamma\DD{D}_*(\C{W})$,
\item a Breuning determinant functor $\det\colon\C T\r \Gamma\D{D}_*(^{b}\!\C{T})$,
\item a special determinant functor $\det\colon\C T\r \Gamma\D{D}_*(^{d}\!\C{T})$,
\item a virtual determinant functor $\det\colon\C T\r \Gamma\D{D}_*(^{v}\!\C{T})$,
\item a determinant functor $\det\colon\C T_3\r \Gamma\D{D}_*(^{s}\!\C{T}_3)$,
\end{itemize}
all of which are universal. They are defined 
(in the notation of Definition \ref{gammatron}) by:
\begin{itemize}
 \item $\det(\X)=[\X]$ for any object,
\item $\det(\X\r \X')=([\X'],[\X\r\X'])$ for any weak equivalence or isomorphism,
\item $\det(\Delta)=([\Y],[\Delta])$ for any cofiber sequence, \exact\ triangle or virtual triangle, as in \eqref{cofibersequence} or \eqref{triangle}.
\end{itemize}
Moreover, if we simply regard the targets as categorical groups, forgetting the symmetry, these determinant functors are also universal among non-commutative determinant functors.
\end{thm}

This result follows from Theorem \ref{main1} and Corollary \ref{tambien}.

If our Waldhausen or (strongly) triangulated category has functorial coproducts, we may use a quotient of the stable quadratic module in Definition \ref{univtargetwald} instead.

\begin{defn}\label{funco}
A category has \emph{functorial coproducts} if it has a monoidal structure $+$ which is strictly associative and unital, the unit $0$ is an initial object, and the following diagram is a coproduct for any two objects $X$ and $Y$,
$$\X = \X+0\longrightarrow \X+\Y\longleftarrow 0+\Y = \Y.$$

If $\C W$ and $\C T$ are  a Waldhausen category and a (strongly) triangulated category with functorial coproducts, respectively, we define the stable quadratic modules $\D{D}^{+}_{*}(\C W)$ and $\D{D}^{+}_{*}({}^{\filledsquare}\C T)$, $\filledsquare=b,d,v,s$, as the quotient of $\D{D}_{*}(\C W)$ and $\D{D}_{*}({}^{\filledsquare}\C T)$  by the following extra relation,
\begin{itemize}
 \item[(R10)] $[Y\into X+Y\onto X]=0$ for any pair of objects $\X$ and $\Y$ in $\C{W}$.
  \item[(R10)] $[Y\r X+Y\r X\st{0}\r \Sigma Y]=0$ for any pair of objects $\X$ and $\Y$ in $\C{T}$.
\end{itemize}
\end{defn}

\begin{rem}\label{dosmas}
In $\D{D}^{+}_{*}(\C W)$ and $\D{D}^{+}_{*}({}^{\filledsquare}\C T)$, for any two objects $X$ and $Y$, we have
\begin{align*}
[X+Y]&=[X]+[Y],\\
(\mathrm{R}9')\qquad \grupo{[X],[Y]}&=[Y+X\cong X+Y].
\end{align*}
Actually, (R7), (R9$'$) and (R10) imply (R9), compare \cite[Remark 4.1]{k1wc}. Notice also that (R5$'$) is a special case of (R10). 

In $\D{D}^{+}_{*}({}^{\filledsquare}\C T)$, given two \exact\ (or virtual if $\filledsquare=v$) triangles,
\begin{align*}
&\hspace{-30pt}[X+X'\st{f+f'}\To Y+Y'\st{i+i'}\To Z+Z'\st{q+q'}\To \Sigma X+\Sigma X']\\[-3pt]
={}&[X\st{f}\r Y\st{i}\r Z\st{q}\r \Sigma X]^{[Y']}
+[X'\st{f'}\r Y'\st{i'}\r Z'\st{q'}\r \Sigma X']
+\grupo{[X],[Z']}.
\end{align*}
See Corollary \ref{suma2} below. See also \cite[Lemma 4.8]{k1wc} for the corresponding result in $\D{D}^{+}_{*}(\C W)$.

\end{rem}

\begin{prop}\label{kkk}
Let $\C W$ and $\C T$ be a Waldhausen category and a (strongly) triangulated category with functorial coproducts, respectively. Assume the set of objects is free as a monoid under $+$ in both cases. Then the natural projection,
\begin{align*}
\D{D}_{*}(\C W)&\onto\D{D}^{+}_{*}(\C W);&
\D{D}_{*}({}^{\filledsquare}\C T)&\onto\D{D}^{+}_{*}({}^{\filledsquare}\C T),\quad\filledsquare=b,d,v,s;
\end{align*}
is a weak equivalence. It is actually part of a strong deformation retraction.
\end{prop}

This follows from Proposition \ref{sumnormalization} below.

\begin{rem}\label{masmenos}
Under the hypotheses of the previous theorem, $\D{D}^{+}_{*}(\C W)$ and $\D{D}^{+}_{*}({}^{\filledsquare}\C T)$ are also $0$-free. The degree $0$ group of nilpotency class $2$ is freely generated by any basis of the free monoid of objects.
\end{rem}

Let us consider a futher simplification for additive categories. It will lead to some explicit computations.

\begin{prop}\label{minimismo}
Let $\C A$ be an additive category, regarded as a split exact category. Suppose it satisfies the hypotheses of Theorem \ref{kkk}. 
Then $\D{D}^{+}_{*}(\C A)$ has a presentation with generators  $(\mathrm{G}1)$ and $(\mathrm{G}2)$ and relations $(\mathrm{R}1)$, $(\mathrm{R}6)$, $(\mathrm{R}9')$ and
\begin{itemize}
\item[$(\mathrm{R}7')$] $[f\colon X\cong X']^{[Y']}+[g\colon Y\cong Y']=[f +  g\colon  X +  Y\cong X' +  Y']$.
\end{itemize}
\end{prop}

\begin{proof}
Applying (R7) and (R10) to
$$\xymatrix{Y\ar@{>->}[r]\ar[d]^{\cong}_{g}&X+Y\ar[d]^{\cong}_{f+g}\ar@{->>}[r]&X\ar[d]^{\cong}_{f}\\
Y'\ar@{>->}[r]&X'+Y'\ar@{->>}[r]&X'}$$
we obtain (R7$'$).

Given a short exact sequence $\Delta\colon X\st{f}\into Y\st{p}\onto C$, a splitting
$$\xymatrix{X\ar@{>->}[r]\ar[d]^{\cong}_{1_{X}}&C+X\ar[d]^{\cong}_{(s_{\Delta},f)}\ar@{->>}[r]&C\ar[d]^{\cong}_{1_{C}}\\
X\ar@{>->}[r]^{f}&Y\ar@{->>}[r]^{p}&C}$$
and (R7) yield
$$[\Delta]=[(s_{\Delta},f)].$$
Such a splitting is not unique, but if $(s_{\Delta}',f)$ is another then there is a unique $h\colon C\r X$ such that $s_{\Delta}'=s_{\Delta}+fh$, i.e.
$$(s_{\Delta}',f)=(s_{\Delta},f)\left(
\begin{smallmatrix}
1_{C}&0\\h&1_{X}
\end{smallmatrix}
\right).$$
Hence, by (R6), 
$$[(s_{\Delta}',f)]=[(s_{\Delta},f)]+\left[\left(
\begin{smallmatrix}
1_{C}&0\\h&1_{X}
\end{smallmatrix}
\right)\right]=[(s_{\Delta},f)].$$
compare the proof of \cite[Proposition 1.1]{ranicki} for the vanishing of $\left[\left(
\begin{smallmatrix}
1_{C}&0\\h&1_{X}
\end{smallmatrix}
\right)\right]$.

Let us check that (R7) and (R8) follow from the relations in the statement. Given
$$\xymatrix{
\X\ar@{>->}[r]^-f\ar[d]^\cong_{\alpha}&\Y\ar@{->>}[r]\ar[d]^\cong_{\beta}&C^f\ar[d]^\cong_{\gamma}\\
\X'\ar@{>->}[r]^-{f'}&\Y'\ar@{->>}[r]&C^{f'}}$$
we can choose compatible splittings defined by $s_{\Delta}$ and $s_{\Delta'}=\beta s_{\Delta}\gamma^{-1}$. Hence by (R6) and  (R7$'$)
\begin{align*}
[(s_{\Delta'},f')]&=[(\beta s_{\Delta}\gamma^{-1},\beta f\alpha^{-1})]
=\left[\beta (s_{\Delta},f)
(\gamma^{-1}+\alpha^{-1})
\right]\\
&=
[\beta]+ [(s_{\Delta},f)]+[\gamma^{-1}]^{[X]}+[\alpha^{-1}]
=
[\beta]+ [(s_{\Delta},f)]-[\gamma]^{[X]}-[\alpha].
\end{align*}
This is (R7).

Given a staircase diagram
$$\xymatrix{&&C^g\\
&C^f\ar@{>->}[r]^{\bar g}&C^{gf}
\ar@{->>}[u]\\
\X\ar@{>->}[r]^-f&\Y\ar@{>->}[r]^-g\ar@{->>}[u]&\Z\ar@{->>}[u]}$$
we can choose splittings defined by morphisms $s_{\Delta_{f}}$, $s_{\Delta_{g}}$, $s_{\Delta_{gf}}$ and $s_{\widetilde{\Delta}}$, such that $s_{\Delta_{g}}=s_{\Delta_{gf}}s_{\widetilde{\Delta}}$ and $s_{\Delta_{gf}}\bar{g}=gs_{\Delta_{f}}$. Therefore,
\begin{align*}
(s_{\Delta_{g}},g)(1_{C^{g}}+(s_{\Delta_{f}},f))&=(s_{\Delta_{gf}},gf)(
(s_{\widetilde{\Delta}},\bar{g})+1_{X}).
\end{align*}
Now (R8) follows from (R6) and (R7$'$) applied to this formula.
\end{proof}

An additive category category $\C A$ satisfies the \emph{Krull--Remak--Schmidt theorem} if any object in $\C A$ is a direct sum of indecomposables and, up to a permutation, the indecomposable components in such a direct sum are uniquely determined up to isomorphism.

\begin{cor}
Let $\C A$ be an additive category satisfying  the Krull--Remak--Schmidt theorem. Denote $S$ a skeletal set of indecomposables. Then $\D{D}_{*}(\C A)$ is weakly equivalent to
$$K_{0}(\C A)\otimes K_{0}(\C A)\st{\grupo{\cdot,\cdot}}\To K_{1}(\C A)\st{0}\To K_{0}(\C A)=\abg{S},$$
where $\grupo{[X],[X]}=[-1_{X}]$ and $\grupo{[X],[Y]}=0$ for $X,Y\in S$, $X\neq Y$.
\end{cor}

\begin{proof}
We can suppose without loss of generality that $\C A$ has functorial coproducts and that the monoid of objects is freely generated by $S$, see \cite[Proposition 4.3]{k1wc}. By Remark \ref{masmenos}, $\D{D}_{0}^{+}(\C A)=\nig{S}$. Moreover, by the Krull--Remak--Schmidt theorem the image of $\partial$ is the commutator subgroup and $\pi_{0}\D{D}_{*}^{+}(\C A)=\abg{S}=K_{0}(\C A)$. The commutator bracket induces a monomorphism $\wedge^{2}\abg{S}\hookrightarrow\nig{S}$,  $[X]\wedge [Y]\mapsto[[X],[Y]]$. This injection factors through $\D{D}_{1}^{+}(\C A)$. In order to define a factorization, we need to choose a total order $\leq$ in $S$. A factorization is given by
\begin{align*}
\wedge^{2}\abg{S}&\hookrightarrow \D{D}_{1}^{+}(\C A),&
[X]\wedge [Y]&\mapsto \grupo{[Y],[X]},\qquad X<Y\in S.
\end{align*}
Dividing out $\wedge^{2}\abg{S}$ from $\D{D}^+_{*}(\C A)$ we obtain a weakly equivalent stable quadratic module $C_*$ with $C_0=\pi_0C_0=\abg{S}=K_{0}(\C A)$ and $\partial=0$. We must now identify $C_1$ and the bracket. Proposition \ref{minimismo} and the explicit description of a stable quadratic module defined by a presentation show that $C_1$ is the abelian group defined by generators $[f\colon X\cong Y]$ for each isomorphism in $\C A$,
and relations
\begin{itemize}
\item $
2 \cdot\left[\left(
\begin{smallmatrix}
 0&1_X\\1_X&0
\end{smallmatrix}
\right)\right]=0$ for $X\in S$,
\item $
\left[\left(
\begin{smallmatrix}
 0&1_X\\1_Y&0
\end{smallmatrix}
\right)\right]=0$ for $X,Y\in S$, $X\neq Y$,
\item $[g\colon Y\cong Z]+[f\colon X\cong Y]=[gf]$ for any two composable isomorphisms,
\item $[f\colon X\cong X']+[g\colon Y\cong Y']=[f +  g]$ for any two isomorphisms.
\end{itemize}
The first relation follows from the third one. Hence, $C_1$ is the quotient of Ranicki's $K_1^{\operatorname{iso}}(\C A)$ by the isomorphisms of the canonical structure in the sense of \cite[\S5]{ranicki} defined by the permutations of different $S$ factors in a direct sum decomposition. Therefore, $C_1=K_1(\C A)$ by \cite[Proposition 5.3]{ranicki}. The computation of the non-trivial brackets follows as in \cite[Corollary 1.10]{1tk}.
\end{proof}

\begin{rem}\label{warrior}
This corollary can be applied to several additive categories: modules of finite length over a ring, finitely generated projective modules over a semiperfect ring, finitely generated free modules over a ring $R$ with the invariant basis number property, etc. In particular, if $R$ is a commutative local ring or an Euclidean domain we obtain the stable quadratic module in Example \ref{sqmexamples} (4).
\end{rem}

In this section, we have obtained a connection between determinant functors and $K$-theory through an explicit computation of a generic but simple example. In the following section we will see that this relation is much broader.

\subsection{The connection to $K$-theory}\label{contok}

Let $\ho\mathbf{Spec}_0$ be the full coreflective subcategory of the
stable homotopy category spanned by \emph{connective spectra},
i.e.~spectra with trivial homotopy groups in negative
dimensions. Let $\ho\mathbf{Spec}_0^1$ be the full reflective subcategory of $\ho\mathbf{Spec}_0$ spanned by spectra with homotopy groups concentrated in dimensions $0$ and $1$. The reflection functor $\ho\mathbf{Spec}_0\r \ho\mathbf{Spec}_0^1$ takes a connective spectrum to its \emph{$1$-type}. It is well known that the homotopy category of \Picardgroupoids\ is equivalent to  $\ho\mathbf{Spec}_0^1$, and the equivalence is compatible with the corresponding notions of homotopy groups and $k$-invariant. Recall that, on homotopy groups of spectra, the $k$-invariant is simply the action of the stable Hopf map. There are several ways of realizing this equivalence. The equivalence in \cite{1tk} between $\ho\mathbf{Spec}_0^1$ and the homotopy category of stable quadratic modules $\ho\mathbf{squad}$ is particularly well adapted to the goal of this paper.

\begin{lem}[{\cite[Lemma 4.22]{1tk}}]\label{equi}
There is a functor
$$\lambda_0\colon\ho\mathbf{Spec}_0\To \ho\mathbf{squad}$$
together with natural isomorphisms
$$\pi_i\lambda_0X\cong\pi_iX,\;\;i=0,1,$$ 
compatible with the $k$-invariants, which restricts to an equivalence of
categories
$$\lambda_0\colon\ho\mathbf{Spec}_0^1\st{\sim}\To \ho\mathbf{squad}.$$
\end{lem}

Therefore the functor $\lambda_0$ can be regarded as an algebraic model for the $1$-type of  a connective spectrum.

\begin{exm}\label{sqmexamples3}
If $S$ is the sphere spectrum, its first two homotopy groups and the $k$-invariant are as in Example \ref{sqmexamples2} (3), hence $\lambda_{0}S$ is weakly equivalent to Example \ref{sqmexamples} (3).

If $R$ is a commutative local ring, the first two homotopy groups of its connective $K$-theory spectrum $K(R)$  are as in  Example \ref{sqmexamples2} (4). The  $k$-invariant connecting them is also as described therein. This can be easily checked by looking at the canonical map of spectra $S\r K(R)$, the unit of the ring spectrum structure on $K(R)$, which is the identity in $\pi_{0}$ and $\mathbb Z/2\r R^{\times}\colon 1\mapsto -1$ in $\pi_{1}$. Therefore $\lambda_{0}K(R)$ is weakly equivalent to Example \ref{sqmexamples} (4).
\end{exm}

Examples of connective spectra  are  Quillen's $K$-theory
of an exact category $K(\C{E})$ \cite{hakt} (although there is also a non-connective version), Waldhausen's $K$-theory of a category
with cofibrations and weak equivalences $K(\C{W})$ \cite{akts}, Garkusha's derived
$K$-theory  of an exact category $DK(\C{E})$   \cite{sdckt2}, its generalization for a Waldhausen category $DK(\C{W})$ \cite{malt}, Maltsiniotis's $K$-theory of a strongly
triangulated category $K(^{s}\!\C{T}_{\infty})$ \cite{cts}, 
and
two of Neeman's $K$-theories of a
triangulated category, $K({}^d\!\C{T})$ and $K({}^v\!\C{T})$ \cite{kttcneeman}, see Section \ref{tambien} below.

\begin{thm}\label{main2}
Let $\C{W}$ be a Waldhausen category, $\C{T}$ a triangulated category,  and 
$\C{T}_{\infty}$ a strongly triangulated category. 
There are natural isomorphisms in $\ho\mathbf{squad}$:
\begin{align*}
\D{D}_* (\C{W})&\cong\lambda_0K(\C{W}),&
\DD{D}_*(\C{W})&\cong\lambda_0 DK(\C{W}),&
\D{D}_* (^{s}\!\C{T}_{\infty})&\cong\lambda_0 K(^{s}\!\C{T}_{\infty}),\\
\D{D}_* (^{d}\!\C{T})&\cong\lambda_0 K(^{d}\!\C{T}),&
\D{D}_* (^{v}\!\C{T})&\cong\lambda_0 K(^{v}\!\C{T}). 
\end{align*}
\end{thm}

This follows from Theorem \ref{cosa} and Examples \ref{mainexm} and \ref{mainexmk}.

The stable quadratic module $\D{D}_* (^{d}\!\C{T})$ is not related to any $K$-theory spectrum. Actually, Breuning defines the $K$-theory of a triangulated category in dimensions $i=0,1$ as  $$K_i(^b\!\C{T})=\pi_i\D{D}_*(^b\!\C{T}).$$

\subsection{Comparison morphisms}\label{comparase}
There are several comparison morphisms between the $K$-theories which appeared in the introduction. In this section we recover these morphisms in dimensions $i=0,1$ from certain explicit morphisms between the stable quadratic modules of Definition \ref{univtargetwald} that calculate, by Theorem \ref{main2}, the 1-types of the $K$-theory spectra. 
For more explicit descriptions of the maps of spectra involved, and of the spectra themselves, we refer the reader to Section \ref{torremolinos}.


For a Waldhausen category $\C W$, we have honest and derived determinant functors, see Definitions \ref{detwald} and \ref{derdet}. A derived determinant functor $\det'\colon\C W\r\C P$ yields an honest determinant functor, defined by precomposition with the functor sending a weak equivalence in $\C W$ to its corresponding isomorphism in the homotopy category $\ho\C W$,
$$\det\colon \we(\C W)\To\iso{\ho \C W}\st{\det'}\To\C P.$$
Additivity data for $\det$ are defined as for $\det'$. This gives rise to a morphism of stable quadratic modules defined on generators by,
\begin{align}\label{normalder}
\D{D}_{*}(\C W)&\To \DD{D}_{*}(\C W),\\
\nonumber [X]&\;\mapsto\;[X],\\
\nonumber [f\colon X\st{\sim}\r X']&\;\mapsto\;[\{f\}\colon X\st{\cong}\r X'],\\
\nonumber [\Delta]&\;\mapsto\;[\Delta].
\end{align}
Here $X$ is an object of $\C W$, $\{f\}$ denotes the homotopy class of a weak equivalence $f$ in $\C W$, and $\Delta$ is a cofiber sequence. This morphism was shown in \cite{malt} to be an isomorphism, hence
\begin{align*}
K_{0}(\C W)&\cong DK_{0}(\C W),&
K_{1}(\C W)&\cong DK_{1}(\C W).
\end{align*}

Let $\C T$ be a triangulated category. Special octahedra are also ordinary, and \exact\ triangles and ordinary octahedra are also virtual. Hence, restricting additivity data, any virtual determinant functor produces a Breuning determinant functor, and any Breuning determinant functor yields a special determinant functor. This gives rise to obvious morphisms of stable quadratic modules defined by the aforementioned inclusions of sets of generators,
\begin{equation}\label{dv}
\D{D}_{*}(^{d}\C T)\To \D{D}_{*}(^{b}\C T)\To \D{D}_{*}(^{v}\C T).
\end{equation}
These morphisms are the identity in degree $0$. They can be easily shown to be isomorphisms on $\pi_{0}$, see \cite{kttcneeman,dftc}, 
$$K_{0}(^{d}\C T)\cong K_{0}(^{b}\C T)\cong K_{0}(^{v}\C T).$$
Moreover, the first one is surjective in degree $1$, since  $\D{D}_{*}(^{d}\C T)$ and $\D{D}_{*}(^{b}\C T)$ have the same generators, but the latter has more relations than the former, corresponding to non-special octahedra. In particular, the first morphism induces an epimorphism in $\pi_{1}$,
\begin{equation*}
K_{1}(^{d}\C T)\onto K_{1}(^{b}\C T).
\end{equation*}
We do not know of any example where this morphism has a non-trivial kernel. This may be due to our lack of knowledge about non-special octahedra.

Let $\C T$ be a triangulated category with a $t$-structure with heart $\C A$. Exact sequences in $\C A$ extend uniquely to exact triangles in $\C T$, see Section \ref{forgrad}, and staircase diagrams extend uniquely to octahedra, which are therefore special, see Remark \ref{speziale}. Hence, a special determinant functor on $\C T$ restricts to a determinant functor on $\C A$. This defines an obvious stable quadratic module morphism
\begin{equation}\label{heartd}
\D{D}_{*}(\C A)\To \D{D}_{*}(^{d}\C T).
\end{equation}
Theorem \ref{todo} below shows that, if the $t$-structure is bounded and non-degenerate, this morphism and the two morphisms in \eqref{dv}   are weak equivalences.  In particular,
$$K_{i}(\C A)\cong K_{i}(^{d}\C T)\cong K_{i}(^{b}\C T)\cong K_{i}(^{v}\C T),\qquad i=0,1.$$
Actually, the morphisms in \eqref{dv} are isomorphisms by the five lemma, since  they are the identity in degree $0$. The isomorphism $K_{1}(\C A)\cong K_{1}(^{b}\C T)$ was previously obtained by Breuning \cite[Corollary 5.3]{dftc} using different methods.

Similarly, if $\C E$ is an exact category, a determinant functor on the strongly triangulated category $D^{b}(\C E)$ `restricts' to a Deligne determinant functor on the Waldhausen category $C^{b}(\C E)$ along the canonical functor $C^{b}(\C E)\r D^{b}(\C E)$, and also on $\C E$, regarded as complexes concentrated in degree $0$. Hence we have stable quadratic module morphisms
\begin{equation}\label{heartdexact}
\D{D}_{*}(\C E)\To\D{D}_{*}(C^{b}(\C E))\To \D{D}_{*}(^{s}D^{b}(\C E)).
\end{equation}
The second morphism here is the identity in degree $0$. It is surjective in degree~$1$, since any \exact\ triangle in $D^{b}(\C T)$ is isomorphic to an \exact\ triangle coming from a cofiber sequence in $C^{b}(\C E)$, and any isomorphism in $D^{b}(\C E)$ can be represented by a zigzag of weak equivalences in $C^{b}(\C E)$. Hence, we obtain an isomorphism in $\pi_0$ and a surjection in $\pi_{1}$,
\begin{align*}
K_{0}(C^{b}(\C E))&\cong K_{0}(^{s}D^{b}(\C E)),&
K_{1}(C^{b}(\C E))&\onto K_{1}(^{s}D^{b}(\C E)).
\end{align*}
The first morphism in \eqref{heartdexact} is a weak equivalence by the following theorem and \cite{tckt}, actually
$$K_{i}(\C E)\cong K_{i}(C^{b}(\C E)),\qquad i\geq 0.$$

\begin{thm}\label{pepillo}
The images under the functor $\lambda_{0}$ of Lemma \ref{equi} of the maps between $K$-theory spectra 
\begin{align*}
K(\C W)\To &DK(\C W),\qquad K({}^{d}\C T)\To K({}^{v}\C T),\qquad K(\C A)\To K({}^{d}\C T),\\
&K(\C E)\st{\sim}\To K(C^{b}(\C E)),\qquad K(C^{b}(\C E))\To K({}^{s}D^{b}(\C E)).
\end{align*}
of Examples \ref{lasrelaciones2} (1), (2), (3), (7) and (8), respectively, coincide with the stable quadratic module morphisms given in \eqref{normalder}, the composite of \eqref{dv}, and \eqref{heartd} and \eqref{heartdexact}. 
\end{thm}

This follows from Theorem \ref{cosa} and Example \ref{lasrelaciones2}.
\medskip

There is yet another comparison morphism of stable quadratic modules which is not related to spectra, since it is connected to Breuning determinant functors. Let $\C T_3$ be a $3$-pretriangulated category. \Exact\ octahedra are ordinary octahedra in the underlying triangulated structure. Therefore, a Breuning determinant functor on the triangulated category underlying $\C T_3$ yields a determinant functor on $\C T_3$ in the sense of Definition \ref{3eme}. This gives rise to a stable quadratic module morphism
$$\D{D}_{*}(^{s}\C T_{3})\To \D{D}_{*}(^{b}\C T_{3}),$$
which is the identity on generators, but the target has more relations than the source, corresponding to non-\exact\ octahedra. It is actually the identity in degree $0$ and surjective in degree $1$. In particular, we obtain an isomorphism on $\pi_{0}$ and a surjection on $\pi_{1}$,
\begin{align*}
K_{0}(^{s}\C T)&\cong K_{0}(^{b}\C T),&
K_{1}(^{s}\C T)&\onto K_{1}(^{b}\C T).
\end{align*}

\section{Applications}

\subsection{Derived and non-derived determinant functors on a Waldhausen category}\label{dnd}

Grothendieck asked in a letter to Knudsen whether determinant functors on an additive or abelian category $\C{E}$ coincide essentially with determinant functors in the bounded derived category $D^{b}(\C{E})$ regarded as triangulated category equipped with a `category of true triangles' \cite[Appendix B]{dfec}. We extend the question to exact categories. We interpret the `category of true triangles' to be the bounded derived category of the exact category $S_{2}(\C{E})$ of short exact sequences in $\C{E}$, which coincides with the homotopy category of the Waldhausen category $S_{2}C^{b}(\C{E})$ of short exact sequences of bounded complexes in $\C{E}$. With this interpretation,  determinant functors in the  triangulated category equipped with a `category of true triangles' are  derived determinant functors in $C^{b}(\C{E})$. 


The Waldhausen category $C^{b}(\C{E})$ has cylinders and a saturated class of weak equivalences, therefore the two following results answer Grothendienck's question and its generalization positively. 

\begin{cor}\label{knudmum}
If we regard $\C{E}$ as the full subcategory of complexes in $C^{b}(\C{E})$ concentrated in degree $0$, then 
any
determinant functor on $\C{E}$ factors through a  determinant
functor in $C^{b}(\C{E})$ in an essentially unique way.
\end{cor}

This follows from Theorems \ref{mainconcreto} and \ref{main2} and from the Gillet--Waldhausen theorem \cite{tckt}. A direct proof of this result can be found in \cite{dfec}.

\begin{cor}\label{wder}
Let $\C{W}$ be a Waldhausen category with cylinders and a saturated
class of weak equivalences, i.e.~weak equivalences are exactly those
maps in $\C{W}$ which become invertible in $\ho\C{W}$. Then any
determinant functor on $\C{W}$ factors through a derived determinant
functor in an essentially unique way.
\end{cor}

This follows from Theorem \ref{mainconcreto}  and \cite[Theorem 6.1]{malt}. 

\begin{rem}
If the class of weak equivalences in $\C{W}$ is not saturated then Weiss's Whitehead group $\operatorname{Wh}(\C{W})$ may not vanish \cite{hlwc}, and in this case the universal determinant functor $\det\colon\C{W}\r\D{D}_*(\C{W})$ need not factor through a derived determinant functor, compare \cite[Remark 6.3]{malt}.
\end{rem}
  

\subsection{Generators and (some) relations for $K_{1}$}\label{genrel}

Nenashev \cite{K1gr} considered pairs of short exact sequences over the same objects in an exact category $\C{E}$,
\begin{equation}\label{nenashev}
\xymatrix@R=5pt@C=10pt{
X\ar@<.7ex>@{>->}[r]^{f}\ar@<-.7ex>@{>->}[r]_{f'}&Y\ar@<.7ex>@{->>}[r]^{p}\ar@<-.7ex>@{->>}[r]_{p'}&C.}
\end{equation}
Such a pair yields an element in $K_{1}(\C{E})$. 
Nenashev proved that any element in $K_{1}(\C{E})$ is of this kind and computed a set of relations among them, associated to $3\times 3$ diagrams, yielding a presentation of $K_{1}(\C{E})$.

Vaknin \cite{k1tc} considered pairs of \exact\ triangles over the same objects in a triangulated category $\C{T}$,
\begin{equation}\label{vaknin}
\xymatrix{X\ar@<.5ex>[r]^f\ar@<-.5ex>[r]_{f'}&Y\ar@<.5ex>[r]^i\ar@<-.5ex>[r]_{i'}&Z\ar@<.5ex>[r]^q\ar@<-.5ex>[r]_{q'}&\Sigma X.}
\end{equation}
Using similar techniques, Vaknin \cite{k1tc} proved that any element in Neeman's $K_{1}(^{d}\!\C{T})$ is of this kind and computed a set of relations among them, extending Nenashev's, yielding a presentation of $K_{1}(^{d}\!\C{T})$, as in the exact case, see Proposition \ref{relaciontri}.

Muro and Tonks considered in \cite{k1wc} diagrams in a Waldhausen category $\C{W}$,
\begin{equation}\label{murotonks}
\xymatrix@R=5pt@C=10pt{&&\cone{f}\ar@{<-}[rd]^\sim&\\
X\ar@<.7ex>@{>->}[r]^{f}\ar@<-.7ex>@{>->}[r]_{f'}&B\ar@{->>}[ru]\ar@{->>}[rd]&&C,\\
&&\cone{f'}\ar@{<-}[ru]_\sim&}
\end{equation}
consisting of two cofiber sequences and two weak equivalences. 
They extended Nenashev's results, showing that any element in $K_{1}(\C{W})$ is of this kind and computing a set of relations among them generalizing Nenashev's. Some evidence was given for the conjecture that these relations  define a presentation of $K_{1}(\C{W})$.

In this section we indicate how these results extend to Breuning's $K_{1}(^b\!\C{T})$, Neeman's  $K_{1}(^{v}\!\C{T})$, and Maltsiniotis's~$K_{1}(^s\!\C{T}_\infty)$. 

Let $\filledsquare=b,d,v$ or $s$. Denote simply by $\C T$ a triangulated category, or a strongly triangulated category if $\filledsquare=s$. A \emph{$\filledsquare$-triangle} is a \exact\ triangle if $\filledsquare=b,d,s$ and a virtual triangle if $\filledsquare=v$. Given  a pair of $\filledsquare$-triangles \eqref{vaknin} we define
\begin{align*}
[X\mathop{\rightrightarrows}\limits^{f}_{f'}Y\mathop{\rightrightarrows}\limits^{i}_{i'}Z \mathop{\rightrightarrows}\limits^{q}_{q'}\Sigma X]&
=-[X\st{f}\r Y\st{i}\r Z\st{q}\r\Sigma X]+[X\st{f'}\r Y\st{i'}\r Z\st{q'}\r\Sigma X]\in \D{D}_*(^{\filledsquare}\C{T}).
\end{align*}
Notice that this element is in the kernel of $\partial$, i.e.~$$[X\mathop{\rightrightarrows}\limits^{f}_{f'}Y\mathop{\rightrightarrows}\limits^{i}_{i'}Z \mathop{\rightrightarrows}\limits^{q}_{q'}\Sigma X]\in\pi_1\D{D}_*(^{\filledsquare}\C{T})\cong
K_1(^{\filledsquare}\C{T}).$$
Moreover, a pair given by twice the same $\filledsquare$-triangle is zero $[X\mathop{\rightrightarrows}\limits^{f}_{f}Y\mathop{\rightrightarrows}\limits^{i}_{i}Z \mathop{\rightrightarrows}\limits^{q}_{q}\Sigma X]=0.$

\begin{thm}\label{otromaintriang}
Any element in $K_1(^{\filledsquare}\C{T})$ is represented by a pair of $\filledsquare$-triangles.
\end{thm}

This result follows from Theorem \ref{otromain}. 


\begin{defn}\label{27}
A \emph{$3\times3$ diagram} in $\C T$ is a diagram
\begin{equation}\label{444}
\xymatrix{
X'\ar[r]^-{f^{X}}\ar[d]_-{f'}&X\ar[r]^-{i^{X}}\ar[d]^-{f}&X''\ar[r]^-{q^{X}}\ar[d]^-{f''}&\Sigma X'\ar[d]^-{\Sigma f'}\\
Y'\ar[r]^-{f^{Y}}\ar[d]_-{i'}&Y\ar[r]^-{i^{Y}}\ar[d]^-{i}&Y''\ar[r]^-{q^{Y}}\ar[d]^-{i''}&\Sigma Y'\ar[d]^-{\Sigma i'}\\
Z'\ar[r]^-{f^{Z}}\ar[d]_-{q'}&Z\ar[r]^-{i^{Z}}\ar[d]^-{q}&Z''\ar[r]^-{q^{Z}}\ar[d]^-{q''}\ar@{}[rd]|*+<.15cm>[o][F-]{\scriptscriptstyle -1}&\Sigma Z'\ar[d]^-{-\Sigma q'}\\
\Sigma X'\ar[r]_-{\Sigma f^{X}}&\Sigma X\ar[r]_-{\Sigma i^{X}}&\Sigma X''\ar[r]_-{-\Sigma q^{X}\;}&\Sigma^{2}X'}
\end{equation}
where all squares commute except for the bottom right square, which is anticommutative, $(\Sigma q')q^{Z}+(\Sigma q^{X})q''=0$. Moreover, all rows and columns are $\filledsquare$-triangles. 

A \emph{$\filledsquare$-octahedron} is an ordinary octahedron if $\filledsquare=b$, a special octahedron if $\filledsquare=d$, a \exact\ octahedron if $\filledsquare=s$, and a virtual octahedron if $\filledsquare=v$, 

A $3\times3$ diagram is \emph{$\filledsquare$-coherent} if there exist four 
$\filledsquare$-octahedra 
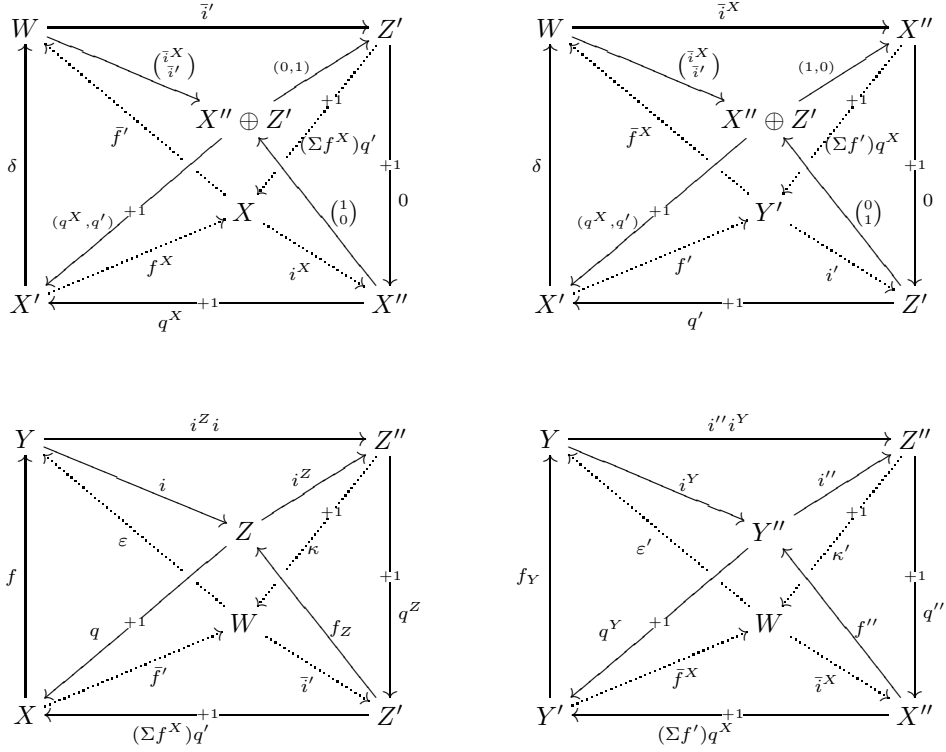
\begin{equation*}
\quad\octahedron{objA=X',objB=X,objC=W,objD=X'',objE=X''\oplus Z',objF=Z',
            arAB=f^{X},arBC=\bar{f}',arAC=\delta
,arDE=\binom{1}{0},arEF={\scriptscriptstyle (0,1)\;},arCE=\;\binom{\bar{i}^{X}}{\bar{i}'},arCF=\bar{i}',arBD=i^{X},arDA=q^{X},arFD=0,
            arFB={\hspace{-3pt}\scriptstyle(\Sigma f^{X})q'},arEA={\scriptscriptstyle (q^{X},q')\hspace{-5pt}}}
                        \qquad\qquad\qquad
\octahedron{objA=X',objB=Y',objC=W,objD=Z',objE=X''\oplus Z',objF=X'',
            arAB=f',arBC=\bar{f}^{X},arAC=\delta
,arDE=\binom{0}{1},arEF={\scriptscriptstyle (1,0)\;},arCE=\;\binom{\bar{i}^{X}}{\bar{i}'},arCF=\bar{i}^{X},arBD=i',arDA=q',arFD=0,
            arFB={\hspace{-3pt}(\Sigma f')q^{X}},arEA={\scriptscriptstyle (q^{X},q')\hspace{-5pt}}}
\end{equation*}
$$\quad\octahedron{objA=X,objB=W,objC={Y},objD=Z',objE=Z,objF=Z'',
            arAB=\bar{f}',arBC=\varepsilon,arAC=f,arDE=f_{Z},arEF=i^{Z},arCE=i,arCF=i^{Z}i,arBD=\bar{i}',arDA=(\Sigma f^{X})q',arFD=q^{Z},
            arFB=\kappa,arEA=q}
                        \qquad\qquad\qquad
\octahedron{objA=Y',objB=W,objC=Y,objD=X'',objE=Y'',objF=Z'',
            arAB=\bar{f}^{X},arBC=\varepsilon',arAC=f_{Y},arDE=f'',arEF=i'',arCE=i^{Y},arCF=i''i^{Y},arBD=\bar{i}^{X},arDA=(\Sigma f')q^{X},arFD=q'',
            arFB=\kappa',arEA=q^{Y}}
$$
such that in the two last octahedra 
$\varepsilon=\varepsilon'$ and $\kappa=\kappa'$. 
\end{defn}

This kind of diagram has previously appeared in \cite{k1tc,largo}.

\begin{rem}\label{227}
As Vaknin pointed out in \cite[Remark 5.2]{k1tc}, 
if we start with a $3\times3$ diagram of \exact\ triangles as in \eqref{444} it is always possible to construct four  octahedra as above, but, in general, one cannot guarantee $\varepsilon=\varepsilon'$ or $\kappa=\kappa'$, i.e.~the third and fourth octahedra may contain different \exact\ triangle completions of $i^{Z}i=i''i^{Y}\colon Y\r Z''$. The procedure is as follows:
\begin{enumerate}
\item Extend $(q^X,q')$ to an exact triangle. This produces $\bar{i}^{X}$, $\bar{i}'$ and $\delta$. All morphisms in the first octahedron are now defined, except for $\bar f'$. The octahedral axiom yields a morphism $\bar f'$ such that the resulting diagram is an octahedron. This octahedron is always special since we have the following triangle isomorphisms where the bottom triangle is \exact, being the direct sum of two \exact\ triangles,
$$\xymatrix@C=50pt{
X\ar[r]^-{\binom{\bar f'}{-i^X}}\ar@{=}[d]&
W\oplus X''\ar[r]^-{\left(\begin{smallmatrix}
                    \bar{i}^{X}&1\\\bar{i}'&0
                   \end{smallmatrix}\right)
}
\ar[d]^-{\left(\begin{smallmatrix}
                    \bar{i}^{X}&1\\1&0
                   \end{smallmatrix}\right)}_-{\cong}
&
X''\oplus Z'\ar[r]^-{(0,(\Sigma f^{X})q')}
\ar@{=}[d]
&\Sigma X\ar@{=}[d]\\
X\ar[r]_-{0\oplus{\bar f'}}&
X''\oplus W\ar[r]_-{1\oplus\bar{i}'}&
X''\oplus Z'\ar[r]_-{0\oplus (\Sigma f^{X})q'}
&\Sigma X
}$$

$$\xymatrix@C=50pt{
X''\oplus Z'\ar[r]^-{\left(\begin{smallmatrix}
                     q^X&q'\\0&-1
                    \end{smallmatrix}\right)
}\ar@{=}[d]&
\Sigma X'\oplus Z'\ar[r]^-{(\Sigma f^X,(\Sigma f^X)q')}
\ar[d]^-{\left(\begin{smallmatrix}
                    1&q'\\0&-1
                   \end{smallmatrix}\right)}_-{\cong}
&
\Sigma X\ar[r]^-{\binom{\Sigma i^X}{0}}\ar@{=}[d]&
\Sigma X''\oplus \Sigma Z'\ar@{=}[d]\\
X''\oplus Z'\ar[r]_-{q^X\oplus1}&
\Sigma X'\oplus Z'\ar[r]_-{\Sigma f^X\oplus 0}&
\Sigma X\ar[r]_-{{\Sigma i^X}\oplus{0}}&
\Sigma X''\oplus \Sigma Z'
}$$

\item All morphisms in the second octahedron are defined, except for $\bar f^{X}$. Apply the octahedral axiom to obtain  a morphism $\bar f^{X}$ such that the resulting diagram is an octahedron. This octahedron is also always special since we have the following triangle isomorphisms with \exact\ bottom part,
$$\xymatrix@C=50pt{
Y'\ar[r]^-{\binom{\bar f^X}{-i'}}\ar@{=}[d]&
W\oplus Z'\ar[r]^-{\left(\begin{smallmatrix}
                   \bar{i}^X&0\\\bar{i}'&1
                  \end{smallmatrix}\right)
}
\ar[d]^-{\left(\begin{smallmatrix}
                    1&0\\\bar{i}'&1
                   \end{smallmatrix}\right)}_-{\cong}
&
X''\oplus Z'\ar[r]^-{((\Sigma f')q^X,0)}\ar@{=}[d]&
\Sigma Y'\ar@{=}[d]\\
Y'\ar[r]_-{{\bar f^X}\oplus 0}&
W\oplus Z'\ar[r]_-{\bar{i}^X\oplus1}&
X''\oplus Z'\ar[r]_-{(\Sigma f')q^X\oplus0}&
\Sigma Y'
}$$

$$\xymatrix@C=50pt{
X''\oplus Z'\ar[r]^-{\left(\begin{smallmatrix}
                   q^X&q'\\-1&0
                  \end{smallmatrix}\right)
}\ar@{=}[d]&
\Sigma X'\oplus X''\ar[r]^-{(\Sigma f',(\Sigma f')q^X)}\ar[d]^-{\left(\begin{smallmatrix}
                    0&-1\\1&q^{X}
                   \end{smallmatrix}\right)}_-{\cong}&
\Sigma Y'\ar[r]^-{\binom{0}{\Sigma i'}}\ar@{=}[d]&
\Sigma X''\oplus \Sigma Z'\ar@{=}[d]\\
X''\oplus Z'\ar[r]_-{1\oplus q'}&
X''\oplus\Sigma  X'\ar[r]_-{0\oplus\Sigma f'}&
\Sigma Y'\ar[r]_{{0}\oplus{\Sigma i'}}&
\Sigma X''\oplus \Sigma Z'
}$$

\item The morphisms $\varepsilon$ and $\kappa$ in the third octahedron are obtained by applying the octahedral axiom, the rest has been previously defined. Similarly $\varepsilon'$ and $\kappa'$ in the fourth octahedron. Nothing guarantees that these octahedra are special, nor that $\varepsilon=\varepsilon'$ and $\kappa=\kappa'$. 
\end{enumerate}

\end{rem}


%

The following result extends Vaknin's relations to $\filledsquare=b,s,v$, but we do not prove their sufficiency in these contexts.

\begin{prop}\label{relaciontri}
Suppose we have two $\filledsquare$-coherent  $3\times3$ diagrams over the same objects, 
$j=1,2$,
\begin{equation}\label{2444}
\xymatrix{
X'\ar[r]^-{f^{X}_{j}}\ar[d]_-{f'_{j}}&X\ar[r]^-{i^{X}_{j}}\ar[d]^-{f_{j}}&X''\ar[r]^-{q^{X}_{j}}\ar[d]^-{f''_{j}}&\Sigma X'\ar[d]^-{\Sigma f'_{j}}\\
Y'\ar[r]^-{f^{Y}_{j}}\ar[d]_-{i'_{j}}&Y\ar[r]^-{i^{Y}_{j}}\ar[d]^-{i^{j}}&Y''\ar[r]^-{q^{Y}_{j}}\ar[d]^-{i''_{j}}&\Sigma Y'\ar[d]^-{\Sigma i'_{j}}\\
Z'\ar[r]^-{f^{Z}_{j}}\ar[d]_-{q'_{j}}&Z\ar[r]^-{i^{Z}_{j}}\ar[d]^-{q^{j}}&Z''\ar[r]^-{q^{Z}_{j}}\ar[d]^-{q''_{j}}\ar@{}[rd]|*+<.15cm>[o][F-]{\scriptscriptstyle -1}&\Sigma Z'\ar[d]^-{-\Sigma q'_{j}}\\
\Sigma X'\ar[r]_-{\Sigma f^{X}_{j}}&\Sigma X\ar[r]_-{\Sigma i^{X}_{j}}&\Sigma X''\ar[r]_-{-\Sigma q^{X}_{j}}&\Sigma^{2}X'}
\end{equation}
Then the following equation holds in  $K_{1}(^\filledsquare\!\C{T})$,
\begin{align*}
\label{rel6}&\hspace{-20pt}[X'\mathop{\rightrightarrows}\limits^{f^{X}_{1}}_{f^{X}_{2}}X\mathop{\rightrightarrows}\limits^{i^{X}_{1}}_{i^{X}_{2}}X'' \mathop{\rightrightarrows}\limits^{q^{X}_{1}}_{q^{X}_{2}}\Sigma X']
-[Y'\mathop{\rightrightarrows}\limits^{f^{Y}_{1}}_{f^{Y}_{2}}Y\mathop{\rightrightarrows}\limits^{i^{Y}_{1}}_{i^{Y}_{2}}Y'' \mathop{\rightrightarrows}\limits^{q^{Y}_{1}}_{q^{Y}_{2}}\Sigma Y']
+[Z'\mathop{\rightrightarrows}\limits^{f^{Z}_{1}}_{f^{Z}_{2}}Z\mathop{\rightrightarrows}\limits^{i^{Z}_{1}}_{i^{Z}_{2}}Z'' \mathop{\rightrightarrows}\limits^{q^{Z}_{1}}_{q^{Z}_{2}}\Sigma Z']\\
\nonumber ={}&
[X'\mathop{\rightrightarrows}\limits^{f_{1}'}_{f_{2}'}Y'\mathop{\rightrightarrows}\limits^{i_{1}'}_{i_{2}'}Z' \mathop{\rightrightarrows}\limits^{q_{1}'}_{q_{2}'}\Sigma X']
-[X\mathop{\rightrightarrows}\limits^{f_{1}}_{f_{2}}Y\mathop{\rightrightarrows}\limits^{i_{1}}_{i_{2}}Z \mathop{\rightrightarrows}\limits^{q_{1}}_{q_{2}}\Sigma X]
+[X''\mathop{\rightrightarrows}\limits^{f_{1}''}_{f_{2}''}Y''\mathop{\rightrightarrows}\limits^{i_{1}''}_{i_{2}''}Z'' \mathop{\rightrightarrows}\limits^{q_{1}''}_{q_{2}''}\Sigma X''].
\end{align*}
\end{prop}

This result follows from Corollary \ref{hartura}.


\subsection{On additivity 
for low-dimensional $K$-theory of triangulated categories}\label{adloc}
In this section we prove an additivity theorem for low-dimensional $K$-theories of (strongly) triangulated categories.


\begin{thm}[Additivity]
Let $F,G,H\colon \C{T}\r \C{T}'$ be exact functors between (strongly) triangulated categories. Denote $\filledsquare= b,d,$ or $v$ in the triangulated case and $\filledsquare =s$ in the strongly triangulated case. Suppose we have a natural $\filledsquare$-triangle,
\begin{equation}\label{paraad1}
F(X)\st{f(X)}\To  G(X)\st{i(X)}\To  H(X)\st{q(X)}\To \Sigma F(X),
\end{equation}
such that for any $\filledsquare$-triangle $X\st{f}\r Y\st{i}\r Z\st{q}\r\Sigma X$, the $3\times 3$ diagram 
\begin{equation}\label{paraad2}
\xymatrix{F(X)\ar[r]^-{f(X)}\ar[d]_-{F(f)}&G(X)\ar[r]^-{i(X)}\ar[d]^-{G(f)}&H(X)\ar[r]^-{q(X)}\ar[d]^-{H(f)}&\Sigma F(X)\ar[d]^-{\Sigma F(f)}\\
F(Y)\ar[r]^-{f(Y)}\ar[d]_-{F(i)}&G(Y)\ar[r]^-{i(Y)}\ar[d]^-{G(i)}&H(Y)\ar[r]^-{q(Y)}\ar[d]^-{H(q)}&\Sigma F(Y)\ar[d]^-{\Sigma F(q)}\\
F(Z)\ar[r]^-{f(Z)}\ar[d]_-{F(q)}&G(Z)\ar[r]^-{i(Z)}\ar[d]^-{G(q)}&H(Z)\ar[r]^-{q(Z)}\ar[d]^-{H(q)}\ar@{}[rd]|*+<.15cm>[o][F-]{\scriptscriptstyle -1}&\Sigma F(Z)\ar[d]^{\Sigma F(q)}\\
\Sigma F(X)\ar[r]_-{\Sigma f(X)}&\Sigma G(X)\ar[r]_-{\Sigma i(X)}&\Sigma H(X)\ar[r]_-{\Sigma q(X)}&\Sigma^2 F(X)}
\end{equation}
is $\filledsquare$-coherent. Then 
\begin{align*}
K_{i}(^{\filledsquare}F)+K_{i}(^{\filledsquare}H)&=K_{i}(^{\filledsquare}G)\colon K_{i}(^{\filledsquare}\C{T})\To  K_{i}(^{\filledsquare}\C{T}'),\qquad i=0,1.
\end{align*}
\end{thm}

\begin{proof}
For $K_{0}$ the result follows from the following equation,
\begin{align*}
\partial[F(X)\st{f(X)}\To  G(X)\st{i(X)}\To  H(X)\st{q(X)}\To \Sigma F(X)]&=-[G(X)]+[H(X)]+[F(X)].
\end{align*}

Any element in $K_{1}$ is represented by a pair of $\filledsquare$-triangles, $$[X\mathop{\rightrightarrows}\limits^{f}_{f'}Y\mathop{\rightrightarrows}\limits^{i}_{i'}Z \mathop{\rightrightarrows}\limits^{q}_{q'}\Sigma X],$$ see Theorem \ref{otromaintriang}. If we apply Proposition  \ref{relaciontri} to the induced $\filledsquare$-coherent $3\times 3$ diagrams \eqref{paraad2} we obtain the following equation in $K_{1}$,
\begin{align*}
0={}&[F(X)\!\!\mathop{\rightrightarrows}\limits^{f(X)}_{f(X)}\!\!G(X)\!\!\mathop{\rightrightarrows}\limits^{i(X)}_{i(X)}H(X)\!\! \mathop{\rightrightarrows}\limits^{q(X)}_{q(X)}\!\!\Sigma F(X)]-[F(Y)\!\!\mathop{\rightrightarrows}\limits^{f(Y)}_{f(Y)}\!\!G(Y)\!\!\mathop{\rightrightarrows}\limits^{i(Y)}_{i(Y)}\!\!H(Y) \!\!\mathop{\rightrightarrows}\limits^{q(Y)}_{q(Y)}\!\!\Sigma F(Y)]\\
&+[F(Z)\!\!\mathop{\rightrightarrows}\limits^{f(Z)}_{f(Z)}\!\!G(Z)\!\!\mathop{\rightrightarrows}\limits^{i(Z)}_{i(Z)}\!\!H(Z) \!\!\mathop{\rightrightarrows}\limits^{q(Z)}_{q(Z)}\!\!\Sigma F(Z)]\\
={}&[F(X)\!\!\mathop{\rightrightarrows}\limits^{F(f)}_{G(f')}\!\!F(Y)\!\!\mathop{\rightrightarrows}\limits^{F(i)}_{F(i')}\!\!F(Z) \!\!\mathop{\rightrightarrows}\limits^{F(q)}_{F(q')}\!\!\Sigma F(X)]-[G(X)\!\!\mathop{\rightrightarrows}\limits^{G(f)}_{G(f')}\!\!G(Y)\!\!\mathop{\rightrightarrows}\limits^{G(i)}_{G(i')}\!\!G(Z) \!\!\mathop{\rightrightarrows}\limits^{G(q)}_{G(q')}\!\!\Sigma G(X)]\\
&+[H(X)\!\!\mathop{\rightrightarrows}\limits^{H(f)}_{G(f')}\!\!H(Y)\!\!\mathop{\rightrightarrows}\limits^{H(i)}_{H(i')}\!\!H(Z) \!\!\mathop{\rightrightarrows}\limits^{H(q)}_{H(q')}\!\!\Sigma H(X)],
\end{align*}
hence we are done.
\end{proof}

Notice that this additivity theorem does not contradict \cite[Remark 2.3]{kttc}.

\begin{rem}
The hypotheses are satisfied if the natural \exact\ triangle has a model, e.g.~if $\C{T}$ and $\C{T}'$ are categories of perfect complexes over two rings $R$ and $R'$, and the exact functors $F$, $G$ and $H$ are given by the derived tensor product with perfect complexes of $R'$-$R$-bimodules $F_{*}$, $G_{*}$ and $H_{*}$  fitting into a $\filledsquare$-triangle $F_{*}\r G_{*}\r H_{*}\r \Sigma F_{*}$.
\end{rem}

\subsection{Low-dimensional $K$-theory of a triangulated category with a $t$-struc\-ture}\label{heart} 
\label{gradun}\label{forgrad}

A \emph{$t$-structure} in a triangulated category $\C T$ consists essentially of a full abelian subcategory $\C A\subset \C T$, called \emph{core} or \emph{heart}, and  a splitting of the inclusion given by a cohomological functor
$$H^{0}\colon\C{T}\To \C{A}.$$
Some axioms modelled on the canonical exampe $\C T = D^{b}(\C A)$ must be satisfied, see \cite[IV.4.2 and IV.4.11]{mha}. The \emph{cohomology} of an object $X$ in $\C T$ is defined as
$$H^{n}X=H^{0}\Sigma^{n}X,\quad n\in\mathbb{Z}.$$

In this section we consider $t$-structures satisfying the following additional properties, which hold in the canonical example but which are often relaxed. Our  $t$-structures will be \emph{non-degenerate}, i.e.~$H^{*}X=0$ iff $X=0$, and  \emph{bounded}, i.e.~the graded object $H^{*}X$ must be bounded for any $X$.

Any object $X$ in $\C T$ fits into a functorial \exact\ triangle,
$$\Delta_{X}\colon\quad X_{\leq -1}\To  X\To  X_{\geq 0}\To  \Sigma X_{\leq -1},$$
where, as the notation suggests, the homology of $X_{\leq -1}$ (resp. $X_{\geq 0}$) is concentrated in negative (resp. non-negative) degrees. In particular, the first (resp. second) arrow induces an isomorphism in $H^{n}$ for all $n<0$ (resp. $n\geq 0$). 

An object $X$ in $\C T$ is \emph{connective} if $H^{n}X=0$ for $n>0$. For a connective object $X$, the \exact\ triangle $\Delta_{X}$ looks like
$$\Delta_{X}\colon \quad X_{\leq -1}\To  X\To  H^{0}X\To  \Sigma X_{\leq -1}.$$
Since the our $t$-structures are bounded, all objects become connective after suspending a finite number of times. 

A short exact sequence in $\C A$,
$$\Delta\colon\; A\st{i}\into B\st{q}\onto C,$$
extends uniquely to a \exact\ triangle, that we denote in the same way,
$$\Delta\colon\; A\st{i}\To  B\st{q}\To  C\st{\epsilon}\To \Sigma A.$$
This defines a natural isomorphism
$$\ext_{\C A}^1(C,A)\cong\C{T}(C,\Sigma A),$$
and moreover, a morphism of stable quadratic modules, $\filledsquare= b,d,v$,
$$j\colon \D{D}_{*}(\C A)\To  \D{D}_{*}({}^\filledsquare\C{T}).$$

\begin{lem}\label{inyectivo}
The stable quadratic module morphism $j$ is levelwise injective.
\end{lem}

\begin{proof}
The morphism is obviously injective in degree $0$, since it is given by an inclusion of the basis of $\D{D}_{0}(\C A)$ (the non-trivial objects of $\C A$) into the basis of $\D{D}_{0}({}^\filledsquare\C{T})$ (the non-trivial objects of $\C T$). An easy diagram chase shows that $j$ is  injective in degree $1$ if and only if $\pi_{1}j$ is injective. The latter is true. Actually, the comparison morphism $K_{n}(\C A)\r K_{n}(^{\filledsquare}\C T)$ is split injective for any $n\geq 0$, $\filledsquare =d,v$, see \cite[Theorem 50 (i)]{kttcneeman},  and $\pi_{1}j$ is the case $n=1$, see Theorem \ref{pepillo}. The case $\filledsquare =b$ then follows from the factorization \eqref{dv}. Alternatively, the cases $\filledsquare= b,d$ also follow from \cite[Corollary 5.3]{dftc}. It is also possible to give a unified direct proof, not invoking Neeman's or Breuning's results, see Remark \ref{notinv} below.
\end{proof}

For any object $X$ in $\C T$, consider the \exact\ triangle
$$\Gamma_{X}\colon\quad X\r 0\r \Sigma X\st{1}\r\Sigma X.$$

\begin{prop}\label{bastan}
The stable quadratic module $\D{D}_{*}({}^\filledsquare\C{T})$, $\filledsquare= b,d,v$, is generated by the image of the generators $(G1)$ and $(G3)$ of $\D{D}_{*}(\C{A})$ by $j$ together with the generators $[\Gamma_{X}]$ and $[\Delta_{Y}]$, where $X$  runs over all objects of $\C T$ and $Y$ runs over all connective objects.
\end{prop}

This follows from Lemmas \ref{facil},  \ref{corte1} and \ref{too2} below, by induction on the interval where the cohomology of a given object  (or the cohomology of the three objects in a given $\filledsquare$-triangle) is concentrated.

The following lemma is simply an application of (R2).

\begin{lem}\label{facil}
Given objects $X$ and $Y$ in $\C T$, with $Y$ connective,  the following formulas hold in $\D{D}_{0}({}^\filledsquare\C{T})$, $\filledsquare= b,d,v$:
\begin{align*}
\partial[\Gamma_{X}]&=[\Sigma X]+[X],&
[Y]+\partial[\Delta_{Y}]&=j[H^{0}Y]+[Y_{\leq -1}]. 
\end{align*}
\end{lem}

\begin{lem}
Let $\Delta$ in \eqref{triangle} be a $\filledsquare$-triangle where $X$, $Y$ and $C^{f}$ are connective, $\filledsquare= b,d,v$. There is a  $\filledsquare$-octahedron
$$\qquad\begin{array}{c}
\octahedronles{objA=X,objB=Y^{t} 
,objC=Y ,objD=C^{f}_{\leq -1},objE=C^{f} ,objF=H^{0}C^{f},
            arAB=,arBC=,arAC=f,arDE=,arEF=, arDA=, arBD=, arCE=i, arEA=q, arCF=,arFB=
            }
            \end{array}$$
which is unique up to a unique isomorphism.
\end{lem}

\begin{proof}
Everything is proven in  \cite[Lemma 3.9 and Theorem 3.13]{vt}, except that the octahedron is special for $\filledsquare=d$. This follows from Remark \ref{speziale}. 
\end{proof}

The \exact\ triangle on the bottom of the octahedron in the previous lemma,
\begin{equation}\label{tronquito}
X\To  Y^{t}\To  C^{f}_{\leq -1}\To \Sigma X,
\end{equation}
is known as the \emph{truncation} of $\Delta$.

\begin{lem}\label{corte1}
In the conditions of the previous lemma,  the following formula holds in $\D{D}_{1}({}^\filledsquare\C{T})$, $\filledsquare= b,d,v$:
\begin{align*} 
&[\Delta_{Y}]+j[\ker H^{0}i^{f}\into H^{0}Y\onto H^{0}C^{f}]^{[Y_{\leq -1}]}+[X\r Y^{t}\r C^{f}_{\leq -1}\r\Sigma X]\\
&=[\Delta_{Y^{t}}]+[\Delta]+[\Delta_{C^{f}}]^{[X]}.
\end{align*}
\end{lem}

\begin{proof}
Applying (R8) to the $\filledsquare$-octahedron in the previous lemma we obtain
\begin{align*}
[Y^{t}\r Y\r H^{0}C^{f}\r\Sigma Y^{t}]+[X\r Y^{t}\r  C^{f}_{\geq 1}\r \Sigma X]&=[\Delta]+[\Delta_{C^{f}}]^{[X]}.
\end{align*}
By connectivity reasons, the morphism $Y_{\geq 1}\r Y$ factors uniquely through $Y^{t}\r Y$. The octahedral axiom yields an octahedron
$$\qquad\begin{array}{c}
\octahedronles{objA=Y_{\leq -1},objB=Y^{t} ,objC=Y ,objD=\ker H^{0}i^{f},objE=H^{0}Y ,objF=H^{0}C^{f},
            arEF=H^{0}i^{f}\;
            }
            \end{array}$$
whose byproduct is necessarily the \exact\ triangle associated to the short exact sequence $\ker H^{0}i^{f}\into H^{0}Y\onto H^{0}C^{f}$, since this short exact sequence is obtained by taking $H^{0}$ on $Y^{t}\r Y\r H^{0}C^{f}\r\Sigma Y^{t}$. This ochahedron is special by Remark \ref{speziale}. Hence, by (R8), 
\begin{align*}
[\Delta_{Y^{t}}]+[Y^{t}\r Y\r H^{0}C^{f}\r\Sigma Y^{t}]&=[\Delta_{Y}]+j[\ker H^{0}i^{f}\into H^{0}Y\onto H^{0}C^{f}]^{[Y_{\leq -1}]}.
\end{align*}
The statement follows from the combination of the two previous formulas.
\end{proof}

The \emph{translation} of a $\filledsquare$-triangle $\Delta$ as in \eqref{triangle} is the $\filledsquare$-triangle
\begin{equation}\label{trasladalo}
\Delta_{\tr}\colon\quad Y\st{i^{f}}\To  C^{f}\st{-q^{f}}\To  \Sigma X\st{\Sigma f}\To \Sigma Y.
\end{equation}
Usually, the sign is placed in the third arrow, but both triangles are isomorphic. Actually, we could also place the sign in the first arrow.

\begin{lem}\label{too2}
Given a $\filledsquare$-triangle $\Delta$ \eqref{triangle}, the following formula holds in $\D{D}_{1}({}^\filledsquare\C{T})$:
\begin{align*}
[\Delta_{\tr}]+[\Delta]={}&[\Gamma_{X}]+\grupo{[\cone{f}],[\Sigma X]}.
\end{align*}
\end{lem}

\begin{proof}
Consider the following $\filledsquare$-octahedron,
$$\octahedron{objA=X,objB=Y,objC=\cone{f} ,objD=\cone{f},objE=\cone{f}\oplus\Sigma X,objF=\Sigma X,
            arAB=f,arBC=i ,arAC=0,arDE=\binom{1}{0},arEF={\scriptscriptstyle (0,1)\;},
            arBD=i ,arDA=q ,arCF=-q ,arFD=0,arFB=\;\Sigma f,arEA={\scriptscriptstyle (q ,1)\;},arCE=\;\binom{\;\;\,1}{-q}}
$$
This octahedron is indeed special for $\filledsquare=d$, since we have the following isomorphisms where the lower rows are (direct sums of) \exact\ triangles, 
$$\xymatrix{
Y\ar[r]^-{\binom{\;\;\,i}{-i}}\ar@{=}[d]& 
\cone{f}\oplus\cone{f}\ar[r]^{
\bigl( \begin{smallmatrix}
\;\;\,1&1\\ -q&0
\end{smallmatrix} \bigr)
}
\ar[d]_{
\bigl( \begin{smallmatrix}
1&1\\ 1&0
\end{smallmatrix} \bigr)
}^{\cong}&
\cone{f}\oplus\Sigma X\ar[r]^-{\scriptscriptstyle (0,\Sigma f)}\ar@{=}[d]&\Sigma Y\ar@{=}[d]\\
0\oplus Y\ar[r]_-{{0}\oplus{i}}& \cone{f}\oplus\cone{f}\ar[r]_{
1\oplus(-q)
}&
\cone{f}\oplus\Sigma X\ar[r]_-{\scriptscriptstyle 0\oplus\Sigma f}&\Sigma (0\oplus Y)}$$

$$\xymatrix{
\cone{f}\oplus\Sigma X\ar[r]^-{
\bigl( \begin{smallmatrix}
q&\;\;\,1\\ 0&-1
\end{smallmatrix} \bigr)
}\ar@{=}[d]&
\Sigma X\oplus\Sigma X\ar[r]
\ar[r]^-{\scriptscriptstyle
({\Sigma f},{\Sigma f})
}\ar[d]_-{\bigl( \begin{smallmatrix}
1&\;\;\,1\\ 0&-1
\end{smallmatrix} \bigr)}^{\cong}&
\Sigma Y\ar[r]^-{\binom{\Sigma i}{0}}\ar@{=}[d]& 
\Sigma \cone{f}\oplus\Sigma^{2} X\ar@{=}[d]\\
\cone{f}\oplus\Sigma X\ar[r]_-{
q\oplus1
}&
\Sigma X\oplus\Sigma X\ar[r]
\ar[r]_-{\scriptscriptstyle
{\Sigma f}\oplus0
}&
\Sigma (Y\oplus 0)\ar[r]_-{{\Sigma i}\oplus{0}}& 
\Sigma \cone{f}\oplus\Sigma^{2} X}$$
Relation (R8) yields the following equation,
\begin{align}\label{primen}
[\Delta_{\tr}]+[\Delta]={}&[X\st{0}\r C^{f}\st{\binom{1}{-q}}\To  C^{f}\oplus\Sigma X\st{(q,1)}\To \Sigma X]\\
\nonumber &+[C^{f}\st{\binom{1}{0}}\To  C^{f}\oplus\Sigma X\st{(0,1)}\To  \Sigma X\st{0}\r\Sigma C^{f}].
\end{align}

Applying (R7) to the isomorphisms
$$\begin{array}{c}\xymatrix{X\ar[r]^-{0}\ar@{=}[d]&\cone{f}\ar[r]^-{\binom{\;\;\,1}{-q}}\ar@{=}[d]&\cone{f}\oplus\Sigma X\ar[r]^-{\scriptscriptstyle (q ,1)}\ar[d]^-{
\bigl( \begin{smallmatrix}
1&0\\ q&1
\end{smallmatrix} \bigr)
}_{\cong}&\Sigma X\ar@{=}[d]\\
X\ar[r]^-{0}&\cone{f}\ar[r]^-{\binom{1}{0}}&\cone{f}\oplus\Sigma X\ar[r]^-{\scriptscriptstyle (0 ,1)}&\Sigma X}\end{array}$$
$$\begin{array}{c}\xymatrix{\Sigma X\ar[r]^-{\binom{0}{1}}\ar@{=}[d]&\cone{f}\oplus\Sigma X\ar[r]^-{\scriptscriptstyle (1 ,0)}\ar[d]^-{\bigl( \begin{smallmatrix}
1&0\\ q&1
\end{smallmatrix} \bigr)}_{\cong}&\cone{f}\ar[r]^-{0}\ar@{=}[d]&\Sigma^{2} X\ar@{=}[d]\\
\Sigma X\ar[r]^-{\binom{0}{1}}&\cone{f}\oplus\Sigma X\ar[r]^-{\scriptscriptstyle (1 ,0)}&\cone{f}\ar[r]^-{0}&\Sigma^{2} X}
\end{array}$$
we obtain,
\begin{align*}
[X\st{0}\r C^{f}\st{\binom{1}{0}}\To  C^{f}\oplus\Sigma X\st{(0,1)}\To \Sigma X]+
\left[\left(\begin{smallmatrix}
1&0\\q&1
\end{smallmatrix}\right)
\right]&=[X\st{0}\r C^{f}\st{\binom{1}{-q}}\To  C^{f}\oplus\Sigma X\st{(q,1)}\To \Sigma X],\\
\left[\left(\begin{smallmatrix}
1&0\\q&1
\end{smallmatrix}\right)
\right]+[\Sigma X\st{\binom{0}{1}}\r C^{f}\oplus\Sigma X\st{(1,0)}\To  C^{f}\st{0}\r\Sigma X]
&=[\Sigma X\st{\binom{0}{1}}\r C^{f}\oplus\Sigma X\st{(1,0)}\To  C^{f}\st{0}\r\Sigma X].
\end{align*}
In particular,
\begin{align}
\nonumber\left[\left(\begin{smallmatrix}
1&0\\q&1
\end{smallmatrix}\right)
\right]
&=0,\\
\label{segun}[X\st{0}\r C^{f}\st{\binom{1}{0}}\To  C^{f}\oplus\Sigma X\st{(0,1)}\To \Sigma X]&
=[X\st{0}\r C^{f}\st{\binom{1}{-q}}\To  C^{f}\oplus\Sigma X\st{(q,1)}\To \Sigma X].
\end{align}

Applying (R8) to the special octahedron
$$\octahedron{objA=X,objB=0,objC=\cone{f} ,objD=\Sigma X,objE=\cone{f}\oplus\Sigma X,objF=C^{f},
            arAB=,arBC= ,arAC=,arDE=\binom{0}{1},arEF={\scriptscriptstyle (1,0)\;},
            arBD= ,arDA=1 ,arCF=1 ,arFD=,arFB=,arEA={\scriptscriptstyle (0 ,1)\;},arCE=\;\binom{1}{0}}
$$
we obtain
\begin{align}\label{terce}
[\Gamma_{X}]&=[X\st{0}\r C^{f}\st{\binom{1}{0}}\To  C^{f}\oplus\Sigma X\st{(0,1)}\To \Sigma X]
+[\Sigma X\st{\binom{0}{1}}\r C^{f}\oplus\Sigma X\st{(1,0)}\To  C^{f}\st{0}\r\Sigma X].
\end{align}

Combining \eqref{primen}, \eqref{segun} and \eqref{terce}, and using (R9),
\begin{align*}
[\Delta_{\tr}]+[\Delta]={}&[\Gamma_{X}]-[\Sigma X\st{\binom{0}{1}}\r C^{f}\oplus\Sigma X\st{(1,0)}\To  C^{f}\st{0}\r\Sigma X]\\
&+[C^{f}\st{\binom{1}{0}}\To  C^{f}\oplus\Sigma X\st{(0,1)}\To  \Sigma X\st{0}\r\Sigma C^{f}]\\
={}&[\Gamma_{X}]+\grupo{[C^{f}],[\Sigma X]}.
\end{align*}
\end{proof}

\begin{thm}\label{todo}
The  inclusion $ j$ fits into a unique strong deformation retraction,
$$\EZDIAG{\D{D}_{*}(^{\filledsquare}\C{T})}{{\D{D}}_{*}(\C{A}),}{  p}{ j}{\alpha}$$
where $\alpha \colon  j {p}\Rightarrow 1$ satisfies the following equations, $\filledsquare= b,d,v$. Given an arbitrary object $X$ in $\C T$ and a connective object $Y$ in $\C T$,
\begin{align}
\label{todo.1}0&=[\Gamma_X]+\alpha([\Sigma X])^{[X]}+\alpha([X]),\\
\label{todo.2}\alpha([Y])&=[\Delta_{Y}]
+\alpha([Y_{\leq -1}]).
\end{align}
\end{thm}

\begin{proof}
We inductively define $\alpha$ on generators $[X]$.  Let $X$ be a non-trivial object in $\C T$. Since the $t$-structure is non-degenerate, $H^{*}X\neq 0$. Let $n\in\mathbb{Z}$ and $ m\in\mathbb{Z}$ be the minimum and maximum integers such that $H^{n}X\neq 0\neq H^{m}X$, respectively, i.e.~$[n,m]$ is the smallest interval where the homology of $X$ is concentrated. We now define $p[X]$ by induction on $(m-n,\abs{n})\in\mathbb{N}^{2}$ with respect to the lexicographic order. If $m-n=0$ and $\abs{n}=0$ then  $X=A$ in $\C A$, and we must define 
\begin{align}\label{pobjeto1}
\alpha([A])&=0
\end{align}
so that $\alpha j=0$ holds. 
There are two kinds of induction steps: $(x,y)\r(x,y+1)$ and $(x,y)\forall y\in\mathbb N\r (x+1,0)$. In the first case we define $\alpha([X])$  so that \eqref{todo.1} is satisfied,
\begin{align}\label{pobjeto2}
\alpha([X])&=\left\{
\begin{array}{ll}
-[\Gamma_{\Sigma^{-1}X}]^{-[\Sigma^{-1}X]}-\alpha([\Sigma^{-1}X])^{-[\Sigma^{-1}X]},&\text{if }n>0,\\
-\alpha([\Sigma X])^{[X]}-[\Gamma_{X}],&\text{if }n<0.
\end{array}
\right.
\end{align}
In the second case, $X=Y$ is connective, and we define $\alpha([Y])$ by \eqref{todo.2}.

We will now follow  Lemmas \ref{abouthomo} and \ref{cacharro}. We first check that $f=1_{\D{D}_{*}(^{\filledsquare}\C{T})}+\alpha$ factors through $j$, which is essentially an inclusion by Lemma \ref{inyectivo}. It is enough to prove that $f$ applied to any of the generators in Proposition \ref{bastan} lies in the image of $j$. Given $A$ in $\C A$, a short exact sequence $A\into B\onto C$ in $\C A$, and two objects $X$ and $Y$ in $\C T$, with $Y$ connective,
\begin{align*}
f [A]&= [A]+\partial\alpha  [A]=[A]+\partial(0)= [A],\\
f [A\into B\onto C]&=[A\into B\onto C]+\alpha\partial  [A\into B\onto C]\\
&=[A\into B\onto C]+\alpha(- [B]+[C]+[A])\\
&=[A\into B\onto C]-\alpha([B])^{- [B]+[C]+[A]}+\alpha([C])^{[A]}+\alpha[A]\\
\scriptstyle\eqref{pobjeto1}\qquad&=[A\into B\onto C],\\
f[\Gamma_{X}]&=[\Gamma_{X}]+\alpha\partial [\Gamma_{X}]\\
&=[\Gamma_{X}]+\alpha([\Sigma X]+[X])\\
&=[\Gamma_{X}]+\alpha([\Sigma X])^{[X]}+\alpha[X]\\
\scriptstyle\eqref{todo.1}\qquad&=0,\\
f[\Delta_{Y}]&=[\Delta_{Y}]+\alpha\partial [\Delta_{Y}]\\
&=[\Delta_{Y}]+\alpha(-[Y]+[H^{0}Y]+[Y_{\leq -1}])\\
&=[\Delta_{Y}]-\alpha([Y])^{\partial [\Delta_{Y}]}+\alpha([H^{0}Y])^{[Y_{\leq -1}]}+\alpha([Y_{\leq -1}])\\
\scriptstyle\eqref{pobjeto1}\qquad&=-\alpha([Y])+[\Delta_{Y}]+\alpha([Y_{\leq -1}])\\
\scriptstyle\eqref{todo.2}\qquad&=0.
\end{align*}
The first two equations also prove that $fj=j$.

The homotopy $\alpha$ has been defined so that $\alpha[A]=0$ for any object $A$ in $\C A$. Hence, it is only left to prove that
\begin{align}\label{finito}
\alpha([X]+\partial\alpha[X])&=0.
\end{align}
In order to check this equation, we follow the same induction pattern as for the definition of $\alpha[X]$.  If $X$ is in $\C A$ then \eqref{finito} follows immediately from \eqref{pobjeto1}. Equation \eqref{finito}  for $[X]$ is equivalent to equation \eqref{finito}  for $[\Sigma X]$. Indeed, applying $\partial$ to \eqref{todo.1} we obtain
\begin{align*}
0&=[\Sigma X]+\partial\alpha[\Sigma X]+[X]+\partial\alpha[X].
\end{align*}
Hence
\begin{align*}
0&=\alpha(0)=\alpha([\Sigma X]+\partial\alpha[\Sigma X])^{[X]+\partial\alpha[X].}+\alpha([X]+\partial\alpha[X]).
\end{align*}
Finally, if $Y$ is connective then \eqref{finito} for $[Y]$ is equivalent to \eqref{finito} for $[Y_{\leq -1}]$. Actually, applying $\partial$ to \eqref{todo.2} we derive
\begin{align*}
\partial\alpha[Y]&=-[Y]+[H^{0}Y]+[Y_{\leq -1}]+\partial\alpha[Y_{\leq -1}].
\end{align*}
Therefore, 
\begin{align*}
\alpha([Y]+\partial\alpha(Y))&=\alpha([H^{0}Y]+[Y_{\leq -1}]+\partial\alpha[Y_{\leq -1}])\\
&=\alpha([H^{0}Y])^{[Y_{\leq -1}]+\partial\alpha[Y_{\leq -1}]}+\alpha([Y_{\leq -1}]+\partial\alpha[Y_{\leq -1}])\\
\scriptstyle\eqref{pobjeto1}\qquad&=\alpha([Y_{\leq -1}]+\partial\alpha[Y_{\leq -1}]).
\end{align*}

\end{proof}

\begin{cor}
The comparison maps in Theorem \ref{pepillo} induce isomorphisms
$$K_{i}(\C A)\cong K_{i}(^{d}\C T)\cong K_{i}(^{b}\C T)\cong K_{i}(^{v}\C T),\qquad i=0,1.$$
\end{cor}

\begin{defn}
For $\filledsquare= b,d,v$, we define
$\D{D}^t_{*}({}^\filledsquare\C{T})$ as the quotient of $\D{D}_{*}({}^\filledsquare\C{T})$ by the following relations:
\begin{itemize}
\item[(R11)] $[\Delta_{Y}]=0$ for any connective object $Y$.
\item[(R12)] $[\Gamma_{X}]=0$ for any object $X$.
\end{itemize}
\end{defn}

\begin{prop}
The composite of the morphism $j$ with the natural projection is an isomorphism, $\filledsquare= b,d,v$, $$
\bar\jmath\colon \D{D}_{*}(\C{A})\st{\cong}\To \D{D}^t_{*}({}^\filledsquare\C{T}).$$
\end{prop}

\begin{proof}
We use the notation in Theorem \ref{todo}. We have seen in its proof that the morphism $p$ satisfies $p[\Delta_{Y}]=0=p[\Gamma_{Y}]$, hence it factors through $\D{D}_{*}({}^\filledsquare\C{T})$. Denote $\bar p$ the (unique) factorization. We are going to check that 
$$\xymatrix{\D{D}^t_{*}({}^\filledsquare\C{T})\ar@<.5ex>[r]^{\bar p}&\D{D}_{*}(\C{A})\ar@<.5ex>[l]^{\bar \jmath}}$$
are mutually inverse isomorphisms. The equation $\bar p\bar \jmath=1$ follows from $pj=1$. In order to check that $\bar\jmath\bar p=1$ it is enough to notice that the image of $\alpha$ vanishes in $\D{D}^t_{1}({}^\filledsquare\C{T})$. This follows easily from \eqref{todo.1} and \eqref{todo.2} by the induction procedure applied twice in the previous proof.
\end{proof}

\begin{rem}\label{notinv}
This proposition admits a direct proof. Indeed, the equations in the proof of Theorem \ref{todo} yield an inductive formula for $p$. Using this formula one can easily check that $\bar p\bar\jmath=1$. The proof of $\bar\jmath\bar p=1$ is straightforward, but rather tedious. A previous version of this paper contained a sketch. We invite the reader interested in practicing with the algebra of stable quadratic modules to reconstruct it as an exercise.

The advantage of a direct proof is that we would obtain Lemma \ref{inyectivo} as an immediate corollary, avoiding invoking Neeman's result, whose proof is long and complicated.
\end{rem}

\subsection{The $K$-theory of some unusual triangulated categories}\label{unusual}

Let $R$ be a commutative local ring with maximal ideal $(\eps)\neq 0$ such that $\eps^2=0$ and with  residue field $k=R/(\eps)$ of characteristic $2$. This ring is quasi-Frobenius.
Notice that  either $\eps=2$ or $R=k[\eps]/\eps^{2}$. Recall from \cite{tcwm} that the category $\mathcal{F}(R)$ of finitely generated free $R$-modules admits a unique  triangulated structure with identity suspension functor $\Sigma =1_{\mathcal{F}(R)}$ such that  the following triangle is \exact,
$$\Delta_{\eps}\colon\quad R\st{\eps}\To  R\st{\eps}\To  R\st{\eps}\To  R.$$
This triangulated category does not admit models if $\eps=2$. Otherwise it is the compact derived category of  a certain differential graded algebra, in particular it can be described as the homotopy category of a Waldhausen category.

\begin{thm}\label{raro}
Neeman's and Breuning's $K$-theories of the triangulated category $\mathcal{F}(R)$ satisfy:
\begin{align*}
K_{0}(^{b}\mathcal{F}(R))\cong K_{0}(^{d}\mathcal{F}(R))\cong K_{0}(^{v}\mathcal{F}(R))&\cong0,& 
K_{1}(^{b}\mathcal{F}(R))\cong K_{1}(^{d}\mathcal{F}(R))&\cong 0.
\end{align*}
Moreover, there is a surjective homomorphism $K_{1}(^{v}\mathcal{F}(R))\onto  k^\times/(k^\times)^2$.
\end{thm}

Notice that $k^\times/(k^\times)^2\neq 0$ as long as $k$ is non-perfect, thus we obtain examples of triangulated categories $\C{T}$ such that $K_{1}(^{b}\!\C{T})$ and $K_{1}(^{d}\!\C{T})$  are not isomorphic to $K_{1}(^{v}\!\C{T})$. 

An acyclic $3$-periodic complex  in $\mathcal{F}(R)$,
$$\Delta\colon\quad X_{0}\st{d_{2}}\To X_{2}\st{d_{1}}\To X_{1}\st{d_{0}}\To X_{0},$$
fits into a natural short exact sequence of complexes,
$$\eps\cdot \Delta\into \Delta\st{\eps}\onto \eps\cdot \Delta,$$
which induces  isomorphisms in homology,
$$\sigma_{n}^{\Delta}\colon H_{n+1}(\eps\cdot \Delta)\cong H_{n}(\eps\cdot \Delta),\quad n\in\mathbb{Z}/3.$$
A \exact\ triangle  in $\mathcal{F}(R)$ is the same as  an acyclic $3$-periodic chain complex~$\Delta$ such that the automorphism
$$\rho^{\Delta}_{n}=\sigma_{n}^{\Delta}\sigma_{n+1}^{\Delta}\sigma_{n+2}^{\Delta}\colon H_{n}(\eps\cdot \Delta)\cong H_{n}(\eps\cdot \Delta)$$
is the identity for some, and hence all, $n\in\ZZ/3$, see \cite[Remark 7]{tcwm}. 

\begin{defn}
We define the \emph{determinant} of an acyclic $3$-periodic complex $\Delta$ in $\mathcal{F}(R)$ as
$\det(\Delta)=\det(\rho_{n}^{\Delta})\in k^{\times}$, which is independent of $n\in\ZZ/3$.
\end{defn}

The determinant is clearly invariant under shifts of the complex $\Delta$ and isomorphisms. Notice that the determinant of a \exact\ triangle is $1\in k^{\times}$. 

\begin{lem}\label{sesi}
Given a short exact sequence of acyclic $3$-periodic complexes $\Delta'\into \Delta\onto \Delta''$ in $\mathcal{F}(R)$ we have
$\det(\Delta)=\det(\Delta')\det(\Delta'')$ mod $(k^\times)^2$.
\end{lem}

\begin{proof}
The short exact sequence in the statement splits in each degree, so we have a short exact sequence $\eps\cdot \Delta'\into \eps\cdot  \Delta\onto \eps\cdot \Delta''$ which induces a long exact sequence in homology,
$$\cdots \r H_n(\eps\cdot \Delta')\To H_n(\eps\cdot  \Delta)\To H_n(\eps\cdot \Delta'')\To H_{n-1}(\eps\cdot \Delta')\r\cdots.$$
Moreover, the following $3\times 3$ diagram of short exact sequences of complexes
$$\xymatrix{
\eps\cdot \Delta'\ar@{>->}[r]\ar@{>->}[d]&\eps\cdot \Delta\ar@{->>}[r]\ar@{>->}[d]&\eps\cdot \Delta''\ar@{>->}[d]\\
\Delta'\ar@{>->}[r]\ar@{->>}[d]&\Delta\ar@{->>}[r]\ar@{->>}[d]&\Delta''\ar@{->>}[d]\\
\eps\cdot \Delta'\ar@{>->}[r]&\eps\cdot \Delta\ar@{->>}[r]&\eps\cdot \Delta''}
$$
shows that the following diagram is commutative, since we are in characteristic $2$, 
$$\xymatrix{H_n(\eps\cdot \Delta')\ar[r]\ar[d]^{\sigma_n^{\Delta'}}& H_n(\eps\cdot  \Delta)\ar[r]\ar[d]^{\sigma_n^{\Delta}}& H_n(\eps\cdot \Delta'')\ar[r]\ar[d]^{\sigma_n^{\Delta''}}&H_{n-1}(\eps\cdot \Delta')\ar[d]^{\sigma_{n-1}^{\Delta'}}\\
H_{n-1}(\eps\cdot \Delta')\ar[r]& H_{n-1}(\eps\cdot  \Delta)\ar[r]& H_{n-1}(\eps\cdot \Delta'')\ar[r]&H_{n-2}(\eps\cdot \Delta')}$$
Therefore we have an automorphism of a $9$-periodic long exact sequence, $n\in\ZZ/3$,
$$\xymatrix@C=25pt{H_n(\eps\cdot \Delta')\ar[r]^\phi\ar[d]^-{\rho_n^{\Delta'}}& H_n(\eps\cdot  \Delta)\ar[r]\ar[d]^{\rho_n^{\Delta}}& H_n(\eps\cdot \Delta'')\ar[r]^{}\ar[d]^{\rho_n^{\Delta''}}&H_{n-1}(\eps\cdot \Delta')\ar[d]^{\rho_{n-1}^{\Delta'}}\\
H_{n}(\eps\cdot \Delta')\ar[r]^-{\phi}& H_{n}(\eps\cdot  \Delta)\ar[r]& H_{n}(\eps\cdot \Delta'')\ar[r]^{}&H_{n-1}(\eps\cdot \Delta')}$$
Using the multiplicative property of determinants with respect to automorphisms of short exact sequences, if $\rho'\colon\ker\phi\cong \ker\phi$ is the automorphism induced by $\rho_n^{\Delta'}$ we get
\begin{align*}
\det(\rho')^2={}&\det(\rho_n^{\Delta'})\det(\rho_n^{\Delta})^{-1}\det(\rho_n^{\Delta''})\\
&\det(\rho_{n-1}^{\Delta'})^{-1}\det(\rho_{n-1}^{\Delta})\det(\rho_{n-1}^{\Delta''})^{-1}\\
&\det(\rho_{n-2}^{\Delta'})\det(\rho_{n-2}^{\Delta})^{-1}\det(\rho_{n-2}^{\Delta''}). 
\end{align*}
Since these determinants are independent of $n\in\ZZ/3$ we deduce, as desired,  that $$\det(\rho')^2\det(\rho_n^\Delta)=\det(\rho_n^{\Delta'})\det(\rho_n^{\Delta''}).$$

\end{proof}

\begin{lem}\label{virtualson}\label{son}
A virtual triangle in $\mathcal{F}(R)$ is the same as an  acyclic $3$-periodic complex $\Delta$. Moreover, $\Delta$ is the direct sum of a contractible  triangle and 
$$R^{d}\st{\eps}\To  R^{d}\st{\eps}\To  R^{d}\st{\eps\cdot\bar{\rho}}\To  R^{d},$$
where $d=\dim_{k}H_{n}(\eps\cdot \Delta)$ and $\bar{\rho}$ is any automorphism of $R^{d}$ with $\bar{\rho}\otimes_{R}k=\rho^{\Delta}_{n}$ for some basis of $H_{n}(\eps\cdot \Delta)$, $n\in\ZZ/3$.
\end{lem}

\begin{proof}
It is clear that  a virtual triangle is acyclic. Consider an acyclic $3$-periodic complex in $\mathcal{F}(R)$,
$$\Delta\colon\quad X_{0}\st{d_2}\To X_{2}\st{d_1}\To X_{1}\st{d_0}\To X_{0}.$$
Let $X_{n}'\subset \ker d_{n}\subset X_{n}$ be an injective envelope of $\eps\cdot \ker d_{n}$. 
Since $\Delta$ is acyclic, we can factor this inclusion as 
$$X_{n}'\r X_{n+1}\st{d_{n}}\To X_{n}.$$ This allows us to split $\Delta=\Delta'\oplus \Delta''$ as the direct sum of a contractible factor $\Delta'$ and a second factor $\Delta''$,
$$\begin{array}{c}
\Delta'\colon\quad X_{0}'\oplus X_{2}'\st{\binom{0\;1}{0\;0}}\To X_{2}'\oplus X_{1}'\st{\binom{0\;1}{0\;0}}\To X_{1}'\oplus X_{0}'\st{\binom{0\;1}{0\;0}}\To X_{0}'\oplus X_{2}',\vspace{-5pt}\\{}\\
\Delta''\colon\quad X_{0}''\st{d_{2}''}\To X_{2}''\st{d_{1}''}\To X_{1}''\st{d_{0}''}\To X_{0}'',
\end{array}$$
with $\im d_{n}''=\ker d_{n-1}''\subset \eps\cdot X_{n-1}$, so $d_{n}''=\eps\cdot\bar{d}_{n}$ for some $\bar{d}_{n}\colon X_{n+1}\r X_{n}$. One can easily check that
$\sigma_{n}^{\Delta}=\bar{d}_{n}\otimes_{R}k$, therefore $\bar{d}_{n}\otimes_{R}k$, and hence $\bar{d}_{n}$, is an isomorphism. Now the following isomorphism of $3$-periodic complexes proves the lemma
$$\xymatrix{
X''_0\ar[r]^{\eps}\ar[d]^{1}&X''_0\ar[r]^{\eps}\ar[d]^{\bar{d}_{2}}&X''_0\ar[rr]^{\eps\cdot \bar{d}_{0}\bar{d}_{1}\bar{d}_{2}}\ar[d]^{\bar{d}_{1}\bar{d}_{2}}&&X''_0\ar[d]^{1}\\
X''_0\ar[r]^{\eps\cdot \bar{d}_{2}}&X''_2\ar[r]^{\eps\cdot \bar{d}_{1}}&X''_1\ar[rr]^{\eps\cdot \bar{d}_{0}}&&X''_0
}$$
\end{proof}

\begin{lem}\label{kira}
Given a virtual octahedron in $\mathcal{F}(R)$
$$\octahedronada{objA=X,objB=Y,objC=Z,objD=\cone{f},objE=\cone{gf},objF=\cone{g},
            arAB=f,arBC=g,arAC=gf}$$
formed by virtual triangles $\Delta_{f}$, $\Delta_{g}$, $\Delta_{gf}$, $\widetilde{\Delta}$, the following formula holds, $$\det(\Delta_{g})\det(\Delta_{f})=\det(\Delta_{gf})\det(\widetilde{\Delta})\mod (k^\times)^2.$$
\end{lem}

\begin{proof}
The octahedron contains  morphisms of complexes,
$$\!\!
\xymatrix{
\Delta_{f}\colon\ar@<-1.5ex>[d]_{\varphi}\\
\Delta_{gf}\colon\\}
\xymatrix{
X\ar[r]^{f}\ar[d]^{1}&Y\ar[r]^{i^{f}}\ar[d]^{g}&\cone{f}\ar[d]^{\bar g}\ar[r]^{q^{f}}&X\ar[d]^{1}\\
X\ar[r]^{gf}&Z\ar[r]^{i^{gf}}&\cone{gf}\ar[r]^{q^{gf}}&X\\}
\quad\xymatrix{
\Delta_{g}\colon\ar@<-1.ex>[d]_{\psi}\\
\widetilde{\Delta}\colon\\}
\xymatrix{
Y\ar[r]^{g}\ar[d]^{i^{f}}&Z\ar[r]^{i^{g}}\ar[d]^{i^{gf}}&\cone{g}\ar[d]^{1}\ar[r]^{q^{g}}&Y\ar[d]^{g}\\
\cone{f}\ar[r]^{\bar g}&\cone{gf}\ar[r]^{\bar{f}}&\cone{g}\ar[r]^{i^{f}q^{g}}&\cone{f}\\}$$
The mapping cones of these morphisms fit in the middle of  well known short exact sequences of complexes involving the target and a translation of the source, hence by Lemma \ref{sesi},
\begin{align*}
\det(\operatorname{Cone}(\varphi))&=\det(\Delta_{gf})\det(\Delta_{f})\mod(k^{\times})^{2},\\
\det(\operatorname{Cone}(\psi))&=\det(\widetilde{\Delta})\det(\Delta_{g})\mod(k^{\times})^{2}.
\end{align*}
In this case they also fit into the following short exact sequences,
$$\xymatrix{\Gamma_{X}\colon\ar@<-1.5ex>@{>->}[d]\\\operatorname{Cone}(\varphi)\colon\ar@<-1.5ex>@{->>}[d]\\ \Delta'\;\colon}
\quad\xymatrix@C=40pt{
X\ar[r]\ar@{>->}[d]_-{\binom{1}{-f}}&0\ar[r]\ar@{>->}[d]&X\ar[r]^-{1}\ar@{>->}[d]^-{\binom{-i^{gf}}{1}}&X\ar@{>->}[d]^-{\binom{1}{-f}}\\
X\oplus Y\ar[r]^-{\bigl( \begin{smallmatrix}
gf &g\\ 0&-i^{f}
\end{smallmatrix} \bigr)}\ar@{->>}[d]_-{(f,1)}&
Z\oplus \cone{f}
\ar[r]^-{\bigl( \begin{smallmatrix}
i^{gf} &\bar g\\ 0&-q^{f}
\end{smallmatrix} \bigr)}\ar@{->>}[d]^-{1}&
\cone{gf}\oplus X\ar[r]^-{\quad\bigl( \begin{smallmatrix}
q^{gf} &1\\ 0&-f
\end{smallmatrix} \bigr)}\ar@{->>}[d]^-{(1,i^{gf})}&
X\oplus Y\ar@{->>}[d]^-{(f,1)}\\
 Y\ar[r]_-{\bigl( \begin{smallmatrix}
g\\ -i^{f}
\end{smallmatrix} \bigr)}&
Z\oplus \cone{f}
\ar[r]_-{(i^{gf} ,\bar g)}&
\cone{gf}\ar[r]_-{fq^{gf}}&
 Y}$$
 $$\xymatrix{\Gamma_{X}'\colon\ar@<-1.5ex>@{>->}[d]\\\operatorname{Cone}(\psi)\colon\ar@<-1.5ex>@{->>}[d]\\ \Delta''\colon}
\quad\xymatrix@C=40pt{
0\ar[r]\ar@{>->}[d]&X\ar[r]^-{1}\ar@{>->}[d]^-{\binom{-i^{gf}}{1}}&X\ar@{>->}[d]^-{\binom{1}{-f}}
\ar[r]&0\ar@{>->}[d]\\
Z\oplus \cone{f}
\ar[r]^-{\bigl( \begin{smallmatrix}
i^{gf} &\bar g\\ 0&-q^{f}
\end{smallmatrix} \bigr)}\ar@{->>}[d]^-{1}&
\cone{gf}\oplus X\ar[r]^-{\bigl( \begin{smallmatrix}
q^{gf} &1\\ 0&-f
\end{smallmatrix} \bigr)}\ar@{->>}[d]^-{(1,i^{gf})}&
X\oplus Y\ar@{->>}[d]^-{(f,1)}\ar[r]^-{\bigl( \begin{smallmatrix}
gf &g\\ 0&-i^{f}
\end{smallmatrix} \bigr)}&Z\oplus \cone{f}\ar@{->>}[d]^{1}\\
Z\oplus \cone{f}
\ar[r]_-{(i^{gf} ,\bar g)}&
\cone{gf}\ar[r]_-{fq^{gf}}&
 Y\ar[r]_-{\bigl( \begin{smallmatrix}
g\\ -i^{f}
\end{smallmatrix} \bigr)}&
Z\oplus \cone{f}}$$
Moreover, $\Delta''$ is the translation of $\Delta'$. Therefore, using again  Lemma \ref{sesi}, 
\begin{align*}
\det(\operatorname{Cone}(\varphi))&=\det(\Gamma_{X})\det(\Delta')=\det(\Delta')\mod(k^{\times})^{2},\\
\det(\operatorname{Cone}(\psi))&=\det(\Gamma_{X}')\det(\Delta'')=\det(\Delta')\mod(k^{\times})^{2},
\end{align*}
so we are done.
\end{proof}

\begin{proof}[Proof of Theorem \ref{raro}]
In this proof, $\filledsquare = b,d$. 
We can suppose  that the objects of $\mathcal{F}(R)$ are simply $R^{n}$, $n\geq 0$, hence Proposition \ref{kkk} applies. 
Generators (G2) vanish in $\D{D}_{*}^{+}(^{\filledsquare}\mathcal{F}(R))$. Indeed, given an isomorphism $h\colon R^{n}\cong R^{n}$, the  \exact\ triangle isomorphism
$$\xymatrix{
R^{n}\ar[r]^-{\eps}\ar[d]^{h}_{\cong}&R^{n}\ar[r]^-{\eps}\ar[d]^{h}_{\cong}&R^{n}\ar[r]^-{\eps}\ar[d]^{h}_{\cong}&R^{n}\ar[d]^{h}_{\cong}\\
R^{n}\ar[r]^-{\eps}&R^{n}\ar[r]^-{\eps}&R^{n}\ar[r]^-{\eps}&R^{n}
}$$
together with (R7) and (R10) yield
\begin{align*}
[h]+[h]^{n[R]}&=-n[\Delta_{\eps}]+[h]+n[\Delta_{\eps}]=-n[\Delta_{\eps}]+n[\Delta_{\eps}]+[h]^{\partial(n[\Delta_{\eps}])}=[h]^{n[R]},
\end{align*}
therefore $[h]=0$. In particular, the bracket vanishes in $\D{D}_{*}^{+}(^{\filledsquare}\mathcal{F}(R))$ by (R9$'$).
Moreover, $2[\Delta_{\eps}]=[\Gamma_{X}]$ by Lemma \ref{too2}. Furthermore, (R7) shows that two isomorphic \exact\ triangles represent the same (G3) generator. Any exact triangle in $\mathcal{F}(R)$ is a direct sum of a contractible triangle and copies of $\Delta_{\eps}$, hence $\D{D}_{*}^{+}(^{\filledsquare}\mathcal{F}(R))$ is generated by $[R]$ and $[\Delta_{\epsilon}]$, see Remark \ref{dosmas}. Let $\alpha\colon f\rr 0$ be the homotopy to the trivial endomorphism of $\D{D}_{*}^{+}(^{\filledsquare}\mathcal{F}(R))$ defined by Lemma \ref{abouthomo} and  $\alpha([R])=[\Delta_{\eps}]$. We have
\begin{align*}
f([R])={}&\partial\alpha([R])=\partial([\Delta_{\eps}])=-[R]+[R]+[R]=[R],\\
f([\Delta_{\eps}])={}&\alpha\partial([\Delta_{\eps}])=\alpha(-[R]+[R]+[R])=\alpha([R])=[\Delta_{\eps}].
\end{align*}
Thus $f$ is the identity, i.e. $\D{D}_{*}^{+}(^{\filledsquare}\mathcal{F}(R))$ is contractible. In particular $K_{i}(^{\filledsquare}\mathcal{F}(R))\cong\pi_{i}\D{D}_{*}^{+}(^{\filledsquare}\mathcal{F}(R))=0$, $i=0,1$. Recall also that $K_{0}(^{v}\mathcal{F}(R))=K_{0}(^{\filledsquare}\mathcal{F}(R))$.

Let us regard the abelian group $k^{\times}/(k^{\times})^2$ as a stable quadratic module concentrated in degree $1$. By Proposition \ref{menos} and Lemma \ref{kira} the determinant of virtual triangles defines a 
morphism $p\colon \D{D}_{*}(^{v} \mathcal{F}(R))\r k^{\times}/(k^{\times})^2$ with $p[\Delta]=\det(\Delta)$ for any virtual triangle $\Delta$. 
The induced morphism $\pi_1(p)\colon K_{1}(^{v} \mathcal{F}(R))\r k^{\times}/(k^{\times})^2$ 
is 
surjective since
\begin{align*}
p[R\mathop{\rightrightarrows}\limits_{\eps}^{\eps} R\mathop{\rightrightarrows}\limits_{\eps}^{\eps} R\mathop{\rightrightarrows}\limits_{\eps\cdot\lambda}^{\eps} R]
&= \det(\Delta_{\eps})^{-1}\det(R\st{\eps}\r  R\st{\eps}\r  R\st{\eps \cdot \lambda}\To  R) =\lambda. 
\end{align*}
\end{proof}

\subsection{A counterexample to two conjectures by Maltsiniotis}\label{counter}
Based on the following example, which goes back to Deligne, Vaknin, Ferrand and Breuning \cite{dtc,ferrand,dpcecrk0g}, we disprove two conjectures due to Maltsiniotis.

Let $\C{E}=\mathbf{proj}(R)$ be the category of finitely generated free modules over the ring of dual numbers $R=k[\eps]/\eps^2$ over a field $k$. We regard $D^{b}(\C{E})$ as a strongly triangulated category with the structure indicated in \cite{cts}. It is well known that $K_1(\C{E})\cong K_1(R)\cong R^\times$ is the group of units. 
There is an isomorphism,
$k\times k^{\times}\cong  R^{\times}\colon (x,u)\mapsto u(1+x\eps)$.
Given $x\in k$, the element $1+x\varepsilon \in R^\times$ corresponds to $$[1+x\varepsilon\colon R\st{\sim}\To  R]\in K_1(\C{E}).$$ This element is in the kernel of $K_1(\C{E})\r K_1(^{s}\!D^b(\C{E}))$ since we have the following automorphism of \exact\ triangles, see (R7),  
$$
\xymatrix{R\ar[r]^{\eps}\ar[d]_{1}&R\ar[r]\ar[d]^{1+x\eps}&C\ar[r]\ar[d]^{1}&\Sigma R\ar[d]^{1}\\
R\ar[r]^{\eps}&R\ar[r]&C\ar[r]&\Sigma R}
$$
Indeed,
$C$ is the complex $\cdots\r 0\r R\st{\eps}\r  R\r 0\r\cdots$,
and the square in the middle commutes in the derived category since we have a homotopy defined by the homomorphism $R\r R\colon 1\mapsto x$. Maltsiniotis conjecture in \cite{cts} that $K_n(\C{E})\r K_n(^{s}\!D^b(\C{E}))$ would be an isomorphism for all $n\geq 0$.

The same example shows that, if we only regard $D^b(\C{E})$ as a triangulated category, then the comparison homomorphisms,
\begin{align*}
K_1(\C{E})&\To  K_1(^{b}\!D^b(\C{E})),&K_1(\C{E})&\To  K_1(^{d}\!D^b(\C{E})),&K_1(\C{E})&\To  K_1(^{v}\!D^b(\C{E})),
\end{align*}
are not isomorphisms.

Moreover, if $\mathbb{D}_{\C{E}}$ is the triangulated derivator associated to $\C{E}$ \cite{dtce}, the comparison homomorphism  $K_{1}(\mathbb{D}_{\C{E}})\r  K_{1}(^{s}\mathbb{D}_{\C{E}}(*))$ in \cite{cts} is neither an isomorphism because the composite,
$$K_{1}(\C{E})\st{\cong}\To  K_{1}(\mathbb{D}_{\C{E}})\To  K_{1}(^{s}\mathbb{D}_{\C{E}}(*))\cong K_1(^{s}\!D^b(\C{E}))$$
is the previous comparison homomorphism between Quillen's $K$-theory and Maltsiniotis $K$-theory of a strongly triangulated category, which is not injective. The first arrow is the natural comparison homomorphism in \cite{ktdt}, which is an isomorphism by \cite[Theorem 1]{malt}. Maltsiniotis also conjectured that $K_{n}(\mathbb{D}_{\C{E}})\r  K_{n}(^{s}\mathbb{D}_{\C{E}}(*))$  would be an isomorphism for all $n\geq 0$, see \cite{cts}.

We can actually compute Neeman's $K_{1}(^{d}\!D^{b}(\C{E}))$ and  Breuning's $K_{1}(^{b}\!D^{b}(\C{E}))$. This improves and generalizes some computations in \cite{dpcecrk0g}.

\begin{prop}
For $\filledsquare=b,d$, 
the stable quadratic module 
$\D{D}_{*}(^{\filledsquare}D^{b}(\C E))$ 
is weakly equivalent to
$$\mathbb{Z}\otimes \mathbb{Z}\st{\grupo{\cdot,\cdot}}\To  k^{\times}\st{\partial}\To  \mathbb{Z},\qquad \grupo{1,1}=-1,\qquad \partial=0.$$
Moreover, the comparison homomorphism,
$$k\times k^{\times}\cong K_{1}(k[\eps]/\eps^2)\To  K_{1}(^{\filledsquare}D^{b}(\C E))\cong k^{\times},$$
is the natural projection onto the second factor. 
\end{prop}

\begin{proof}
Theorem \ref{todo} and Remark \ref{warrior} show that 
$\D{D}_{*}(^{\filledsquare}D^{b}(k))$ is weakly equivalent to  the stable quadratic module in the statement. We have already seen that the subgroup $k\subset k\times k^{\times}$ is in the kernel of the comparison homomorphism, which is known to be surjective, see Section \ref{comparase}. Therefore, it induces an epimorphism,
$k^{\times}\onto K_{1}(^{d}D^{b}(k[\eps]/\eps^2))$. 
This epimorphism is also injective since the following composite is the identity,
$$k^{\times}\onto K_{1}(^{d}D^{b}(k[\eps]/\eps^2))\To  K_{1}(^{d}D^{b}(k))\cong k^{\times}.$$
Here the second arrow is induced by the change of coefficients along the $k$-algebra morphism $k[\eps]/\eps^2\onto k\colon\eps\mapsto 0$. In particular, the map $\D{D}_{*}(^{\filledsquare}D^{b}(\C E))\r \D{D}_{*}(^{\filledsquare}D^{b}(k))$ induced by the previous change of coefficients is an isomorphism in $\pi_1$. It is also an isomorphism in $\pi_0$, since $K_0(R)\cong K_0(k)\cong\mathbb{Z}$ generated by the free module of rank $1$, compare again Section \ref{comparase}. Hence we are done.
\end{proof}


\numberwithin{equation}{subsection}

\section{Strict \Picardcategories}\label{victim}

In this section we review strictification results for \Picardcategories\ and related categorical structures. This theory is essential for the proof of our main results on determinant functors. 

\subsection{Categorical groups}

A monoidal groupoid  $(\C G,\otimes, I)$ with unit object $I$ is a \emph{categorical group}
if for each object $x$ of $\C G$ there is an object $x^*$ and a map
$j_x:x^*\otimes x\cong I$.
Equivalently, there is a contravariant functor $*$ on $\C G$
such that the endofunctors  $\sp\otimes x$ and
$x\otimes\sp$ are equivalences of categories with inverses 
$\sp\otimes x^*$ and  $x^*\otimes\sp$, respectively~\cite{ccgsaa}.

A categorical group is {\em braided} or {\em symmetric} if the underlying
monoidal category is. Recall that a  {\em braiding} is a natural
isomorphism, $$\comm{x}{y}{\C{G}} \colon x\otimes y\To  y\otimes x,$$
satisfying certain coherence laws~\cite{btc}, and is a {\em symmetry}
if $\comm{y}{x}{\C{G}}\circ\comm{x}{y}{\C{G}}=1_{x\otimes y}$ is the identity.
A \emph{\Picardcategory} is just a symmetric categorical group. 

A \emph{tensor functor} between categorical groups is a functor
$F\colon\C{G}\r\C{H}$ 
together with comparison maps for tensor units and multiplication,
$$
\up{}\colon I_{\C H}\To F(I_{\C G}),
\qquad
\hp{x}{y}{F}\colon F(x)\otimes F(y)\To  F(x\otimes y),
$$
which are natural and compatible with the associativity
and unit isomorphisms~\cite{fbtc}. 
A tensor functor between braided (or symmetric) categorical groups is {\em symmetric}
if it is also compatible with the braiding isomorphisms.
A \emph{tensor natural transformation}  $\alpha\colon F\rr G$ 
is one which commutes with the comparison maps
 for multiplication. (Braided, symmetric) categorical groups, (symmetric) tensor
functors and tensor natural transformations form $2$-categories. Actually, categories enriched in groupoids.  Notice that the obvious forgetful $2$-functors
$$\text{symmetric cat.\ groups}\To  \text{braided cat.\ groups}\To \text{cat. groups}$$
are faithful in dimension $1$ and fully faithful in dimension $2$, i.e.~injective on tensor functors and bijective on tensor natural transformations.

 The
{\em homotopy groups} of a (braided, symmetric) categorical group $\C{G}$ are,
\begin{align*}
\pi_0(\C{G})&=\text{isomorphism classes of objects, with $+$ induced by $\otimes$},\\
\pi_1(\C{G})&=\aut_\C{G}(I).
\end{align*}
Homotopy groups detect equivalences. The group  $\pi_0(\C{G})$ acts on $\pi_1(\C{G})$ by $$x^*\otimes x\otimes( I\st{f}\r I)^{[x]}= x^*\otimes( I\st{f}\r I)\otimes x\colon x^*\otimes x\To  x^*\otimes x,$$ and the action is trivial in the braided  case. 
One can define the {\em $k$-invariant} in the braided case as the natural quadratic map
$$
\eta\colon\pi_0(\C{G})\To \pi_1(\C{G}),
$$
such that $x\otimes x\otimes \eta([x])=\comm{x}{x}{\C{G}}$, and $\C{G}$ is symmetric if and only if the $k$-invariant factors through a homomorphism $$\eta:\pi_0(\C{G})\otimes\ZZ/2\To \pi_1(\C{G}).$$


A (braided, symmetric) categorical group is \emph{strict} if the
associativity and unit isomorphisms are identities and the isomorphisms $j_x$ can be
chosen to be identities. Thus
the underlying monoidal category is strict and 
the functors  $\sp\otimes x$ and
$x\otimes\sp$ are isomorphisms of categories.
If $\C G$ and $\C H$ are strict then
$F\colon\C{G}\r\C{H}$ 
is a {\em strict tensor functor}
if the comparison maps for multiplication are all identities. 

Strict (braided, symmetric) categorical groups, strict (symmetric) tensor functors and tensor natural transformations again form a $2$-category. 


\subsection{Crossed modules}\label{crossedmodules}
Recall that a \emph{crossed module} is a group homomorphism $\partial\colon C_1\r C_0$ together with a right action of $C_0$ on $C_1$ such that, for $c_{i},c_{i}'\in C_{i}$,
\begin{enumerate}
\item $\partial({c_1}^{c_0})=-c_0+\partial(c_1)+c_0$,
\item ${c_1}^{\partial(c_1')}=-c_1'+c_1+c_1'$.
\end{enumerate}
We denote the group laws additively, although the groups may be non-abelian. 
It follows that the image of $\partial$ is always a
normal subgroup, and the kernel is always central. The
{\em homotopy groups} of $C_*$ are defined as in Definition \ref{homotopygroups}. 
The action of $C_0$ on $C_1$ induces an action of $\pi_0(C_*)$ on $\pi_1(C_*)$.

A \emph{reduced $2$-crossed module}, or simply \emph{reduced $2$-module}, is a crossed module together with a map,
$$\grupo{\cdot,\cdot}\colon C_0\times C_0\To  C_1,$$
which controls commutators. It must satisfy:
\begin{enumerate}\setcounter{enumi}{2}
\item $\partial\grupo{c_0,c_0'}=[c_{0}',c_{0}]$,
\item $c_1^{c_0}=c_1+\grupo{c_0,\partial(c_1)}$,
\item $\grupo{c_0,\partial(c_1)}+\grupo{\partial(c_1),c_0}=0$,
\item $\grupo{c_0,c_0'+c_0''}=\grupo{c_0,c_0'}^{c_0''}+\grupo{c_0,c_0''}$,
\item $\grupo{c_0+c_0',c_0''}=\grupo{c_0',c_0''}+\grupo{c_0,c_0''}^{c_0'}$.
\end{enumerate}
The crossed module $\partial$ and the bracket $\grupo{\cdot,\cdot}$ form a \emph{stable $2$-crossed module}, or simply \emph{stable $2$-module}, if (3), (4), (6) and
\begin{enumerate}\setcounter{enumi}{7}
\item $\grupo{c_0,c_0'}+\grupo{c_0',c_0}=0$
\end{enumerate}
are satisfied. In a reduced or stable $2$-module the action of $C_0$ on $C_1$ is completely determined by the
bracket $\grupo{\cdot,\cdot}$, by (4),  so (1) is redundant and (2)
becomes
\begin{enumerate}\setcounter{enumi}{8}
\item $\grupo{\partial(c_1),\partial(c_1')}=[c_{1}',c_{1}]$.
\end{enumerate}
The $k$-{\em invariant} of a reduced $2$-module $C_*$ is the natural quadratic map,
\begin{align*}
\eta:\pi_0(C_*)&\To \pi_1(C_*), \\
{ [c_0]}&\;\mapsto\; \grupo{c_0,c_0}.
\end{align*}
In fact $C_*$ is stable if and only if the $k$-invariant factors through a homomorphism, $$\eta:\pi_0(C_*)\otimes\ZZ/2\To \pi_1(C_*).$$

A \emph{crossed module morphism} $f\colon C_*\to D_*$
is a pair of group homomorphisms $f_i\colon C_i\r D_i$, $i=0,1$,
which respect the actions and satisfy $\partial f_1=f_0\partial$.
A \emph{reduced} or \emph{stable $2$-module morphism} 
is a morphism $f$ between the underlying crossed modules which
preserves the bracket, $\grupo{f_0,f_0}=f_1\grupo{\cdot,\cdot}$.

A \emph{homotopy} $\alpha\colon f\rr g$ between two such morphisms
is a function $\alpha\colon C_0\r D_1$ satisfying the equations in Definition \ref{baues}. Horizontal and vertical compositions are defined as there. Thus we obtain $2$-categories of crossed modules and of reduced and
stable $2$-modules, together with their morphisms
and homotopies of morphisms. All $2$-morphisms are invertible, hence we actually have categories enriched in groupoids. Notice that the obvious forgetful $2$-functors
$$\text{stable $2$-modules}\To  \text{reduced $2$-modules}\To \text{crossed modules}$$
are faithful in dimension $1$ and fully faithful in dimension $2$, i.e.~injective on morphisms and bijective on homotopies.

The $2$-category of
stable quadratic modules, introduced in Definition \ref{baues}, can be identified with the full reflective
sub-$2$-category of the $2$-category of stable $2$-modules given by those
objects $C_*$
for which the bracket vanishes whenever one argument lies in
the commutator subgroup of $C_0$, 
\begin{equation*}
\grupo{c_0,[c_0',c_0'']}=0.
\end{equation*}
Compare \cite[Lemma 4.18]{1tk}.



\subsection{Crossed modules and strict categorical groups}

The construction $\Gamma$ in Definition \ref{gammatron} applied to a braided (resp. symmetric) $2$-crossed module $C_{*}$ yields a strict braided (resp. symmetric) categorical group $\Gamma C_{*}$. Moreover, it can also be applied to an ordinary crossed module, producing a categorical group with no braiding. Furthermore, it is also defined on morphisms and homotopies as indicated there. Thus we obtain $2$-functors
$$\text{(reduced, stable $2$-)crossed modules}\st{\Gamma}\To \text{strict (braided, symmetric) cat.\ groups}.$$

Recall that a strong equivalence of $2$-categories is a $2$-functor which is fully faithful on $1$-morphisms and $2$-morphisms, and essentially surjective on objects in the classical sense, i.e.~any object in the target is isomorphic, not just equivalent, to an object in the image. 

\begin{prop}\label{anequivalence}
The three $2$-functors called $\Gamma$ above are strong equivalences of $2$-categories. 
%
\end{prop}

\begin{proof}
The result is essentially due to Verdier, who discovered the construction $\Gamma$, see \cite{ggcm} for some history. 
To recover a crossed module from a strict categorical group $\C G$ is
straightforward: $C_0$ is the object group, $C_1$ is the kernel of
target homomorphism, and
$$\partial(x\st a\to I)=x,\qquad\quad(x\st a\to
I)^y=y^*\otimes (x\st a\to I)\otimes y.$$ 
A braiding or symmetry also defines a bracket on this crossed module, 
$$\grupo{x,y}=y^*\otimes x^*\otimes
(y\otimes x\st{\com}\longrightarrow x\otimes y).$$
The morphism defined by a strict functor is the obvious one, and a tensor natural transformation $\alpha\colon f\rr g$ between strict (symmetric) tensor functors $f,g\colon\C{G}\r\C{H}$ yields a homotopy defined by the map $x\mapsto g(x)^{*}\otimes(\alpha(x)\colon f(x)\r g(x))$.
\end{proof}


\subsection{Strictifying tensor functors}
A strict (braided, symmetric)  categorical group is called  \emph{$0$-free} if
the group of objects is free. A 
(reduced, stable $2$-)crossed module $C_*$ is \emph{$0$-free}
if $C_0$ is a free group. 

A weak equivalence of $2$-categories is a $2$-functor which is fully faithful on  $2$-morphisms, essentially surjective on $1$-morphisms, and such that any object in the target is equivalent  to an object in the image (in the $2$-categorical sense).

In this section we shall prove the following result.

\begin{thm}\label{strictification}
The inclusion $2$-functor induces a weak equivalence between the $2$-categories of:
\begin{itemize}
\item $0$-free strict (braided, symmetric)  categorical groups, strict (symmetric) tensor functors and tensor natural transformations,

\item (braided, symmetric) categorical groups, (symmetric) tensor functors and tensor natural transformations.
\end{itemize}
\end{thm}

Obviously the former is a sub-$2$-category of the latter, full in dimension $2$. We give some details of the (folklore) results that (braided, symmetric) categorical groups can be strictified, and that one can replace a strict categorical group by a $0$-free one.

\begin{lem}\label{strictificationobjects}
Any  (braided, symmetric) categorical group is (symmetric) tensor equivalent to a $0$-free strict one.
\end{lem}

\begin{proof}
We know that tensor equivalence classes of categorical groups $\C{G}$ with fixed isomorphisms $\pi_0(\C{G})\cong G$, $\pi_1(\C{G})\cong M$ of groups and $G$-modules, respectively, are in bijection with cohomology classes $H^3(G,M)$ \cite[Chapitre 1 \S 1, Proposition 10]{grc}. We also know that any such class can be represented by a crossed module \cite{ctagIII}, therefore any categorical group is equivalent to a strict one. In addition a crossed module $C_*$ can be replaced by a $0$-free one $D_*$ via the pull-back construction
$$\xymatrix{D_1\ar[d]\ar[r]^\partial\ar@{}[rd]|{\text{pull}}&\grupo{E}\ar[d]\\C_1\ar[r]_\partial&C_0}$$
Here $E\subset C_0$ is a set of generators of $\pi_0(C_*)$, $\grupo{E}$ is the free group with basis $E$, and $D_0=\grupo{E}\r C_0$ is induced by the inclusion. This commutative square is a morphism of crossed modules which induces isomorphisms on homotopy groups, compare \cite[Proposition 4.15]{2hg1}, and therefore an equivalence between the corresponding categorical groups. Notice however that the inverse equivalence need not be strict.

The braided and symmetric case go along the same lines. If $\C{G}$ is braided or symmetric, we can strictify the underlying categorical group and then transfer the symmetry constraint along the equivalence. In this way we obtain an equivalent (braided, symmetric) strict categorical group.  The pull-back construction allows us again to replace any reduced or stable $2$-module by a $0$-free one,  compare \cite[Proposition 4.15]{2hg1}.
\end{proof}

When the source is $0$-free,  (symmetric) tensor functors can also be
strictified. We have not found any reference for the following lemma in the literature.

\begin{lem}\label{blah}
Let $\C G$ and $\C H$ be strict categorical groups where $\C G$ is
$0$-free. Then for any tensor functor $\phi:\C G\to \C H$ there exists
a strict tensor functor $\phi^s:\C G\to \C H$ together with a tensor  natural
transformation $\alpha\colon\phi^s\nattrans\phi$.

Moreover, if $\phi$ is a symmetric tensor functor between braided or symmetric categorical groups $\C G$ and $\C H$, then $\phi^s$  can be taken to be symmetric.
\end{lem}

\begin{proof}
Suppose $\ob(\C G)$
is free on a set $B$, and define $\phi^s:\ob(\C G)\to\ob(\C H)$ to be
the unique group homomorphism with $\phi^s(b)=\phi(b)$ for $b\in
B$. The transformation $\alpha:\phi^s\nattrans\phi$ is defined
on the neutral element $I_{\C G}$, elements $b\in B$ and products
$b\otimes b'$, for $b,b'\in  B$, as follows:
\begin{align*}
\alpha(I_{\C G})&=\quad\xymatrix{I_{\C H}\ar[r]^-{\up\phi}&\phi(I_{\C G})},\\
\alpha(b)&=\quad\xymatrix{\phi(b)\ar[r]^-{1}&\phi(b)},\\
\alpha(b\otimes b')&=\quad\xymatrix{\phi(b)\otimes\phi(b')\ar[r]^-{\hp b{b'}\phi}&\phi(b\otimes b')}.
\end{align*}
In general $\alpha$ is defined on objects by induction on the reduced word length in the free group, by the following commutative diagram
$$\xymatrix@l{\phi(x\otimes b)\ar@{<-}[d]_{\alpha(x\otimes b)}
\ar@{<-}[r]^-{\hp xb\phi}&\phi(x)\otimes\phi(b)
\ar@{<-}[d]^{\alpha(x)\otimes1} \\
\phi^s(x\otimes b)\ar@{<-}[r]_-=&\phi^s(x)\otimes\phi^s(b)
}$$
This diagram defines $\alpha(x\otimes b)$ from $\alpha(x)$ provided the last letter in the reduced word $x$ is not $b^{-1}$. At the same, if $x=y\otimes b^{-1}$ is a reduced word, it defines $\alpha(y\otimes b^{-1})$ from $\alpha(y)=\alpha(x\otimes b)$. 
Notice also that this diagram is one case of the condition that $\alpha$ is a
tensor natural transformation. The condition is verified in general using
induction (on word length of $y$) and the following commutative diagram
$$\xymatrix@l@C=11pt{
\phi(x\otimes y\otimes b)\ar@{<-}[ddd]|{\alpha(x\otimes y\otimes
  b)}\ar@{<-}[dr]|{\hp{x\otimes y}b\phi}
\ar@{<-}[rrr]^-{\hp x{y\otimes b}\phi}
&&&\phi(x)\otimes\phi(y\otimes b)\ar@{<-}[dl]|{1\otimes\hp yb\phi}
\ar@{<-}[ddd]|{\alpha(x)\otimes \alpha(y\otimes b)} \\
&
\phi(x\otimes y)\otimes \phi(b)\ar@{<-}[d]_{\alpha(x\otimes y)\otimes1}
\ar@{<-}[r]^-{
\begin{array}{c}
{}\vspace{-10pt}\\
\scriptstyle\hp xy\phi\otimes 1
\end{array}
}&\phi(x)\otimes\phi(y)\otimes \phi(b)
\ar@{<-}[d]^{\alpha(x)\otimes \alpha(y)\otimes1}
\\
&\phi^s(x\otimes y)\otimes\phi^s(b)\ar@{<-}[r]_-=&\phi^s(x)\otimes\phi^s(y)\otimes\phi^s(b)
\\
\phi^s(x\otimes y\otimes b)\ar@{<-}[ru]_-=\ar@{<-}[rrr]_-=&&&
\phi^s(x)\otimes\phi^s(y\otimes b)\ar@{<-}[lu]^-=
}$$
Now $\phi^s$ is defined on morphisms $f\colon x\to y$ by the
following commutative diagram:
$$\xymatrix{\phi^s(x)\ar[d]_{\alpha(x)}\ar[r]^-{\phi^s(f)}&\phi^s(y)\ar[d]^{\alpha(y)}\\
\phi(x)\ar[r]_-{\phi(f)}&\phi(y)
}
$$
This is just the naturality condition for $\alpha$.

The following diagram shows the functor $\phi^s$ so defined is a tensor functor:
$$
\xymatrix@l{
\phi^s(y\otimes y') \ar@{<-}[rrr]^-=
\ar@{<-}[ddd]_{\phi^s(f\otimes f')}    &&& \phi^s(y)\otimes\phi^s(y')
                               \ar@{<-}[ddd]^{\phi^s(f)\otimes\phi^s(f')}  \\
&\phi(y\otimes y')\ar@{<-}[d]_{\phi(f\otimes f')}\ar@{<-}[ul]^{\alpha(y\otimes  y')}
\ar@{<-}[r]^-{\hp y{y'}\phi}
&\phi(y)\otimes\phi(y')\ar@{<-}[d]^{\phi(f)\otimes\phi(f')}\ar@{<-}[ur]_{\alpha(y)\otimes\alpha(y')\quad}&\\
&\phi(x\otimes x')\ar@{<-}[dl]_{\alpha(x\otimes
  x')}\ar@{<-}[r]_-{\hp x{x'}\phi}
&\phi(x)\otimes\phi(x')\ar@{<-}[dr]^{\alpha(x)\otimes \alpha(x')\quad}&\\
\phi^s(x\otimes x') \ar@{<-}[rrr]_-=   &&& \phi^s(x)\otimes\phi^s(x')
}
$$
Finally, we note
that if $\phi$ is symmetric then so is $\phi^s$, by the following
commutative diagram:
$$
\xymatrix@l{
\phi^s(x\otimes y) \ar@{<-}[rrr]^-{\phi^s(\comm{y}{x}{\C P})}
\ar@{<-}[ddd]_=                                      &&& \phi^s(y\otimes x)
                                                    \ar@{<-}[ddd]^=  \\
&\phi(x\otimes y)\ar@{<-}[lu]^{\alpha(x\otimes y)}\ar@{<-}[r]^{\phi(\com)}\ar@{<-}[d]_{\hp{x}{y}{}}&\phi(y\otimes x)\ar@{<-}[d]^{\hp{y}{x}{}}\ar@{<-}[ru]_{\alpha(y\otimes x)\quad}\\
&\phi(x)\otimes\phi(y)\ar@{<-}[ld]_{\alpha(x)\otimes\alpha(y)}\ar@{<-}[r]_{\com}&\phi(y)\otimes\phi(x)\ar@{<-}[rd]^{\alpha(y)\otimes\alpha(x) \quad}\\
\phi^s(x)\otimes\phi^s(y) \ar@{<-}[rrr]_-{\comm{\phi^s(y)}{\phi^s(x)}{\C Q}}
                                           &&& \phi^s(y)\otimes\phi^s(x)
}
$$
\end{proof}

Now Theorem \ref{strictification} follows from Lemmas \ref{strictificationobjects} and \ref{blah}.

\begin{prop}
The reflection $2$-functor in \cite[Lemma 4.18]{1tk} induces  a weak equivalence between the $2$-categories of $0$-free stable $2$-modules and $0$-free stable quadratic modules.
\end{prop}

This follows from \cite[Remark 4.21]{1tk}. The following result is a combination of this last result, Proposition \ref{anequivalence}, and Theorem \ref{strictification}.

\begin{cor}\label{pg0sq}
The $2$-functor $\Gamma$ in Definition \ref{gammatron} induces a weak equivalence between the $2$-categories of $0$-free stable quadratic modules and  \Picardgroupoids.
\end{cor}

\section{A unified approach to determinant functors}\label{universaldet}

The motto of category theory is `one proof replaces many'. With this philosophy in mind, in this section we develop an abstract theory of determinant functors encapsulating all examples in Section \ref{sec1}. In this context, we show the existence of universal determinant functors. We explicitly construct the target, i.e.~the category of virtual objects, as the \Picardgroupoid\ associated to a stable quadratic module defined by a presentation. These results give our main theorems, stated in Section~\ref{sec1}, as corollaries.

The techniques of this section are simplicial. In fact, we only need the low-dimensional part of certain simplicial categories. The strictification results of the previous section are crucial to make the proof of our main results as short as possible. Another advantage is our simple construction of the category of virtual objects.

\subsection{Determinant functors for $S_\bullet$-categories}\label{unified}


Most $K$-theories are defined via a simplicial category, similar to Waldhausen's $S_{\bullet}$ construction. We now define determinant functors for such simplicial categories. This definition generalizes all notions of determinant functors introduced in Section \ref{sec1}.

\begin{defn}\label{detsimpcat}
An \emph{$S_{\bullet}$-category}  $\C{C}_\bullet$ is a simplicial category

\medskip

$$\xymatrix@C=40pt{\ar@{.}[r]&\C{C}_3\ar@<-1.5ex>[r]_{d_3}\ar@<-.5ex>[r]|<(.6){d_2}\ar@<.5ex>[r]|<(.3){d_1}\ar@<1.5ex>[r]^{d_0}&\C{C}_2\ar@<-1ex>[r]_{d_2}\ar[r]|{d_1}\ar@<1ex>[r]^{d_0}&\C{C}_1\ar@<-.5ex>[r]_{d_1}\ar@<.5ex>[r]^{d_0}\ar@/_20pt/[l]|{s_1}\ar@<-1ex>@/_20pt/[l]_{s_0}&{*},\ar@/_20pt/[l]_{s_0}}$$
such that $\C{C}_0={*}$ is the terminal category, with only one object $*$ and one
morphism (the identity), $\C{C}_n$ has finite coproducts for all $n\geq 0$, and faces and degeneracies preserve coproducts. In particular $s_{0}^{n}(*)$ is initial in $\C C_{n}$ for all $n\geq 0$.
Moreover, $\C{C}_\bullet$ is endowed with a simplicial  subcategory $\operatorname{we}\C{C}_\bullet$ containing all isomorphisms $\operatorname{iso}\C{C}_\bullet\subset \operatorname{we}\C{C}_\bullet$, whose morphisms are called \emph{weak equivalences}. Finite coproducts of weak equivalences are required to be weak equivalences. 

Let $\C P$ be a \Picardgroupoid. A \emph{determinant functor}   $\det\colon\C{C}_\bullet\r \C P$ consists of a functor,   $$\det\colon\operatorname{we}\C{C}_1\To \C{P},$$ together with
\emph{additivity data}: for any object $\Delta$ in $\C{C}_2$, a morphism in $\C P$,
$$\det(\Delta)\colon \det(d_0\Delta)\otimes\det(d_2\Delta)\To  \det(d_1\Delta),$$
natural with respect to morphisms in $\operatorname{we}\C{C}_2$.
The following two axioms must be satisfied.

\begin{enumerate}

\item \emph{Associativity}: Let $\Theta$ be an object in $\C{C}_3$. The following diagram in $\C P$ commutes,
$$\xymatrix@C=-20pt{&\det(d_1d_2\Theta)&\\
\det(d_0d_1\Theta)\otimes\det(d_1d_3\Theta)\ar[ru]^{\det(d_1\Theta)}&&\det(d_0d_2\Theta)\otimes\det(d_2d_3\Theta)\ar[lu]_{\det(d_2\Theta)}\\
\det(d_0d_1\Theta)\otimes(\det(d_0d_3\Theta)\otimes\det(d_2d_3\Theta))\ar[u]^{1\otimes\det(d_3\Theta)}\ar[rr]_{\ass}&&(\det(d_0d_3\Theta)\otimes\det(d_0d_3\Theta))\otimes\det(d_2d_3\Theta)\ar[u]_{\det(d_0\Theta)\otimes1}}$$

\item \emph{Commutativity:} given a coproduct $X\sqcup Y$ of two objects $X$ and $Y$ in $\C{C}_1$ 
the following triangle commutes,
$$\xymatrix@C=0pt{&\det(\X\sqcup \Y)\ar@{<-}[rd]^-{\quad\det(s_0\X\sqcup s_1\Y)}\ar@{<-}[ld]_-{\det(s_1\X\sqcup s_0\Y)\quad}&\\\det(\Y)\otimes\det(\X)\ar[rr]_-{\com}&&\det(\X)\otimes\det(\Y)}$$
\end{enumerate}
\end{defn}

\begin{rem}\label{evenless}
Notice that in the previous definition we do not use all the structure of $\C{C}_\bullet$ but only its $3$-truncation. Actually even less, only the piece of $\C{C}_\bullet$ depicted in the diagram at the beginning of Definition \ref{detsimpcat}. Moreover, we do not use all the structure of that diagram, but just the coproduct operation in $\operatorname{we}\C{C}_1$ and $\operatorname{we}\C{C}_2$, the category structure of $\operatorname{we}\C{C}_1$, the underlying graph of $\operatorname{we}\C{C}_2$, and the set of objects of $\C{C}_3$. This can be illustrated by the following diagram, 
$$\xy
(0,0)* {\bullet},
(10,0)* {\bullet},
(20,0)* {\bullet},
(10,10)* {\bullet},
(20,10)* {\bullet},
(20,20)* {\bullet},
(0,10)* {\circ},
(0,20)* {\circ},
(10,20)* {\circ},
(-10,10)* {\cdots},
(30,0)*[r]{*},
(30,10)*[r]{*},
(30,20)*[r]{*},
(40,0)*[r]{\text{objects}},
(40,10)*[r]{\text{morphisms}},
(40,20)*[r]{\text{composition}},
(2,1.5);(8,1.5) **\dir{-} *\dir{>},
(2,.5);(8,.5) **\dir{-} *\dir{>},
(2,-1.5);(8,-1.5) **\dir{-} *\dir{>},
(2,-.5);(8,-.5) **\dir{-} *\dir{>},
(12,1);(18,1) **\dir{-} *\dir{>},
(12,0);(18,0) **\dir{-} *\dir{>},
(12,-1);(18,-1) **\dir{-} *\dir{>},
(12,11);(18,11) **\dir{-} *\dir{>},
(12,10);(18,10) **\dir{-} *\dir{>},
(12,9);(18,9) **\dir{-} *\dir{>},
(19,18);(19,12) **\dir{-} *\dir{>},
(20,18);(20,12) **\dir{-} *\dir{>},
(21,18);(21,12) **\dir{-} *\dir{>},
(19.5,8);(19.5,2) **\dir{-} *\dir{>},
(20.5,8);(20.5,2) **\dir{-} *\dir{>},
(9.5,8);(9.5,2) **\dir{-} *\dir{>},
(10.5,8);(10.5,2) **\dir{-} *\dir{>},
(19,-2);(11,-2) **\crv{(15,-6)} ?>*\dir{>},
(20.5,-2);(9.5,-2) **\crv{(15,-8)} ?>*\dir{>},
\endxy
$$
\end{rem}

\begin{exm}\label{mainexm}
We now see how the determinant functors presented in
Section 1 are covered by our unified approach. Weak equivalences  are isomorphisms in all $S_{\bullet}$-categories defined  below, except for the first one. We need a distinguished zero object for the definition of degeneracies. This is not a real problem, since all zero objects are canonically isomorphic.

\begin{enumerate}

\renewcommand{\X}{X}
\renewcommand{\Y}{Y}
\renewcommand{\Z}{Z}

\item Determinant functors on a Waldhausen category $\C{W}$ coincide with determinant functors on Waldhausen's  $S_\bullet(\C{W})$ \cite{akttsI, akts}. 
This follows from the fact that $S_1(\C{W})$ is just $\C{W}$, $S_2(\C{W})$
is the category of cofiber sequences in $\C{W}$, $S_3(\C{W})$ is the
category of staircase diagrams in $\C{W}$, such as \eqref{stair}, and the non-trivial faces and degeneracies in low dimensions are,
\begin{equation*}
d_i(\X\into \Y\onto C^{f})=\left\{\begin{array}{ll}
C^{f},& i=0;\\
\Y,& i=1;\\
\X,& i=2;
\end{array}\right.
\end{equation*}
\begin{equation*}
s_i(\X)=\left\{\begin{array}{ll}
0\into \X\st{1}\onto \X,& i=0;\\
\X\st{1}\into \X\onto 0,& i=1;\\
\end{array}\right.
\end{equation*}
\begin{equation*}
d_i(\ref{stair})=\left\{\begin{array}{ll}
C^{f}\into C^{gf}\onto C^{g},& i=0;\\
\Y\into \Z\onto C^{g},& i=1;\\
\X\into \Z\onto C^{gf},& i=2;\\
\X\into \Y\onto C^{f},& i=3.
\end{array}\right.
\end{equation*}

\item  From this description of the low-dimensional part of $S_\bullet(\C{W})$, it also follows that derived determinant functors on a Waldhausen category $\C{W}$ coincide with determinant functors on  ${\ho S_\bullet(\C{W})}$. 



\renewcommand{\X}{X}
\renewcommand{\Y}{Y}
\renewcommand{\Z}{Z}

\item Given a triangulated category $\C{T}$ we can consider the $3$-truncated $S_{\bullet}$-category $\bar{S}_{\leq 3}(^{b}\!\C{T})$,
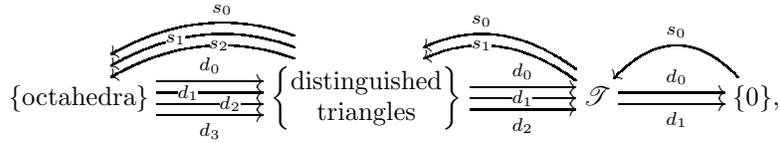
$$\xymatrix@C=40pt{\{\text{octahedra}\}\ar@<-1.5ex>[r]_-{d_3}\ar@<-.5ex>[r]|<(.35){d_2}\ar@<.5ex>[r]|<(.15){d_1}\ar@<1.5ex>[r]^-{d_0}
&\ar@/_20pt/[l]|<(.35){s_2}\ar@<-1ex>@/_20pt/[l]|<(.55){s_1}\ar@<-2ex>@/_20pt/[l]_{s_0}
\left\{
\!\!\!{\begin{array}{c}
\text{\exact} \\
\text{triangles}
\end{array}}\!\!\!
\right\}\ar@<-1ex>[r]_-{d_2}\ar[r]|-{d_1}\ar@<1ex>[r]^-{d_0}&\C{T}\ar@<-.5ex>[r]_{d_1}\ar@<.5ex>[r]^{d_0}\ar@/_20pt/[l]|{s_1}\ar@<-1ex>@/_20pt/[l]_{s_0}&\{0\},\ar@/_20pt/[l]_{s_0}}$$
with faces and degeneracies
\begin{equation*}
d_i(X\st{f}\r Y\r \cone{f}\r\Sigma X)=
\left\{\begin{array}{ll}
\cone{f},& i=0;\\
Y,& i=1;\\
X,& i=2;
\end{array}\right.
\end{equation*}
\begin{equation*}
s_i(\X)=\left\{\begin{array}{ll}
0\r \X\st{1}\r \X\r 0,& i=0;\\
\X\st{1}\r \X\r 0\r\Sigma \X,& i=1;
\end{array}\right.
\end{equation*}
\begin{equation*}
d_i(\ref{octa})=\left\{\begin{array}{ll}
\cone{f}\r \cone{gf}\r \cone{g}\r\Sigma \cone{f},& i=0;\\
Y\st{g}\r Z\r \cone{g}\r \Sigma Y,& i=1;\\
X\st{gf}\r Z\r \cone{gf}\r \Sigma X,& i=2;\\
X\st{f}\r Y\r \cone{f}\r\Sigma X,& i=3.
\end{array}\right.
\end{equation*}
The degeneracies $s_i(X\st{f}\r Y\r \cone{f}\r\Sigma X)$, $i=0,1,2$, are
defined as the unique octahedra with the faces imposed by the simplicial identities. 

Breuning determinant functors on $\C{T}$ are the same as determinant functors on $\bar{S}_{\leq 3}(^{b}\!\C{T})$.


\item We can restrict ourselves to special octahedra,
$$\xymatrix@C=40pt{
\left\{
\!\!\!{\begin{array}{c}
\text{special}\\ 
\text{octahedra}
\end{array}}\!\!\!
\right\}\ar@<-1.5ex>[r]_-{d_3}\ar@<-.5ex>[r]|<(.35){d_2}\ar@<.5ex>[r]|<(.15){d_1}\ar@<1.5ex>[r]^-{d_0}
&\ar@/_20pt/[l]|<(.35){s_2}\ar@<-1ex>@/_20pt/[l]|<(.55){s_1}\ar@<-2ex>@/_20pt/[l]_{s_0}
\left\{
\!\!\!{\begin{array}{c}
\text{\exact} \\
\text{triangles}
\end{array}}\!\!\!
\right\}\ar@<-1ex>[r]_-{d_2}\ar[r]|-{d_1}\ar@<1ex>[r]^-{d_0}&\C{T}\ar@<-.5ex>[r]_{d_1}\ar@<.5ex>[r]^{d_0}\ar@/_20pt/[l]|{s_1}\ar@<-1ex>@/_20pt/[l]_{s_0}&\{0\}.\ar@/_20pt/[l]_{s_0}}$$
Neeman's ${S}_\bullet(^{d}\!\C{T})$ \cite{kttcneeman} is the simplicial set of objects of an $S_{\bullet}$-category $\bar{S}_{\bullet}(^{d}\!\C{T})$ whose $3$-truncation is this one. Determinant functors on ${\bar{S}_\bullet(^{d}\!\C{T})}$ are special determinant functors on $\C{T}$. 

\item We can also consider virtual triangles and octahedra  instead,
$$\xymatrix@C=40pt{
\left\{
\!\!\!{\begin{array}{c}
\text{virtual}\\ 
\text{octahedra}
\end{array}}\!\!\!
\right\}\ar@<-1.5ex>[r]_-{d_3}\ar@<-.5ex>[r]|<(.35){d_2}\ar@<.5ex>[r]|<(.15){d_1}\ar@<1.5ex>[r]^-{d_0}
&\ar@/_20pt/[l]|<(.35){s_2}\ar@<-1ex>@/_20pt/[l]|<(.55){s_1}\ar@<-2ex>@/_20pt/[l]_{s_0}
\left\{
\!\!\!{\begin{array}{c}
\text{virtual}\\ 
\text{triangles}
\end{array}}\!\!\!
\right\}\ar@<-1ex>[r]_-{d_2}\ar[r]|-{d_1}\ar@<1ex>[r]^-{d_0}&\C{T}\ar@<-.5ex>[r]_{d_1}\ar@<.5ex>[r]^{d_0}\ar@/_20pt/[l]|{s_1}\ar@<-1ex>@/_20pt/[l]_{s_0}&\{0\}.\ar@/_20pt/[l]_{s_0}}$$
Neeman's ${S}_\bullet(^{v}\!\C{T})$ \cite{kttcneeman} is the simplicial set of objects of an $S_{\bullet}$-category $\bar{S}_{\bullet}(^{v}\!\C{T})$ whose $3$-truncation is this one. Determinant functors on ${\bar{S}_\bullet(^{v}\!\C{T})}$ are virtual determinant functors on $\C{T}$.

\item  Given a strongly triangulated category $\C{T}_{\infty}$ we consider 
$$\xymatrix@C=40pt{
\left\{
\!\!\!{\begin{array}{c}
\text{\exact} \\
\text{octahedra}
\end{array}}\!\!\!
\right\}
\ar@<-1.5ex>[r]_-{d_3}\ar@<-.5ex>[r]|<(.35){d_2}\ar@<.5ex>[r]|<(.15){d_1}\ar@<1.5ex>[r]^-{d_0}
&\ar@/_20pt/[l]|<(.35){s_2}\ar@<-1ex>@/_20pt/[l]|<(.55){s_1}\ar@<-2ex>@/_20pt/[l]_{s_0}
\left\{
\!\!\!{\begin{array}{c}
\text{\exact} \\
\text{triangles}
\end{array}}\!\!\!
\right\}\ar@<-1ex>[r]_-{d_2}\ar[r]|-{d_1}\ar@<1ex>[r]^-{d_0}&\C{T}_{\infty}\ar@<-.5ex>[r]_{d_1}\ar@<.5ex>[r]^{d_0}\ar@/_20pt/[l]|{s_1}\ar@<-1ex>@/_20pt/[l]_{s_0}&\{0\},\ar@/_20pt/[l]_{s_0}}$$
Maltsiniotis's simplicial set $Q_\bullet(\C{T}_{\infty})$ \cite{cts} is the simplicial set of objects of an $S_{\bullet}$-category $\bar Q_\bullet(\C{T}_{\infty})$ whose $3$-truncation is this one. Determinant functors on $\bar{Q}_\bullet(\C{T}_{\infty})$ are determinant functors on $\C{T}_{\infty}$. 

\item 
For an abelian category $\C{A}$, Neeman defined   a simplicial set $S_\bullet(\operatorname{Gr}^b\!\!\C{A})$ which can be thickened to an $S_\bullet$-category $\bar S_\bullet(\operatorname{Gr}^b\!\!\C{A})$ as in the previous examples, see \cite{kttcneeman}. The $3$-truncation of $\bar S_\bullet(\operatorname{Gr}^b\!\!\C{A})$ looks as follows:
$$\xymatrix@C=30pt{
\left\{
\!\!\!{\begin{array}{c}
\text{diagrams} \\
\text{like \eqref{braid}}
\end{array}}\!\!\!
\right\}
\ar@<-1.5ex>[r]_-{d_3}\ar@<-.5ex>[r]|<(.25){d_2}\ar@<.5ex>[r]|<(.1){d_1}\ar@<1.5ex>[r]^-{d_0}
&\ar@/_20pt/[l]|<(.35){s_2}\ar@<-1ex>@/_20pt/[l]|<(.55){s_1}\ar@<-2ex>@/_20pt/[l]_{s_0}
\left\{
\!\!\!{\begin{array}{c}
\text{long exact} \\
\text{seq. like \eqref{les}}
\end{array}}\!\!\!
\right\}\ar@<-1ex>[r]_-{d_2}\ar[r]|-{d_1}\ar@<1ex>[r]^-{d_0}&
\left\{
\!\!\!{\begin{array}{c}
\text{bounded graded} \\
\text{objects in }\C{A}
\end{array}}\!\!\!
\right\}
\ar@<-.5ex>[r]_-{d_1}\ar@<.5ex>[r]^-{d_0}\ar@/_20pt/[l]|{s_1}\ar@<-1ex>@/_20pt/[l]_{s_0}&\{0\},\ar@/_20pt/[l]_{s_0}}$$
where
\begin{equation}\label{les}
\cdots\r \X_{n}\st{f_{n}}\To \Y_{n}\st{i_{n}}\To \cone{f}_{n}\st{q_{n}}\To \X_{n+1}\r \cdots.
\end{equation}
\begin{equation}\label{braid}
\begin{array}[m]{c}\\[-10mm]
\vcenter{
\xymatrix@!=15pt@ru
{
&\dots\\
\cone{f}_{n-1} \ar[r]_-{\bar{g}_{n-1}} \ar@(rd,ld)[dr]_{q_{n-1}^f}  &\ar[d]^{q_{n-1}^{gf}} \cone{gf}_{n-1} \ar[r]_-{\bar{f}_{n-1}}&
\ar[d]^{q^g_{n-1}} \cone{g}_{n-1} \ar@(ru,lu)[dr]
\\&
\X_{n} \ar[r]_-{f_{n}} \ar@(rd,ld)[dr] &
\ar[d]^{g_{n}} \Y_{n} \ar[r]_-{i^f_{n}}&
\ar[d]^{\bar{g}_{n}} \cone{f}_{n} \ar@(ru,lu)[dr]^{q^f_{n}}
\\&&
\Z_{n} \ar[r]_-{i^{gf}_{n}} \ar@(rd,ld)[dr]_{i^g_{n}} &\ar[d]^{\bar{f}_{n}}
\cone{gf}_{n}
\ar[r]_-{q^{gf}_{n}}&
\ar[d]^{f_{n+1}}
\X_{n+1}
\ar@(ru,lu)[dr]
\\&&&
\cone{g}_{n}
\ar[r]_-{q^g_{n}}
\ar@(rd,ld)[dr] &
\ar[d]^{i^f_{n+1}}
\Y_{n+1}
\ar[r]_-{g_{n+1}}&
\ar[d]^{i^{gf}_{n+1}}
\Z_{n+1}
\ar@(ru,lu)[dr]^{i^g_{n+1}}
\\&&&&
\cone{f}_{n+1}
\ar[r]_-{\bar{g}_{n+1}}
&
\cone{gf}_{n+1}
\ar[r]_-{\bar{f}_{n+1}}&
\cone{g}_{n+1}
\\&&&&&\dots
}}
\\[-44mm]
\end{array}
\end{equation}
Faces and degeneracies are defined by,
\begin{equation*}
d_i\eqref{les}=\left\{\begin{array}{ll}
\cone{f}_*,& i=0;\\
\Y_*,& i=1;\\
\X_*,& i=2;
\end{array}\right.
\end{equation*}
\begin{equation*}
s_i(\X)=\left\{\begin{array}{ll}
\cdots\r0\r \X_n\st{1}\r \X_n\r 0\r\cdots,& i=0;\\
\cdots\r \X_n\st{1}\r \X_n\r 0\r \X_{n+1}\r\cdots,& i=1;
\end{array}\right.
\end{equation*}
 \begin{equation*}
 d_i\eqref{braid}=\left\{\begin{array}{ll}
\cdots\r \cone{f}_{n}\st{\bar{g}_{n}}\To \cone{gf}_{n}\st{\bar{f}_{n}}\To C_{n}^{g}\st{i_{n+1}^fq_{n}^{g}}\To \cone{f}_{n+1}\r \cdots,& i=0;\\
 \cdots\r \Y_{n}\st{g_{n}}\To \Z_{n}\st{i_{n}^{g}}\To C_{n}^{g}\st{q_{n}^{g}}\To \Y_{n+1}\r \cdots,& i=1;\\
 \cdots\r \X_{n}\st{g_nf_{n}}\To \Z_{n}\st{i_{n}^{gf}}\To C_{n}^{gf}\st{q_{n}^{gf}}\To \X_{n+1}\r \cdots,& i=2;\\
 \cdots\r \X_{n}\st{f_{n}}\To \Y_{n}\st{i_{n}^{f}}\To C_{n}^{f}\st{q_{n}^{f}}\To \X_{n+1}\r \cdots,& i=3.
 \end{array}\right.
 \end{equation*}
A \emph{graded determinant functor} on $\C A$ is a determinant functor on  ${\bar{S}_\bullet(\operatorname{Gr}^b\!\!\C{A})}$.

\item The unified approach to determinant functors in Definition~\ref{detsimpcat} allows us to define determinant functors for \emph{triangulated derivators}, and more generally for \emph{right pointed derivators}, using the terminology of \cite{ciscd}. Notice that these are called \emph{left pointed derivators} in \cite{sdckt1, sdckt2}.

Let $\mathbf{Cat}$ be  the $2$-category of small categories and $\mathbf{Dir}_{f}\subset \mathbf{Cat}$ the full sub-$2$-category of directed finite  categories, i.e.~those categories whose nerve has a finite number of non-degenerate simplices, e.g.~finite posets.
 The canonical example of derivator is defined from a Waldhausen category $\C{W}$ with cylinders whose weak equivalences satisfy the two-out-of-three axiom, for instance $\C{W}=C^b(\C{E})$. It is the contravariant $2$-functor 
\begin{align*}
\mathbb{D}_{\C{W}}\colon\mathbf{Dir}_{f}^\op&\To \mathbf{Cat},\\
J&\;\mapsto\;\ho(\C{W}^J),
\end{align*}
which takes a directed finite category $J$ to the homotopy category of $J$-indexed diagrams in~$\C{W}$. 
This derivator is regarded as an enhancement of the homotopy category, which arises as a special value of this $2$-functor $\mathbb{D}_{\C{W}}(*)=\ho(\C W)$.

For $\C{W}=C^b(\C{E})$, $\mathbb{D}^b(\C{E})=\mathbb{D}_{C^b(\C{E})}$ is
\begin{align*}
\mathbb{D}^b(\C{E})\colon\mathbf{Dir}_{f}^\op&\To \mathbf{Cat},\\
J&\;\mapsto\;D^b(\C{E}^J).
\end{align*}
In this case, $\mathbb{D}^b(\C{E})(*)=D^b(\C{E})$  is the bounded derived category of $\C E$.

In general, a right pointed derivator is a $2$-functor $\mathbb{D}\colon\mathbf{Dir}_{f}^\op\r\mathbf{Cat}$ satisfying   the formal properties of $\mathbb{D}_{\C{W}}$. Triangulated derivators are modeled after $\mathbb{D}^b(\C{E})$.

Garkusha defined in  \cite{sdckt1} 
a simplicial category $S_\bullet\mathbb{D}$. It is only reduced up to the existence of distinct zero objects. We can of course identify different zero objects through the unique isomorphisms between them. This transforms $S_\bullet\mathbb{D}$ into an $S_\bullet$-category.
We define a determinant functor on $\mathbb{D}$ to be a determinant functor on ${S_\bullet(\mathbb{D})}$. The interested reader may work out the explicit definition of determinant functors for right pointed derivators along the lines of Section \ref{sec1}. 


\end{enumerate}
\end{exm}

\begin{defn}
An \emph{$S_{\bullet}$-functor} is  a simplicial functor $f_\bullet\colon\C{C}_\bullet\r\C{C}'_\bullet$ between $S_{\bullet}$-categories 
preserving weak equivalences and coproducts. 
\end{defn}

\begin{rem}
Let  $f_\bullet\colon\C{C}_\bullet\r\C{C}'_\bullet$ be an $S_{\bullet}$-functor and $\det'\colon\C{C}'_\bullet\r\C{P}$ a determinant functor. The composite $$\det=\det'\circ f_1\colon\C{C}_\bullet\To \C{P}$$ is a determinant functor on $\C{C}_\bullet$ with $\det(\Delta)=\det'(f_2(\Delta))$ for any object $\Delta$ in $\C{C}_2$. Actually, it is enough to have  a $3$-truncated $S_{\bullet}$-functor  $f_{\leq 3}\colon\C{C}_{\leq 3}\r\C{C}'_{\leq 3}$, and even less, compare Remark \ref{evenless}. 
\end{rem}

\begin{exm}\label{lasrelaciones}
The following are examples of ($3$-truncated) $S_{\bullet}$-functors.

\begin{enumerate} 
\item Weak equivalences in a Waldhausen category $\C{W}$ project to isomorphisms in the homotopy category, so we have an $S_{\bullet}$-functor 
$${S}_\bullet(\C{W})\To  \operatorname{Ho}{S}_\bullet(\C{W}).$$

\item In a triangulated category $\C{T}$, any \exact\ triangle is virtual, any special octahedron is an ordinary octahedron, and any ordinary octahedron is virtual. This gives rise to $3$-truncated $S_{\bullet}$-functors
$$\bar{S}_{\leq 3}(^{d}\!\C{T})\To  \bar{S}_{\leq 3}(^{b}\!\C{T})\To  \bar{S}_{\leq 3}(^{v}\!\C{T}),$$
which are fully faithful and injective on objects. 
Actually, the composite is the truncation of an honest $S_{\bullet}$-functor defined in  \cite{kttcneeman},
$$\bar{S}_\bullet(^{d}\!\C{T})\To  \bar{S}_\bullet(^{v}\!\C{T}),$$ 
also fully faithful and injective on objects.

\item If $\C{T}$ is a triangulated category with a $t$-structure, the inclusion of the heart $\C{A}\subset\C{T}$ induces an $S_{\bullet}$-functor ${S}_\bullet(\C{A})\To  \bar{S}_\bullet(^{d}\!\C{T})$  fully faithful and injective on objects, see  \cite{kttcneeman}. 

\item A strongly triangulated category $\C{T}_{\infty}$ has an underlying triangulated structure. \Exact\ octahedra in $\C T_{\infty}$ are also ordinary octahedra. Therefore we have a $3$-truncated $S_{\bullet}$-functor which is  fully faithful and injective on objects,
$$\bar{Q}_{\leq 3}(\C{T}_{\infty})\To \bar{S}_{\leq 3}(^{b}\!\C{T}_{\infty}).$$

\item If $\C W$ is a Waldhausen category with cylinders satisfying the two out of three axiom, there is an $S_{\bullet}$-functor \cite{sdckt2,malt},
$$\ho S_{\bullet}(\C W)\To S_{\bullet}(\mathbb{D}_{\C W}).$$

\item Maltsiniotis indicated in \cite{cts} how a triangulated derivator $\mathbb{D}$ induces a strongly triangulated  structure on $\mathbb{D}(*)$. There is an $S_{\bullet}$-functor
$$S_\bullet(\mathbb{D})\To  \bar{Q}_\bullet(\mathbb{D}(*))$$ defined by using the canonical evaluation functors from $\mathbb{D}(J)$ to the category of functors $J\r \mathbb{D}(*)$.

\item If $\C E$ is an exact category, the inclusion of complexes concentrated in degree $0$,  $\C E\subset C^{b}(\C E)$, induces an 
$S_{\bullet}$-functor $$ S_{\bullet}(\C E)\To S_{\bullet}(C^{b}(\C E)).$$

\item Let $\C E$ be again an exact category. 
The 
$S_{\bullet}$-functor $$ S_{\bullet}(C^{b}(\C E))\To 
\bar{Q}_{\bullet}(D^{b}(\C E))$$
given by (1), (5) and (6) 
can be directly described in terms of the canonical functor $C^{b}(\C E)\r D^{b}(\C E)$, compare \cite{cts}.
\item Let $\C A$ be an abelian category, the inclusion of objects concentrated in degree $0$ is an $S_\bullet$-functor ${{S}_\bullet(\C{A})}\r {\bar{S}_\bullet(\operatorname{Gr}^b\!\!\C{A})}$.
\end{enumerate}
\end{exm}

\subsection{The groupoid of determinant functors} In this section we introduce universal determinant functor for $S_{\bullet}$-categories.

\begin{defn}\label{detgrd}
Let  $\C{C}_\bullet$ be an $S_{\bullet}$-category and $\C P$ a \Picardcategory. 
Given determinant functors $\det,\det'\colon \C{C}_\bullet\r\C P$, a \emph{morphism} $f\colon\det\r\det'$ is a natural
transformation between the underlying functors $\det,\det'\colon\operatorname{we}\C{C}_1\r\C{P}$ compatible with
the additivity data, i.e.~given an object $\Delta$ in $\C{C}_2$, the following diagram commutes:
$$\xymatrix@l{\det'(d_1\Delta)\ar@{<-}[r]^-{\det'(\Delta)}\ar@{<-}[d]_{f(d_1\Delta)}&
\det'(d_0\Delta)\otimes\det'(d_2\Delta)\ar@{<-}[d]^{f(d_0\Delta)\otimes f(d_2\Delta)}\\
\det(d_1\Delta)\ar@{<-}[r]_-{\det(\Delta)}&
\det(d_0\Delta)\otimes\det(d_2\Delta)}$$
Notice that all these morphisms are invertible since $\C P$ is a groupoid. 
The resulting \emph{groupoid of determinant functors} is denoted by $\operatorname{Det}(\C{C}_\bullet,\C P)$. 
\end{defn}

\begin{rem}
The category $\operatorname{Det}(\C{C}_\bullet,\C P)$ is itself a \Picardcategory. The tensor structure is given as follows. For any determinant functors $\det$, $\det'$, $\det''$, any object $\X$ in $\C C_{1}$, any morphism  $\alpha$ in $\operatorname{we}\C{C}_1$, and  any object
$\Delta$ in $\C{C}_2$, we define:
\begin{align*}
(\det\otimes {\det}')(\X)&=\det(\X)\otimes{\det}'(\X),\\
(\det\otimes{\det}')(\alpha)&=\det(\alpha)\otimes{\det}'(\alpha),\\
(\det\otimes{\det}')(\Delta)&=(\det(\Delta)\otimes{\det}'(\Delta))\circ ( 1\otimes\com\otimes 1). 
\end{align*}
The unit object is the \emph{constant determinant functor}, which sends all objects in $\C C_{1}$ to the tensor unit $I$ of $\C P$, all weak equivalences in $\C C_{1}$ to the identity on $I$, and all objects in $\C C_{2}$ to the isomorphism $I\otimes I\cong I$. Associativity, commutativity and unit constraints are pointwise defined by the corresponding constraints in $\C P$. 
This structure has already been considered in \cite[Proposition 1.13]{dfec} for determinant functors on exact categories.
\end{rem}

\begin{defn}\label{unidoni}
Let $\C{C}_\bullet$ be an $S_\bullet$-category. The composition of a determinant functor and a symmetric tensor functor is again a determinant functor,
$$\xymatrix{\C C_\bullet\ar[r]^-{\det}& \C P\ar[r]^-{f}&\C P'.}$$
The underlying functor $f\circ\det\colon\we\,\C C_1\r\C P'$ is the usual composition, and additivity data $(f\circ\det)(\Delta)$  are defined by
$$\xymatrix@C=20pt{
f(\det(d_0\Delta))\otimes f(\det(d_2\Delta))\ar[r]^-{\hpp}
&f(\det(d_0\Delta)\otimes\det(d_2\Delta)) \ar[rr]^-{f(\det(\Delta))}&&f(\det(d_1\Delta)).}$$
Moreover, a tensor natural transformation $\alpha\colon f\rr g$ as in the following diagram 
$$\xymatrix{\C C_\bullet\ar[r]^-{\det}& \C P\ar@/^15pt/[r]^-{f}_{}="a"\ar@/_15pt/[r]_-{g}^{}="b"&\C P'\ar@{=>}"a";"b"^\alpha}$$
induces a morphism of determinant functors $\alpha\circ\det \colon f\circ\det\rr g\circ\det$ defined as
the usual horizontal composition of functors and natural transformations. This composite  is compatible with additivity data in the sense of Definition \ref{detgrd}. 

All this shows that the \Picardcategory\ $\operatorname{Det}(\C{C}_\bullet,\C P)$ is $2$-functorial in $\C P$, i.e.~it defines a $2$-functor from the $2$-category of \Picardcategories\ to the $2$-category of groupoids. Actually, the target can be taken to be the $2$-category of \Picardcategories\ again, along the lines of the previous remark. In particular, a determinant functor $\det\colon\C C_\bullet\r\C P$ gives rise to a functor
\begin{equation}\label{cdet}
-\circ\det\colon \hom_c^\otimes(\C{P},\C P')\To \operatorname{Det}(\C{C}_\bullet,\C P').
\end{equation}

A determinant functor  $\det\colon\C{C}_\bullet\r V(\C{C}_\bullet)$ is \emph{universal} if 
$$-\circ\det\colon \hom_c^\otimes(V(\C{C}_\bullet),\C P)\To \operatorname{Det}(\C{C}_\bullet,\C P)$$ 
is an equivalence of categories for any \Picardcategory\ $\C P$. In this situation, we say that $V(\C{C}_\bullet)$ is the \emph{category of virtual objects} of $\C C$. This \Picardcategory\ is well defined up to equivalence, since it represents the $2$-functor $\operatorname{Det}(\C{C}_\bullet,-)$.
\end{defn}

\begin{rem}
Universal determinant functors on $S_\bullet$-categories can also be characterized by a $2$-categorical universal property along the lines of Definition \ref{uniwald}. In particular, the various notions of universal determinant functor introduced in Section \ref{sec1} are consistent with Definition \ref{unidoni} and Example \ref{mainexm}. The more abstract approach in Definition \ref{unidoni} will be helpful in the proof of the existence of universal determinant functors.
\end{rem}

For determinant functors with values in strict  \Picardcategories\ it is convenient to introduce also the following notion.

\begin{defn}
A determinant functor $\det\colon \C{C}_\bullet\r \C P$ with values in a strict \Picardcategory\ $\C P$ is \emph{strict} if it satisfies
\begin{align*}
        \det(s_0(*))&=I,&
                \det(s_0^2(*))&=1_{I}.
\end{align*}

We denote by $\operatorname{Det}_s(\C{C}_\bullet,\C P)$ the full subcategory of $\operatorname{Det}(\C{C}_\bullet,\C P)$ whose objects are the strict determinant functors. 
\end{defn}

\begin{lem}\label{strictdets}
The inclusion $\operatorname{Det}_s(\C{C}_\bullet,\C P)\subset\operatorname{Det}(\C{C}_\bullet,\C P)$ is an equivalence, natural in the strict \Picardcategory\ $\C P$.
\end{lem}

\begin{proof}
Let $\det\colon \C{C}_\bullet\r\C{P}$ be a determinant functor. Then we
can define a strict determinant functor $\det'$ by
$$
\det'=\det\otimes \det(s_0(*))^{-1}\colon \operatorname{we}\C{C}_1\r \C{P}
$$
and by
$$
\det'(\Delta)=\big(\det(\Delta)\otimes1_{\det(s_0(*))^{-1}}\big)\circ
\big(\det(s_0 d_0\Delta)^{-1}\otimes 1_{\det(s_0(*))^{-1}\otimes
\det'(d_2\Delta)}\big)
$$
for objects $\Delta$ in $\C{C}_2$. Moreover,
$\X\mapsto \det(s_0 \X)\otimes 1_{\det(s_0(*))^{-1}}$ defines a morphism of
determinant functors  $\det\r\det'$.
\end{proof}

\subsection{The existence of universal determinant functors}

In this section we will show that universal determinant functors
always exist, which is our main result. 
We will actually construct universal determinant functors by using presentations of 
stable quadratic modules.

\begin{defn}\label{univtarget}
Let $\C{C}_\bullet$ be an $S_{\bullet}$-category. We define the stable quadratic module $\D{D}_*(\C{C}_\bullet)$ by generators,
\begin{enumerate}
\renewcommand{\labelenumi}{(G{\arabic{enumi}})}
\item $[\X]$ for any object in $\C{C}_1$, in dimension $0$,
\item $[\X\st{\sim}\r \X']$ for any weak equivalence in $\C{C}_1$, in dimension $1$,
\item $[\Delta]$ for any object in $\C{C}_2$, in dimension $1$,
\end{enumerate}
and relations,
\begin{enumerate}
\renewcommand{\labelenumi}{(R{\arabic{enumi}})}
\item $\partial[\X\st{\sim}\r \X']=-[\X']+[\X]$,
\item $\partial[\Delta]=-[d_1\Delta]+[d_0\Delta]+[d_2\Delta]$,
\item $[s_0(*)]=0$ for the degenerate object of $\C{C}_1$,
\item $[\X\st{1}\r \X]=0$, for all identity morphisms in $\C{C}_1$,
\item $[s_0\X]=0=[s_1\X]$ for any object $\X$ in $\C{C}_1$,
\item for any pair of composable weak equivalences $\X\st{\sim}\r \Y \st{\sim}\r \Z$ in $\C{C}_1$,
$$[\X\st{\sim}\r \Z]\;\;=\;\;[\Y\st{\sim}\r \Z]+[\X\st{\sim}\r \Y],$$
\item for any weak equivalence $\Phi\colon\Delta \st{\sim}\r \Delta'$ in $\C{C}_2$,
$$[d_2\Phi]+[d_0\Phi]^{[d_2\Delta]}\;\;=\;\;-[\Delta']+[d_1\Phi]+[\Delta],$$
\item for any object $\Theta$ in $\C{C}_3$,
$$[d_1\Theta]+[d_3\Theta]\;\;=\;\;[d_2\Theta]+[d_0\Theta]^{[d_2d_3\Theta]},$$
\item for any two objects $X$ and $Y$ in $\C{C}_1$,
$$\grupo{[\X],[\Y]}\;\;=\;\;-[s_0\X\sqcup s_1\Y]+[s_1\X\sqcup s_0\Y].$$
\end{enumerate}
\end{defn}

\begin{rem}\label{lessrel}
This is not a minimal presentation, compare~\cite[Remark 1.4]{k1wc}, but it is the most
intuitive. Relation (R4) follows from (R6). 
Relation (R3) follows from (R5),
$$0=\partial[s_{1}\X]=-[\X]+[\X]+[s_0(*)]=[s_0(*)].$$
Relation (R5) is equivalent to
imposing $[s_0^2(*)]=0$ for the degenerate object of $\C{C}_2$. Indeed, applying (R8) to $s_{0}^{2}\X$ and $s_{1}^{2}\X$, respectively, we obtain,
\begin{align*}
[s_{0}\X]+[s_0^2(*)]&=[s_{0}\X]+[s_{0}\X]^{[s_0(*)]},\\
[s_{1}\X]+[s_{1}\X]&=[s_{1}\X]+[s_0^2(*)]^{[\X]}.
\end{align*}
\end{rem}

\begin{rem}\label{lulu}

The stable quadratic module $\D{D}_*(\C{C}_\bullet)$ is functorial with respect to $S_{\bullet}$-functors,
\begin{align*}
\D{D}_0(f_{\bullet})\colon\D{D}_0(\C{C}_\bullet)&\To  \D{D}_0(\C{C}_\bullet'),&\D{D}_1(f_{\bullet})\colon\D{D}_1(\C{C}_\bullet)&\To  \D{D}_1(\C{C}_\bullet'),\\
[\X]&\;\mapsto \; [f_1(\X)],&[\phi\colon\X\st{\sim}\r \X']&\;\mapsto \; [f_{1}(\phi)\colon f_1(\X)\st{\sim}\r f_1(\X')],\\
&&[\Delta]&\;\mapsto \;[f_2(\Delta)].
\end{align*}
 Moreover, it is $2$-functorial with respect to simplicial natural weak equivalences 
 $\alpha_{\bullet}\colon f_{\bullet}\Rightarrow g_{\bullet}$ between $S_{\bullet}$-functors
$f_\bullet,g_{\bullet}\colon\C{C}_\bullet\r\C{C}'_\bullet$, i.e.~simplicial natural transformations taking values in the subcategories of weak equivalences,
\begin{align*}
\D{D}_*(\alpha_{\bullet})\colon\D{D}_0(\C{C}_\bullet)&\To  \D{D}_1(\C{C}_\bullet'),\\
[\X]&\;\mapsto [\alpha_{1}(\X)\colon f_{1}(\X)\st{\sim}\r g_{1}(\X)].
\end{align*}
Actually, it is $2$-functorial at the $3$-truncated level. 

If we evaluate $\D{D}_*(\C{C}_\bullet)$ at the $S_\bullet$-categories in Example \ref{mainexm} (1--6) we obtain the stable quadratic modules in Definition \ref{univtargetwald}.
\end{rem}

\begin{thm}\label{main1}
There is a universal determinant functor $\det\colon \C{C}_\bullet\r \Gamma\D{D}_*(\C{C}_\bullet)$ defined by 
\begin{itemize}
 \item $\det(\X)=[\X]$ for any object in $\C C_1$,
\item $\det(\X\st{\sim}\r \X')=([\X'],[\X\st{\sim}\r\X'])$ for any weak equivalence in $\C C_1$,
\item $\det(\Delta)=([d_1\Delta],[\Delta])$ for any object in $\C C_2$.
\end{itemize}
This determinant functor is strict. 
Moreover, for any stable quadratic module $C_{*}$, the functor
\begin{align*}
\hom(\D{D}_*(\C{C}_\bullet),C_*)&\To  \operatorname{Det}_s(\C{C}_\bullet,\Gamma C_*)\\
\varphi&\;\mapsto\;(\Gamma\varphi)\circ\det
\end{align*}
is an isomorphism of groupoids. Here the source is a morphism groupoid in the $2$-category of stable quadratic modules.
\end{thm}

\begin{proof}
A strict determinant functor $\det\colon \C C_{\bullet}\r \Gamma C_*$ sends an object $X$ in $\C C_{1}$, a  weak equivalence $f\colon X\st{\sim}\r X'$ in $\C C_{1}$, and an object $\Delta$ in $\C C_{2}$
to elements 
\begin{align}
\nonumber\det (X)&\in C_0,\\
\label{lodeterminan}\det(f)=(\det(X'),\{f\})&\in C_0\ltimes C_1,\\
\nonumber\det(\Delta)=(\det(d_{1}\Delta),\{\Delta\})&\in C_0\ltimes C_1.
\end{align}
The first coordinate of $\det(f)$ is $\det(X')$ since it is the target. The source is
$\det(X')+\partial\{f\}=\det(X)$, i.e.
\begin{align*}
\partial\{f\}=-\det(X')+\det(X).
\end{align*}
For the same reasons, the first coordinate of $\det(\Delta)$ is $\det(d_{1}\Delta)$ and 
\begin{align*}
\partial\{\Delta\}=-\det(d_{1}\Delta)+\det(d_{0}\Delta)+\det(d_{2}\Delta).
\end{align*}
The functor $\det$ preserves identities, $\det(1_{X})=1_{\det(X)}=(\det(X),0)$, i.e.
\begin{align*}
\{1_X\}&=0.
\end{align*}
Moreover, it also preserves compositions. Since
\begin{align*}
\det(gf\colon X\st{f}\r Y\st{g}\r Z)&=(\det(Z),\{gf\}),\\
\det(g)\circ\det(f)&=(\det(Z),\{g\})\circ (\det(Y),\{f\})=(\det(Z),\{g\}+\{f\}),
\end{align*}
then
\begin{align*}
 \{gf\}&=\{g\}+\{f\}.
\end{align*}

Naturality of additivity data with respect to weak equivalences $\Phi\colon\Delta\st{\sim}\r\Delta'$ in $\C C_2$ says that the following diagram in $\Gamma C_*$ must commute,
$$\xymatrix@C=120pt@R=35pt{\det(d_0\Delta)+\det(d_2\Delta)\ar[r]^-{\det(\Delta)=(\det(d_1\Delta),\{\Delta\})}
\ar[d]^{
\begin{array}{l}
\scriptstyle\det(d_0\Phi)+\det(d_2\Phi)\\[-1pt]
\scriptstyle=(\det(d_0\Delta'),\{d_0\Phi\})+(\det(d_2\Delta'),\{d_2\Phi\})\\[-1pt]
\scriptstyle=(\det(d_0\Delta')+\det(d_2\Delta'),\{d_0\Phi\}^{\det(d_2\Delta')}+\{d_2\Phi\})
\end{array}
}
&\det(d_1\Delta)\ar[d]^{\det(d_1\Phi)=(\det(d_1\Delta'),\{d_1\Phi\})}\\
\det(d_0\Delta')+\det(d_2\Delta')\ar[r]_-{\det(\Delta')=(\det(d_1\Delta'),\{\Delta'\})}&\det(d_1\Delta')}$$
i.e.~
\begin{align*}
\{d_1\Phi\} +\{\Delta\}&=\{\Delta'\}+\{d_0\Phi\}^{\det(d_2\Delta')}+\{d_2\Phi\}\\
&=\{\Delta'\}+\{d_2\Phi\}+\{d_0\Phi\}^{\det(d_2\Delta')+\partial\{d_2\Phi\}}\\
&=\{\Delta'\}+\{d_2\Phi\} +\{d_0\Phi\}^{\det(d_2\Delta)}.
\end{align*}

The \Picardgroupoid\ $\Gamma  C_{*}$ is strict, in particular associativity constraints are identities. Hence the associativity axiom says that, for any object $\Theta$ in $\C C_{3}$, the following diagram commutes
$$\xymatrix@C=15pt@R=35pt{
\det(d_0d_1\Theta)+\det(d_1d_3\Theta)\ar[rr]^-{
\begin{array}{l}
\scriptstyle\det(d_1\Theta)\\[-3pt]
\scriptstyle=(\det(d_1d_2\Theta),\{d_1\Theta\})
\end{array}}&&\det(d_1d_2\Theta)\\
\det(d_0d_1\Theta)+\det(d_0d_3\Theta)+\det(d_2d_3\Theta)\ar[u]_{
\begin{array}{l}
\scriptstyle1_{\det(d_0d_1\Theta)}+\det(d_3\Theta)\\[-3pt]
\scriptstyle=(\det(d_0d_1\Theta),0)+(\det(d_1d_3\Theta),\{d_3\Theta\})
\\[-3pt]
\scriptstyle=(\det(d_0d_1\Theta)+\det(d_1d_3\Theta),\{d_3\Theta\})
\end{array}
}\ar[rr]_<(.38){
\begin{array}{l}\\[-7pt]
\scriptstyle\det(d_0\Theta)+1_{\det(d_2d_3\Theta)}\\[-3pt]
\scriptstyle=(\det(d_0d_2\Theta),\{d_0\Theta\})+(\det(d_2d_3\Theta),0)\\[-3pt]
\scriptstyle=(\det(d_0d_2\Theta)+\det(d_2d_3\Theta),\{d_0\Theta\}^{\det(d_2d_3\Theta)})
\end{array}
}&&\det(d_0d_2\Theta)+\det(d_2d_3\Theta)
\ar[u]_{
\begin{array}{l}
\scriptstyle
\det(d_2\Theta)\\[-3pt]
\scriptstyle=(\det(d_{1}d_2\Theta),\{d_2\Theta\})
\end{array}}}$$
i.e.
\begin{align*}
\{d_{1}\Theta\}+\{d_{3}\Theta\}&=\{d_{2}\Theta\}+\{d_{0}\Theta\}^{\det(d_{2}d_{3}\Theta)}.
\end{align*}

The commutativity axiom says that the following diagram commutes for any pair of objects $X$ and $Y$ in $\C C_{1}$,
$$\xymatrix@C=35pt{&\det( X\sqcup  Y)\ar@{<-}[rd]^{\qquad
\begin{array}{l}
\scriptstyle\det(s_0 X\sqcup s_1 Y)=\\[-3pt]
\scriptstyle(\det(X\sqcup Y),\{s_0 X\sqcup s_1 Y\})
\end{array}
}\ar@{<-}[ld]_-{\begin{array}{l}
\scriptstyle\det(s_1 X\sqcup s_0 Y)=\\[-3pt]
\scriptstyle(\det(X\sqcup Y),\{s_1 X\sqcup s_0 Y\})
\end{array}\qquad}&\\\det( Y)+\det( X)\ar[rr]^-{\com}_-{(\det( X)+\det( Y),\langle\det( X),\det( Y)\rangle)}&&\det( X)+\det( Y)}$$
i.e.
\begin{align*}
\{s_0 X\sqcup s_1 Y\}+\langle\det( X),\det( Y)\rangle&=\{s_1 X\sqcup s_0 Y\}.
\end{align*}

Being strict means that $\det(s_{0}(*))=I=0$ and $\det(s^{2}_{0}(*))=1_{I}=(0,0)$, i.e.
\begin{align*}
\det(s_{0}(*))&=0,&\{s^{2}_{0}(*)\}&=0.
\end{align*}

Conversely, any choice of elements
\begin{align*}
\det (X)&\in C_0,&
\{f\}&\in C_1,&
\{\Delta\}&\in C_1,
\end{align*}
satisfying the previous equations
yields a strict determinant functor $\det\colon \C C_{\bullet}\r \Gamma C_*$ defined by \eqref{lodeterminan}.

This, combined with the presentation of $\D{D}_*(\C{C}_\bullet)$ in Definition \ref{univtarget} and the alternative set of relations in Remark \ref{lessrel},  has two important consequences: the formulas in the statement define a strict determinant functor $\det\colon \C{C}_\bullet\r \Gamma\D{D}_*(\C{C}_\bullet)$, and any strict determinant functor $\det'\colon  \C{C}_{\bullet}\r \Gamma C_{*}$ factors uniquely as
$$\det'\colon  \C{C}_{\bullet}\st{\det}\To  \Gamma\D{D}_*(\C{C}_\bullet)\st{\Gamma \varphi}\To \Gamma C_{*},$$
where $\varphi\colon \D{D}_*(\C{C}_\bullet)\r C_{*}$ is a morphism of stable quadratic modules. Hence, it is only left to check that $\det\colon \C{C}_\bullet\r \Gamma\D{D}_*(\C{C}_\bullet)$ is universal. 

Let $\C P$ be any \Picardgroupoid. By Corollary \ref{pg0sq}, there exists a $0$-free stable quadratic module $C_{*}$ and an equivalence $f\colon \Gamma C_{*}\st{\sim}\r\C P$ in the $2$-category of \Picardgroupoids. Consider the following commutative diagram
$$\xymatrix@C=30pt{\hom(\D{D}_*(\C{C}_\bullet), C_*)\ar[r]^-{\cong}_-{(\Gamma-)\circ\det}
\ar[d]^{\sim}_{
\text{Corollary \ref{pg0sq}}
}& \operatorname{Det}_s(\C{C}_\bullet,\Gamma C_*)\ar[dd]_{\sim}^{\text{Lemma \ref{strictdets}}}\\
\hom_{c}^\otimes(\Gamma\D{D}_*(\C{C}_\bullet),\Gamma C_*)\ar[d]^{\sim}_{f\circ-}&\\
\hom_{c}^\otimes(\Gamma\D{D}_*(\C{C}_\bullet),\C P)\ar[r]^-{-\circ\det}& \operatorname{Det}(\C{C}_\bullet,\Gamma C_*)}$$
The upper horizontal arrow is an isomorphism. This is the last part of the statement, that we have just checked. Two vertical arrows are equivalences by the results indicated in the labels. Moreover, $f\circ-$ is an equivalence of categories since $f$ is an equivalence of \Picardgroupoids. Thus the lower horizontal arrow is also an equivalence.

\end{proof}

\subsection{Non-commutative determinant functors}\label{ncdf}

In \cite{dc}, Deligne also considers determinant functors into categorical groups that are not necessarily symmetric. Of course, one has to omit the commutativity axiom in Definition~\ref{detsimpcat} if one chooses to work in this context. We will call those determinant functors \emph{non-commutative determinant functors}. However, as Deligne already noticed, it turns out that this notion is not essentially more general that the theory of commutative determinant functors considered above.



We consider the left adjoint of the functor sending a crossed module $C_{*}$ to the pair of sets $(C_{0},C_{1})$. Objects in the image of this left adjoint are said to be \emph{free}.

Let $\grupo{E}$ denote the free group on a set $E$. The free crossed module $F_{*}^{c}(E_{0},E_{1})$ on a pair of sets $(E_{0},E_{1})$ is defined as  follows: $F_{0}^{c}(E_{0},E_{1})=\grupo{E_{0}\sqcup E_{1}}$ is a free group, $F_{1}^{c}(E_{0},E_{1})=\ker p$ is the kernel of the homomorphism,
\begin{align*}
\grupo{E_{0}\sqcup E_{1}}&\st{p}\onto \grupo{E_{0}},&
E_{0}\ni e_{0}&\mapsto e_{0},&
E_{1}\ni e_{1}&\mapsto 0,
\end{align*}
the homomorphism $\partial \colon F_{1}^{c}(E_{0},E_{1})\hookrightarrow F_{0}^{c}(E_{0},E_{1})$ is the inclusion, and $F_{0}^{c}(E_{0},E_{1})$ acts on $F_{1}^{c}(E_{0},E_{1})$ by conjugation. The universal property of a free crossed module holds since $F_{1}^{c}(E_{0},E_{1})$ is freely generated as a group by the conjugates,
$$e_{1}^{c_{0}}=-c_{0}+e_{1}+c_{0},\qquad e_{1}\in E_{1},c_{0}\in\grupo{E_{0}}.$$
 
 
Given two sets of relations $R_{i}\subset F_{i}^{c}(E_{0},E_{1})$, $i=0,1$, the crossed module $C_{*}$ with generators $(E_{0},E_{1})$ and relations $(R_{0},R_{1})$ is defined as follows: $C_{0}$ is the quotient of $F_{0}^{c}(E_{0},E_{1})$ by the normal subgroup $N_{0}$ generated by $R_{0}\cup \partial R_{1}$, and $C_{1}$ is the quotient of $F_{1}^{c}(E_{0},E_{1})$ by the normal subgroup generated by,
\begin{align*}
r_{1}^{c_{0}},&\quad r_{1}\in R_{1}, c_{0}\in C_{0};&
-c_{1}+c_{1}^{n_{0}},&\quad c_{1}\in C_{1}, n_{0}\in N_{0}.
\end{align*}
The action of $C_{0}$ on $C_{1}$ and the homomorphism $\partial\colon C_{1}\r C_{0}$ are defined so that the natural projection $F_{*}^{c}(E_{0},E_{1})\onto C_{*}$ is a morphism of crossed modules.

Crossed modules defined by a presentation satisfy the obvious universal property.

\begin{defn}
Given an $S_{\bullet}$-category $\C{C}_\bullet$  we define $\mathcal{D}'_*(\C{C}_\bullet)$ as  the crossed module presented by generators (G1--3) and  relations (R1--8) as in Definition~\ref{univtarget}.  
\end{defn}

This crossed module is $0$-free. Indeed, $\mathcal{D}'_0(\C{C}_\bullet)$ is the free group generated by the objects of $\C C_{1}$ different from $s_{0}(*)$.

Universal non-commutative determinant functors are defined as in Definition \ref{unidoni},  taking values in categorical groups instead.

\begin{thm}\label{main1prima}
There exists a universal non-commutative determinant functor $\det\colon \C{C}_\bullet\r \Gamma\D{D}'_*(\C{C}_\bullet)$ defined by 
\begin{itemize}
 \item $\det(\X)=[\X]$ for any object in $\C C_1$,
\item $\det(\X\st{\sim}\r \X')=([\X'],[\X\st{\sim}\r\X'])$ for any weak equivalence in $\C C_1$,
\item $\det(\Delta)=([d_1\Delta],[\Delta])$ for any object in $\C C_2$.
\end{itemize}
This non-commutative determinant functor is strict. 
Moreover, for any crossed module $C_{*}$, the functor
\begin{align*}
\hom(\D{D}'_*(\C{C}_\bullet),C_*)&\To  \operatorname{Det}_s(\C{C}_\bullet,\Gamma C_*)\\
\varphi&\;\mapsto\;(\Gamma\varphi)\circ\det
\end{align*}
is an isomorphism of groupoids. Here the source is a morphism groupoid in the $2$-category of crossed modules.
\end{thm}

This theorem can be proved as Theorem \ref{main1}, mutatis mutandis. 

\begin{prop}\label{malteprop}
There exists a unique map
$$
\langle\cdot,\cdot\rangle \colon\mathcal{D}'_*(\C{C}_\bullet)\times\mathcal{D}'_*(\C{C}_\bullet)\r \mathcal{D}'_*(\C{C}_\bullet)
$$
such that
\begin{enumerate}
 \item $\langle [\X],[\Y]\rangle =-[s_0 \X \sqcup s_1 \Y ]+[s_1 \X \sqcup s_0 \Y ]$ for any two objects $\X, \Y$ in $\C{C}_1$.
 \item $(\mathcal{D}'_*(\C{C}_\bullet),\langle \cdot,\cdot\rangle )$ is a reduced $2$-module.
\end{enumerate}
Moreover, this map satisfies
$$
\langle a,b\rangle +\langle b,a\rangle =0
$$
for any $a,b\in\mathcal{D}'_0(\C{C}_\bullet)$, i.\,e.\, $(\mathcal{D}'_*(\C{C}_{\bullet}),\langle \cdot,\cdot\rangle )$ is a stable $2$-module.
\end{prop}

\begin{proof} 
We use the same argument as in \cite[Lemma~2.2.3]{malte}. 
The relations for the objects $s_i s_j(\X)$ in $\C{C}_3$ imply $\langle [\X],[s_0(*)]\rangle =\langle [s_0(*)],[\X]\rangle =0$ for any object in $\C{C}_1$. Recall that $[s_0(*)]=0$. 
Since the group $\mathcal{D}'_0(\C{C}_{\bullet})$ is the free group over the set $E$ of objects of $\C{C}_1$ minus the degenerate object $s_0(*)$, an induction over the reduced word length of the two arguments shows that the map
$$
E\times E\r \mathcal{D}'_1(\C{C}_\bullet),\qquad (\X,\Y)\mapsto -[s_0 \X \sqcup s_1 \Y ]+[s_1 \X \sqcup s_0 \Y ]
$$
extends in a unique way to a map
$$
\langle \cdot,\cdot\rangle \colon\mathcal{D}'_*(\C{C}_\bullet)\times\mathcal{D}'_*(\C{C}_\bullet)\r \mathcal{D}'_*(\C{C}_\bullet)
$$
satisfying
\begin{enumerate}
 \item $\langle c,c'+c''\rangle =\langle c,c'\rangle ^{c''}+\langle c,c''\rangle $,
 \item $\langle c+c',c''\rangle =\langle c',c''\rangle +\langle c,c''\rangle ^{c'}$.
\end{enumerate}

It remains to show that $(\mathcal{D}'_*(\C{C}_{\bullet}),\langle \cdot,\cdot\rangle )$ is a stable $2$-module. For this, it suffices to check the axioms $(3)$, $(4)$, $(6)$, and $(8)$ in Section~\ref{crossedmodules}. Axioms $(3)$ and $(6)$ are immediate from the definition of $\langle \cdot,\cdot\rangle $.

We verify axiom $(8)$. Let $\X$ and $\Y$ be objects of $\C{C}_1$. Given two coproducts $s_1 \X \sqcup s_0 \Y$ and $ s_{0}\Y\sqcup s_{1}\X$ their universal property yields a unique isomorphism fitting into the following commutative diagram,
$$\xymatrix{
s_{1}\X\ar[r]\ar@{=}[d]&s_1 \X \sqcup s_0 \Y\ar[d]^{\sim}&s_0 \Y\ar[l]\ar@{=}[d]\\
s_{1}\X\ar[r]&s_{0}\Y\sqcup s_{1}\X&s_0 \Y\ar[l]}$$
where the horizontal arrows are the inclusions of the factors.
This isomorphism and the corresponding one after exchanging $\X$ and $\Y$ yield the following relations,
\begin{align*}
&\hspace{-13pt} [d_2(s_1 \X \sqcup s_0 \Y \st{\sim}\to s_{0}\Y\sqcup s_{1}\X)]+ [d_0(s_1 \X \sqcup s_0 \Y \st{\sim}\to s_{0}\Y\sqcup s_{1}\X)]^{[\X]}\\
={}&-[s_0 \Y \sqcup s_1 \X ]+[d_1 (s_1 \X \sqcup s_0 \Y \st{\sim}\to s_{0}\Y\sqcup s_{1}\X)]+ [s_1 \X \sqcup s_0 \Y],\\
&\hspace{-13pt} [d_2(s_1 \Y \sqcup s_0 \X \st{\sim}\to s_{0}\X\sqcup s_{1}\Y)]+ [d_0(s_1 \Y \sqcup s_0 \X \st{\sim}\to s_{0}\X\sqcup s_{1}\Y)]^{[\Y]}\\
={}&-[s_0 \X \sqcup s_1 \Y ]+[d_1 (s_1 \Y \sqcup s_0 \X \st{\sim}\to s_{0}\X\sqcup s_{1}\Y)]+ [s_1 \Y \sqcup s_0 \X].
\end{align*}
Moreover, we have
$$
d_i (s_1 \X \sqcup s_0 \Y \st{\sim}\To  s_{0}\Y\sqcup s_{1}\X) =
\left\{
\begin{array}{ll}
1_{Y},&i=0;\\
\X\sqcup \Y\st{\sim}\To  \Y\sqcup \X,&i=1;\\
1_{X},&i=2;\\
\end{array}\right.
$$
$$
d_i (s_1 \Y \sqcup s_0 \X \st{\sim}\To  s_{0}\X\sqcup s_{1}\Y) =
\left\{
\begin{array}{ll}
1_{X},&i=0;\\
\Y\sqcup \X\st{\sim}\To  \X\sqcup \Y,&i=1;\\
1_{Y},&i=2;\\
\end{array}\right.
$$
and hence,
\begin{align*}
\langle [\X],[\Y]\rangle +\langle [\Y],[\X]\rangle &= -([\X\sqcup \Y\st{\sim}\To  \Y\sqcup \X]+[\Y\sqcup \X\st{\sim}\To  \X\sqcup \Y])^{\partial [s_1Y\sqcup s_0X]}\\
&=-[1_{\Y\sqcup \X}]^{\partial [s_1Y\sqcup s_0X]}=0.
\end{align*}
By induction it follows that $\langle c,c'\rangle +\langle c',c\rangle =0$ for any pair of elements $c$, $c'$ in $\mathcal{D}'_0(\C{C}_{\bullet})$.

Finally, we verify axiom $(4)$. Since both sides of the axiom define operations of $\mathcal{D}'_0(\C{C})$ on $\mathcal{D}'_1(\C{C})$, it suffices to check the relation for the action of an object $U$ of $\C{C}_1$ on a weak equivalence $\alpha\colon \X\r \X'$ in $\C{C}_1$ and on an object $\Delta$ in $\C{C}_2$, respectively. The weak equivalences $s_1 \alpha \sqcup s_0 1_U $ and $s_0 \alpha \sqcup s_1 1_U $ in $\C{C}_2$ imply
\begin{align*}
[s_1 \X' \sqcup s_0 U ]+[\alpha]&= [\alpha\sqcup 1_U]+[s_1 \X \sqcup s_0 U ],\\
[s_0\X'\sqcup s_1U]+[\alpha]^{[U]}&=[\alpha\sqcup 1_U]+[s_0 \X \sqcup s_1 U ],
\end{align*}
and hence,
\begin{align*}
[\alpha]^{[U]}&=\langle [\X'],[U]\rangle +[\alpha]+\langle [U],[\X]\rangle\\& =[\alpha]+\langle [U],-[\X']\rangle ^{[\X]}+\langle [U],[\X]\rangle =[\alpha]+\langle [U],\partial[\alpha]\rangle .
\end{align*}
The objects $s_0(\Delta)\sqcup s_1 s_1(U)$, $s_1(\Delta)\sqcup s_0 s_1(U)$, 
and $s_2(\Delta)\sqcup s_1 s_0(U)$ in $\C{C}_3$ imply the relations
\begin{align*}
[\Delta\sqcup s_1 U]+ [s_0 d_2\Delta \sqcup s_1 U ]&=[s_0 d_1\Delta \sqcup s_1 U ]+[\Delta]^{[U]},\\
[\Delta\sqcup s_1 U ]+ [s_1 d_2\Delta\sqcup s_0 U ]&=[\Delta\sqcup s_0 U ]+[s_0 d_0\Delta \sqcup s_1 U ]^{[d_2\Delta]},\\
[s_1 d_1\Delta \sqcup s_0 U ]+ [\Delta]&=[\Delta\sqcup s_0 U ]+[s_1 d_0\Delta \sqcup s_0 U ]^{[d_2\Delta]}.
\end{align*}
From these, one deduces easily the relation
$$
[\Delta]^{[U]}=[\Delta]+\langle [U],\partial [\Delta]\rangle .
$$
\end{proof}

\begin{cor}
The morphism of stable $2$-modules
\begin{align*}
\D{D}'_{*}(\C{C}_{\bullet})&\To  \D{D}_{*}(\C{C}_{\bullet}),\\
[X]&\;\mapsto\;[X],\\
[X\st{\sim}\r X']&\;\mapsto\;[X\st{\sim}\r X'],\\
[\Delta]&\;\mapsto\;[\Delta],
\end{align*}
is a quasi-isomorphism.
\end{cor}

\begin{proof}
Let $C_{*}$ be a stable quadratic module. 
A morphism of stable $2$-modules $\varphi\colon \D{D}'_{*}(\C{C}_{\bullet})\r C_{*}$ is the same as a morphism between the underlying crossed modules such that $\grupo{\varphi[X],\varphi[Y]}=-\varphi[s_{0}X\sqcup s_{1}Y]+\varphi[s_{1}X\sqcup s_{0}Y]$ for any two objects $X$ and $Y$ in $\C C_{1}$. This, together with the crossed module presentation of $\D{D}'_{*}(\C{C}_{\bullet})$, shows that the morphism in the statement is well defined, and moreover, any  morphism of stable $2$-modules $\varphi\colon \D{D}'_{*}(\C{C}_{\bullet})\r C_{*}$ factors uniquely through the morphism in the stament. Therefore, 
 the stable quadratic module obtained as the reflection of the stable $2$-module $\D{D}'_{*}(\C{C}_{\bullet})$ in the sense of \cite[Lemma 4.18]{1tk}, is $\D{D}_{*}(\C{C}_{\bullet})$, and the unit of the reflection is the morphism in the statement. Hence this morphism is a quasi-isomorphism by \cite[Remark 4.21]{1tk}.
\end{proof}

\begin{cor}\label{tambien}
The determinant functor $\det\colon \C{C}_\bullet\r \Gamma\D{D}_*(\C{C}_\bullet)$ in Theorem \ref{main1} is also universal among non-commutative determinant functors.
\end{cor}

\subsection{The connection with homotopy theory}\label{torremolinos}

In this section we consider the homotopy type of $|\we\, \C C_\bullet|$, the \emph{geometric realization} of the simplicial subcategory of weak equivalences in an $S_\bullet$-category $\C C_\bullet$. This space is the geometric realization of a simplicial set: the diagonal of the bisimplicial nerve of $\we\, \C C_\bullet$. The bisimplicial nerve of a simplicial category is obtained by applying degreewise the nerve functor from categories to simplicial sets.

Notice that the space $|\we\, \C C_\bullet|$ is connected. Moreover, it is reduced as a $CW$-complex, i.e.~it has only one vertex, since $\we\,\C C_0=\C C_0=*$ is the terminal category.

The categories $\C C_n$ have finite coproducts, which are preserved by face and degeneracy operators, and finite coproducts of weak equivalences are weak equivalences. Hence, we can enhance $|\we\, \C C_\bullet|$ to a $\Gamma$-space $A$ with $A(\mathbf{1})=|\we\, \C C_\bullet|$, compare \cite[\S2]{cct}. This $\Gamma$-space yields an $\Omega$-spectrum $K(\C C_{\bullet})$, with underlying sequence of spaces 
$$\Omega A(\mathbf{1}),A(\mathbf{1}),BA(\mathbf{1}),B^2A(\mathbf{1}),\dots$$
Indeed, it is enough to observe that $A(\mathbf{1})$ is the loop space of $BA(\mathbf{1})$, since $A(\mathbf{1})=|\we\, \C C_\bullet|$ is connected, see \cite[Proposition 1.4 and the following note]{cct}.

\begin{exm}\label{mainexmk}
The spectra $K(\C C_{\bullet})$ of the $S_{\bullet}$-categories in Example \ref{mainexm} are:

\begin{enumerate}
\item The $K$-theory spectrum $K(\C W)$ \cite{akttsI} of a Waldhausen category $\C W$.

\item The derived $K$-theory spectrum of $\C W$, $DK(\C W)$ \cite{sdckt2,malt}.

\addtocounter{enumi}{1}



\item Neeman's special $K$-theory spectrum $K(^{d}\C T)$ of a triangulated category $\C T$ \cite{kttcneeman}. This follows from the fact  the inclusion of objects ${S}_\bullet(^{d}\!\C{T})\subset\iso{\bar{S}_\bullet(^{d}\!\C{T})}$ induces a homotopy equivalence on geometric realizations, compare \cite[corollary to Lemma 1.4.1]{akts}.

\item Neeman's virtual $K$-theory spectrum $K(^{v}\C T)$ \cite{kttcneeman}, by the same reason as above.

\item  Maltsiniotis's $K$-theory spectrum $K(^{s}\C T_{\infty})$ of a strongly triangulated category  $\C T_{\infty}$ \cite{cts}, again by the argument of the two previous  cases.

\item Neeman's graded $K$-theory spectrum $K(\operatorname{Gr}^b\!\!\C{A})$ of an abelian category $\C A$ \cite{kttcneeman}.

\item Garkusha's $K$-theory spectrum $K(\mathbb{D})$ of a right pointed derivator~$\mathbb{D}$ \cite[\S5]{sdckt1}, see also \cite{ktdt}.
\end{enumerate}

In some of the previous references, the authors only care about the $K$-theory space of the corresponding $S_{\bullet}$-category $\C C_{\bullet}$, i.e.~$\Omega|\we\,\C C_{\bullet}|$. We can always enhance them to spectra as indicated above.


\end{exm}

Consider the functor $\lambda_{0}$ in Lemma \ref{equi}. 

\begin{thm}\label{cosa}
For any $S_{\bullet}$-category $\C C_{\bullet}$, there is a natural isomorphism $\D{D}_*(\C{C}_\bullet)\cong \lambda_0K(\C C_{\bullet})$ in $\ho\mathbf{squad}$.
\end{thm}

This theorem is a straightforward generalization of \cite[Theorem 1.7]{1tk}. Exactly the same proof works with the appropriate changes in notation.

\begin{exm}\label{lasrelaciones2}
The $S_{\bullet}$-functors in Example \ref{lasrelaciones} induce comparison maps of $K$-theory spectra: 
\begin{enumerate}
\item $K(\C W)\r DK(\C W)$ for any Waldhausen category $\C W$ \cite{sdckt2,malt}. 
\item $K({}^{d}\C T)\r K({}^{v}\C T)$ for any triangulated category $\C T$ \cite{kttcneeman}. 
\item $K(\C A)\r K({}^{d}\C T)$ for any triangulated category $\C T$ with a $t$-structure with heart $\C A$ \cite{kttcneeman}.

\addtocounter{enumi}{1}

\item An equivalence $DK(\C W)\st{\sim}\r K(\mathbb{D}_{\C W})$ for any Waldhausen category $\C W$ with cylinders satisfying the two-out-of-three axiom \cite{sdckt2,malt}.

\item $K(\mathbb{D})\r K({}^{s}\mathbb{D}(*))$ for any triangulated derivator $\mathbb{D}$ \cite{ktdt}.

\item An equivalence $K(\C E)\st{\sim}\r K(C^{b}(\C E))$ for any exact category $\C E$ \cite{tckt}. 

\item $K(C^{b}(\C E))\r K({}^{s}D^{b}(\C E))$ for any exact category $\C E$, compare \cite{cts}.

\item An equivalence $K(\C A)\st{\sim}\r K(\operatorname{Gr}^b\!\!\C{A})$ for any abelian category $\C A$ \cite{kttcneeman}.
\end{enumerate}
Applying $\lambda_0$ to these maps we obtain some stable quadratic module morphisms described in Section \ref{comparase}. The equivalence (9) together with Theorems \ref{main1} and \ref{cosa} shows that determinant functors on an abelian category $\C A$ concide essentially with graded determinant functors. 
\end{exm}

\subsection{Generators and (some) relations for $\pi_{1}$}\label{genrel2}

In this section we extend the results in \cite{k1wc} to the unified context introduced in this paper. We fix an $S_\bullet$-category $\C{C}_{\bullet}$ satisfying the following additional property: 
\begin{itemize}
\item The functor sending an $(n+1)$-simplex to its ${n+2}$ faces,
$$\phi\colon\C{C}_{n+1}\To  \C{C}_{n}\times_{\C{C}_{n-1}} \st{n+2}{\cdots\cdots}\times_{\C{C}_{n-1}} \C{C}_{n},\quad n\geq 0,$$
 is a fibration of categories, i.e.~it satisfies the isomorphism lifting property: any isomorphism $\phi(x)\r y$ in the target is the image by $\phi$ of an isomorphism $x\r x'$ in the source, in particular $y=\phi(x')$ is in the image of $\phi$. 
\end{itemize}
This property is satisfied by Example \ref{mainexm} (1,3--7), but it need not be satisfied by Example \ref{mainexm} (8) on derivators.

\begin{defn}
A \emph{triangle} $\Delta$ in $\C{C}_{\bullet}$ is just an object of $\C{C}_{2}$. A \emph{weak triangle} $(\Delta,f)$ in $\C{C}_{\bullet}$ consists of a triangle $\Delta$  and a morphism $f\colon C\st{\sim}\r d_{0}\Delta$ in $\we(\C{C}_{1})$. We denote,
$$[\Delta,f]=[\Delta]+[f]^{[d_{2}\Delta]}\in\D{D}_{1}(\C{C}_{\bullet}).$$

A \emph{pair of triangles} $(\Delta_{1}; \Delta_{2})$ consists of two  triangles, $\Delta_{1}$ and  $\Delta_{2}$, with the same faces $d_i\Delta_1=d_i\Delta_2$, $i=0,1,2$. 
A pair of triangles yields an element,
$$[\Delta_{1}; \Delta_{2}]=-[\Delta_{1}]+[\Delta_{2}]\in\pi_{1}\D{D}_{*}(\C{C}_{\bullet}).$$
A \emph{pair of weak triangles} $(\Delta_{1},f_{1}; \Delta_{2},f_{2})$ consists of two weak triangles, $(\Delta_{1},f_{1})$ and  $(\Delta_{2},f_{2})$, such that $\Delta_{1}$ and $\Delta_{2}$ have the same two last edges,
$d_{1}\Delta_{1}=d_{1}\Delta_{2}$, $d_{2}\Delta_{1}=d_{2}\Delta_{2}$,
and  $f_{1}$ and $f_{2}$ have the same source,
$$d_{0}\Delta_{1}  \st{f_{1}}\longleftarrow C\st{f_{2}}\To  d_{0}\Delta_{2}.$$
Any pair of weak triangles yields an element,
$$[\Delta_{1},f_{1}; \Delta_{2},f_{2}]=-[\Delta_{1},f_{1}]+[\Delta_{2},f_{2}]\in\pi_{1}\D{D}_{*}(\C{C}_{\bullet}).$$

(Pairs of) triangles are regarded as (pairs of) weak triangles, $\Delta=(\Delta,1_{d_0\Delta})$ and  $(\Delta_{1}; \Delta_{2})=(\Delta_{1},1_{d_0\Delta_1}; \Delta_{2},1_{d_0\Delta_2})$.

Notice that a trivial pair of weak triangles is trivial, i.e.
\begin{equation*}\tag{S2}
[\Delta,f; \Delta,f]=0\in\pi_{1}\D{D}_{*}(\C{C}_{\bullet}).
\end{equation*}
\end{defn}

\begin{thm}\label{otromain}
Any element in $\pi_{1}\D{D}_{*}(\C{C}_{\bullet})$ is a pair of weak triangles.
\end{thm}

Now we extend to our unified framework all results in \cite{k1wc} needed so that the proof of \cite[Theorem 2.1]{k1wc} works for Theorem \ref{otromain}. 


\begin{defn}
A \emph{$3\times 3$ diagram} in $\C{C}_{\bullet}$ consists of four objects $\Theta_{1}$, $\Theta_{2}$, $\Theta_{3}$, $\Theta_{4}$ in $\C{C}_{3}$ such that,
\begin{align*}
d_{2}\Theta_{1}&=d_{2}\Theta_{2},&
d_{1}\Theta_{3}&=d_{1}\Theta_{4},\\
d_{1}\Theta_{1}&=d_{3}\Theta_{3},&
d_{1}\Theta_{2}&=d_{3}\Theta_{4},\\
d_{0}d_{3}\Theta_{1}&=d_{0}d_{1}\Theta_{2},&
d_{0}d_{1}\Theta_{1}&=d_{0}d_{3}\Theta_{2},\\
d_{0}\Theta_{1}&=s_{1}d_{0}d_{3}\Theta_{1}\sqcup s_{0}d_{0}d_{1}\Theta_{1},&
d_{0}\Theta_{2}&=s_{0}d_{0}d_{1}\Theta_{2}\sqcup s_{1}d_{0}d_{3}\Theta_{2}.
\end{align*}
\end{defn}

\begin{prop}\label{tresportres}
Given a $3\times 3$ diagram in $\C{C}_{\bullet}$ the following equation holds in~$\D{D}_{1}(\C{C}_{\bullet})$,
\begin{align*}
\grupo{[d_{0}d_{1}\Theta_{1}],[d_{0}d_{3}\Theta_{1}]}={}&
-[d_{3}\Theta_{1}]-[d_{0}\Theta_{3}]^{[d_{2}d_{3}\Theta_{3}]}
-[d_{2}\Theta_{3}]\\
&+[d_{2}\Theta_{4}]+[d_{0}\Theta_{4}]^{[d_{2}d_{3}\Theta_{4}]}
+[d_{3}\Theta_{2}].
\end{align*}
\end{prop}

This result follows straightforwardly from (R8) and (R9). 

\begin{defn}
A \emph{pair of $3\times 3$ diagrams} consists of two $3\times 3$ diagrams, 
$\Theta_{1}$, $\Theta_{2}$, $\Theta_{3}$, $\Theta_{4}$ and $\Theta_{1}'$, $\Theta_{2}'$, $\Theta_{3}'$, $\Theta_{4}'$, such that for $i=0,1,2$,
\begin{align*}
d_{i}d_{3}\Theta_{1}&=d_{i}d_{3}\Theta_{1}',&
d_{i}d_{3}\Theta_{2}&=d_{i}d_{3}\Theta_{2}',&
d_{i}d_{0}\Theta_{3}&=d_{i}d_{0}\Theta_{3}',\\
d_{i}d_{2}\Theta_{3}&=d_{i}d_{2}\Theta_{3}',&
d_{i}d_{0}\Theta_{4}&=d_{i}d_{0}\Theta_{4}',&
d_{i}d_{2}\Theta_{4}&=d_{i}d_{2}\Theta_{4}'.
\end{align*}
\end{defn}

\begin{cor}\label{hartura}
For any pair of 
$3\times 3$ diagrams in $\C{C}_{\bullet}$, 
$\Theta_{1}$, $\Theta_{2}$, $\Theta_{3}$, $\Theta_{4}$ and $\Theta_{1}'$, $\Theta_{2}'$, $\Theta_{3}'$, $\Theta_{4}'$, the following relation between pairs of triangles in $\pi_{1}\D{D}_{*}(\C{C}_{\bullet})$ holds,
\begin{align*}
&\hspace{-20pt}[d_{3}\Theta_{1};d_{3}\Theta_{1}']-[d_{2}\Theta_{4};d_{2}\Theta_{4}']+[d_{0}\Theta_{3};d_{0}\Theta_{3}']\\
={}&[d_{3}\Theta_{2};d_{3}\Theta_{2}']-[d_{2}\Theta_{3};d_{2}\Theta_{3}']+[d_{0}\Theta_{4};d_{0}\Theta_{4}'].
\end{align*}
\end{cor}

This is an easy consequence of Proposition \ref{tresportres}, compare \cite[Theorem 3.1]{k1wc}.

\begin{defn}
A \emph{weak $3\times 3$ diagram} in $\C{C}_{\bullet}$ consists of a $3\times 3$ diagram  $\Theta_{1}$, $\Theta_{2}$, $\Theta_{3}$, $\Theta_{4}$; two objects $\Delta_{1}$, $\Delta_{2}$ in $\C{C}_{2}$ together with morphisms,
\begin{align*}
w_{1}&\colon\Delta_{1}\st{\sim}\To  d_{0}\Theta_{3},&
w_{2}&\colon\Delta_{2}\st{\sim}\To  d_{0}\Theta_{4};
\end{align*}
and a commutative diagram in $\we(\C{C}_{1})$, 
$$\xymatrix{d_{0}d_{1}\Theta_{3}&d_{0}\Delta_{2}\ar[l]^-{\sim}_-{d_{0}w_{2}}\\
d_{0}\Delta_{1}\ar[u]_-{\sim}^-{d_{0}w_{1}}&C''\ar[l]_-{\sim}^{w^{C}}\ar[u]^-{\sim}_-{w''}}$$
A \emph{pair of weak $3\times 3$ diagrams} consists of two weak $3\times 3$ diagrams, 
the first one as before and the second one given by $\Theta_{1}'$, $\Theta_{2}'$, $\Theta_{3}'$, $\Theta_{4}'$, 
\begin{align*}
w_{1}'&\colon\Delta_{1}'\st{\sim}\To  d_{0}\Theta_{3}',&
w_{2}'&\colon\Delta_{2}'\st{\sim}\To  d_{0}\Theta_{4}';&&\begin{array}{c}\xymatrix{d_{0}d_{1}\Theta_{3}'&d_{0}\Delta_{2}'\ar[l]^-{\sim}_-{d_{0}w_{2}'}\\
d_{0}\Delta_{1}'\ar[u]_-{\sim}^-{d_{0}w_{1}'}&C''\ar[l]_-{\sim}^{(w^{C})'}\ar[u]^-{\sim}_-{w'''}}\end{array}
\end{align*}
such that, for $i=1,2$,
\begin{align*}
d_{i}d_{3}\Theta_{1}&=d_{i}d_{3}\Theta_{1}',&
d_{i}d_{3}\Theta_{2}&=d_{i}d_{3}\Theta_{2}',&
d_{i}d_{2}\Theta_{3}&=d_{i}d_{2}\Theta_{3}',\\
d_{i}d_{2}\Theta_{4}&=d_{i}d_{2}\Theta_{4}',&
d_{i}\Delta_{1}&=d_{i}\Delta_{1}',&
d_{i}\Delta_{2}&=d_{i}\Delta_{2}'.
\end{align*}
\end{defn}

\begin{prop}\label{weaktresportres}
Given a weak $3\times 3$ diagram in $\C{C}_{\bullet}$ as above, the following equation holds in~$\D{D}_{1}(\C{C}_{\bullet})$,
\begin{align*}
\grupo{[d_{2}\Delta_{1}],[d_{2}\Delta_{2}]}={}&
-[d_{3}\Theta_{1},d_{2}\Delta_{2}]-[\Delta_{1},w^{C}]^{[d_{2}d_{3}\Theta_{3}]}-[d_{2}\Theta_{3},d_{1}w_{1}]\\
&+[d_{2}\Theta_{4},d_{1}w_{2}]
+[\Delta_{2},w'']^{[d_{2}d_{3}\Theta_{4}]}
+[d_{3}\Theta_{2},d_{2}w_{1}].
\end{align*}
\end{prop}

The proof of \cite[Proposition 1.6]{k1wc} also works in this case.

\begin{cor}
For any pair of weak $3\times 3$ diagrams in $\C{C}_{\bullet}$ as in the previous definition, the following relation between pairs of weak triangles in $\pi_{1}\D{D}_{*}(\C{C}_{\bullet})$ holds,
\begin{align*}\tag{S1}
&[d_{3}\Theta_{1},d_{2}\Delta_{2};d_{3}\Theta_{1}',d_{2}\Delta_{2}']-[d_{2}\Theta_{4},d_{1}w_{2};d_{2}\Theta_{4}',d_{1}w_{2}']+[\Delta_{1},w^{C};\Delta_{1}',(w^{C})']\\
&=[d_{3}\Theta_{2},d_{2}w_{1};d_{3}\Theta_{2}',d_{2}w_{1}']-[d_{2}\Theta_{3},d_{1}w_{1};d_{2}\Theta_{3}',d_{1}w_{1}']+[\Delta_{2},w'';\Delta_{2}',w'''].
\end{align*}
\end{cor}

This is an easy consequence of Proposition \ref{weaktresportres}, compare \cite[Theorem 3.1]{k1wc}.

\begin{cor}\label{sumalos}
Given two  pairs of weak triangles in $\C{C}_{\bullet}$, $(\Delta_{1},f_{1};\Delta_{1}',f_{1}')$ and $(\Delta_{2},f_{2};\Delta_{2}',f_{2}')$, the following relation holds in $\pi_{1}\D{D}_{*}(\C{C}_{\bullet})$,
\begin{align*}
[\Delta_{1}\sqcup \Delta_{2},f_{1}\sqcup f_{2};\Delta_{1}'\sqcup \Delta_{2}',f_{1}'\sqcup f_{2}']&=[\Delta_{1},f_{1};\Delta_{1}',f_{1}']+[\Delta_{2},f_{2};\Delta_{2}',f_{2}'].
\end{align*}
\end{cor}

\begin{proof}
Denote by $C$ the source of $f_{1}$ and $f_{1}'$, and by $C'$ the source of $f_{2}$ and $f_{2}'$. This corollary follows by 
applying the previous one to the following pair of weak $3\times 3$ diagrams:
\begin{align*}
\Theta_{1}^{i}&= s_{0}^{2}d_{2}\Delta_{i}\sqcup s_{2}\Delta_{i}',& \Theta_{2}^{i}&=  s_{0}s_{1}d_{2}\Delta_{i}\sqcup s_{1}\Delta_{i}', \\
\Theta_{3}&= s_{0}\Delta_{i}\sqcup s_{1}^{2}d_{1}\Delta_{i}', & \Theta_{4}&= s_{1}\Delta_{i}\sqcup s_{2}\Delta_{i}',\\
\Delta_{1}^{i}&=d_{0}\Theta_{3}^{i},&w_{1}^{i}&=1_{d_{0}\Theta_{3}^{i}},\\
\Delta_{2}^{i}&=s_{0}C\sqcup s_{1}C',&w_{2}&=s_{0}f_{i}\sqcup s_{1}f_{i}',\\
w''_{i}&=1_{C},&w^{C}_{i}&=f_{i}.
\end{align*}
\end{proof}

The following result is completely new. It yields a smaller presentation of $\D{D}_*(\C{C}_\bullet)$ which can be applied in some important situations.

\begin{prop}\label{porquillo}
If weak equivalences in $\C{C}_\bullet$ are isomorphisms, then $\D{D}_*(\C{C}_\bullet)$ has a presentation with generators (G1) and (G3) and relations (R2), (R8), (R9) and $[s_0^2(*)]=0$.
\end{prop}

\begin{proof}
This proof consists of an intensive use of the isomorphism lifting property in $\C{C}_{\bullet}$  assumed at the beginning of this section. 
Any isomorphism $f\colon \X\st{\sim}\r \X'$ in $\C{C}_1$ can be lifted to an isomorphism $\Phi(f)\colon s_1(\X)\st{\sim}\r\Delta(f)$ in $\C{C}_2	$ such that $d_0\Phi(f)$ is degenerate,
\begin{align*}
d_1\Phi(f)&=f,& d_2\Phi(f)&=1_\X.
\end{align*}
By (R7), $[f]=[\Delta(f)]$, therefore $\D{D}_*(\C{C})$ is generated by (G1) and (G3). By Remark \ref{lessrel}, we now just have to check that (R6) and (R7) are redundant. 

Given two composable isomorphisms in $\C{C}_1$,
$$\X\st{f}\To  \Y\st{g}\To  \Z,$$
we can take an isomorphism $\Xi^{f,g}\colon s_2s_1(\X)\st{\sim}\r\Theta(f,g)$ in $\C{C}_3$ such that, $d_0\Xi^{f,g}$ is degenerate,
\begin{align*}
d_1\Xi^{f,g}&=\Phi(g),& d_2\Xi^{f,g}&=\Phi(gf),& d_3\Xi^{f,g}&=\Phi(f). 
\end{align*}
If we apply (R8) to $\Theta(f,g)$ we obtain (R6).

Suppose now that $\Phi\colon\Delta_1\r\Delta_2$ is an isomorphism in $\C{C}_2$. We choose two isomorphisms in $\C{C}_2$,
$$\Delta_1\st{\Psi^1}\To \Delta'\st{\Psi^2}\To \Delta'',$$ 
with,
\begin{align*}
d_0(\Psi^1)&=1_{d_0\Delta_1},&d_1(\Psi^1)&=d_1\Phi,&d_2(\Psi^1)&=1_{d_2\Delta_1},\\
d_0(\Psi^2)&=d_2\Phi,&d_1(\Psi^2)&=1_{d_1\Delta_2},&d_2(\Psi^2)&=1_{d_2\Delta_1},
\end{align*}
and two isomorphisms in $\C{C}_3$, $$\Theta_1(\Phi)\st{\Xi^1}\longleftarrow s_1(\Delta)\st{\Xi^2}\To \Theta_2(\Phi),$$ with, 
\begin{align*}
d_0\Xi^1&=1_{s_1d_0\Delta_1},&d_1\Xi^1&=\Phi(d_1\Phi),&d_2\Xi^1&=\Psi^1,&d_3\Xi^1&=1_{\Delta_1},\\
d_0\Xi^2&=1_{d_0\Delta_1},&d_1\Xi^2&=\Phi,&d_2\Xi^2&=\Psi^2\Psi^1,&d_3\Xi^2&=\Phi(d_2\Phi).
\end{align*}
We also consider $\Theta_3(\Phi)=s_2\Delta(d_1\Phi)$ and $\Theta_4(\Phi)=s_2\Delta_2$.

Now (R7) follows from 
Proposition \ref{tresportres} applied to the $3\times 3$ diagram $\Theta_1(\Phi)$, $\Theta_2(\Phi)$, $\Theta_3(\Phi)$, $\Theta_4(\Phi)$. Recall that Proposition \ref{tresportres} only uses (R8) and (R9), hence we are done.
\end{proof}

\begin{defn}\label{R10}
an $S_{\bullet}$-category $\C{C}_{\bullet}$ has \emph{functorial coproducts} if  $\C{C}_{n}$, $n\geq 0$, is endowed with a monoidal structure $+$, strictly compatible with face and degeneracy functors, 
which is strictly associative,
\begin{align*}
(\X+\Y)+\Z&=\X+(\Y+\Z),
\end{align*}
strictly unital with unit object $s_{0}^{n}(*)$, 
\begin{align*}
{s_{0}^{n}(*)}+\X=X&=\X+{s_{0}^{n}(*)},
\end{align*}
and such that
$$\X\;=\;\X+{s_{0}^{n}(*)}\longrightarrow \X+\Y\longleftarrow {s_{0}^{n}(*)}+\Y\;=\;\Y$$
is always a coproduct diagram. Recall that $s_{0}^{n}(*)$ is an initial object in $\C C_{n}$, $n\geq 0$.

We define the stable quadratic module $\D{D}_*^+(\C{C}_\bullet)$ as the quotient of $\D{D}_*(\C{C}_\bullet)$ by the following extra relation,
\begin{itemize}
 \item[(R10)] $[s_0(\X)+ s_1(\Y)]=0$ for any pair of objects $\X$ and $\Y$ in $\C{C}_1$.
\end{itemize}
\end{defn}

\begin{prop}\label{sumnormalization}
Let $\C{C}_{\bullet}$ be an $S_{\bullet}$-category with functorial coproducts. If the set of objects of~$\C{C}_1$ is free as a monoid under $+$, then the natural projection,
$$\D{D}_*(\C{C}_\bullet)\onto \D{D}_*^+(\C{C}_\bullet),$$
is a weak equivalence. It actually forms part of a  strong deformation retraction.
\end{prop}

The proof is the same as the proof of \cite[Theorem 4.2]{k1wc} with the obvious change of terminology. The hypothesis is not very strong.

\begin{prop}\label{sum}
For any $S_{\bullet}$-category $\C{C}_\bullet$ there is another one $\C{C}'_\bullet$ with functorial coproducts whose simplicial monoid of objects is freely generated by the simplicial set of objects in $\C{C}_\bullet$ mod $*$ and its degeneracies, and such that the natural simplicial functor $\C{C}_\bullet\r\C{C}'_\bullet$ is an equivalence levelwise and restricts to a levelwise equivalence $\we(\C{C}_\bullet)\r\we(\C{C}'_\bullet)$.
\end{prop}

For the proof of this proposition one applies levelwise the $\operatorname{Sum}(-)$ construction in \cite[Proposition 4.3]{k1wc}.

\begin{lem}\label{suma1}
Given two weak triangles $(\Delta,f)$ and $(\Delta',f')$ in an $S_{\bullet}$-category with functorial coproducts $\C{C}_{\bullet}$, if we denote $C$ and $C'$ the source of $f$ and $f'$, respectively, then the following relation holds in~$\D{D}_{1}^{+}(\C{C}_{\bullet})$,
\begin{align*}
[\Delta{+} \Delta',f{+} f']&=[\Delta,f]^{[d_{1}\Delta']}+[\Delta',f']+\grupo{[d_{2}\Delta],[C']}.
\end{align*}
\end{lem}

\begin{proof}
Apply Proposition \ref{weaktresportres} to
\begin{align*}
\Theta_{1}&= s_{0}^{2}d_{2}\Delta{+} s_{2}\Delta',& \Theta_{2}&=  s_{0}s_{1}d_{2}\Delta{+} s_{1}\Delta', \\
\Theta_{3}&= s_{0}\Delta{+} s_{1}^{2}d_{1}\Delta', & \Theta_{4}&= s_{1}\Delta{+} s_{2}\Delta',\\
\Delta_{1}&=d_{0}\Theta_{3},&w_{1}&=1_{d_{0}\Theta_{3}},\\
\Delta_{2}&=s_{0}C{+} s_{1}C',&w_{2}&=s_{0}f{+} s_{1}f',\\
w''&=1_{C},&w^{C}&=f.
\end{align*}
\end{proof}

\begin{cor}\label{suma2}
Given two triangles $\Delta$, $\Delta'$ in an $S_{\bullet}$-category with functorial coproducts $\C{C}_{\bullet}$ and two weak equivalences $f\colon \X\st{\sim}\r \Y$, $f'\colon \X'\st{\sim}\r \Y'$ in $\C{C}_{1}$, the following relations hold in~$\D{D}_{1}^{+}(\C{C}_{\bullet})$,
\begin{align*}
[\Delta+\Delta']&=[\Delta]^{[d_{1}\Delta']}+[\Delta']+\grupo{[d_{2}\Delta],[d_{0}\Delta']},\\
[f+f']&=[f]^{[\Y']}+[f'].
\end{align*}
\end{cor}

\begin{lem}\label{simetria}
Let $\C{C}_{\bullet}$ be an $S_{\bullet}$-category with functorial coproducts, and $\X_{1},\dots, \X_{n}$  objects in $\C{C}_{1}$. 
Given a permutation of $n$ elements, $\sigma\in\operatorname{Sym}(n)$, we denote
$$\sigma_{\X_{1},\dots,\X_{n}}\colon \X_{\sigma_{1}}+\cdots+\X_{\sigma_{n}}\To  \X_{1}+\cdots \X_{n}$$
the isomorphism permuting the factors of the coproduct. The following formula holds in $\D{D}_{1}^{+}(\C{C})$,
\begin{align*}
[\sigma_{\X_{1},\dots,\X_{n}}]&=\sum_{
\begin{array}{c}
{}\vspace{-15pt}\\
\scriptstyle i>j \vspace{-5pt}\\
\scriptstyle \sigma_{i}<\sigma_{j}
\end{array}}\grupo{[\X_{\sigma_{i}}],[\X_{\sigma_{j}}]}
\end{align*}
\end{lem}

This lemma can be proved as \cite[Lemma 4.9]{k1wc}.

\begin{proof}[Proof of Theorem \ref{otromain}]
In this proof we translate the argument in the proof of \cite[Theorem 2.1]{k1wc} to our unified framework. By Proposition \ref{sum} we can suppose that $\C{C}_{\bullet}$ has functorial coproducts in such a way that the monoid of objects of $\C{C}_{1}$ is freely generated by a set $\mathbb{S}$ of non-degenerate objects, so we can work with $\D{D}_{1}^{+}(\C{C}_{\bullet})$ by Proposition \ref{sumnormalization}.

Any $x\in\D{D}_{1}^{+}(\C{C}_{\bullet})$ is a sum of triangles and weak equivalences in $\C{C}_{1}$ with coefficients~$\pm 1$. Therefore, by Corollary \ref{suma2}, the following equation holds,
\begin{align*}
x={}&-[f\colon \X\st{\sim}\r \Y]-[\Delta]+[\Delta']+[f'\colon \X'\st{\sim}\r \Y']&&\mod \grupo{\cdot,\cdot}\\
={}&-[f+1_{\X'}]-[\Delta+s_{0}d_{0}\Delta'+s_{1}d_{2}\Delta']\\
&+[s_{0}d_{0}\Delta+s_{1}d_{2}\Delta+\Delta']+[1_{\X}+f']&&\mod \grupo{\cdot,\cdot}.
\end{align*}

If $\partial(x)=0$ modulo commutators then, 
\begin{align*}
0={}&-[\X+\X']+[\Y+\X']-[d_{2}\Delta+d_{2}\Delta']-[d_{0}\Delta+d_{0}\Delta']+[d_{1}\Delta+d_{0}\Delta'+d_{2}\Delta']\\
&-[d_{0}\Delta+d_{2}\Delta+d_{1}\Delta']+[d_{0}\Delta+d_{0}\Delta']+[d_{2}\Delta+d_{2}\Delta']-[\X+\Y']+[\X+\X']\\
&\mod [\cdot,\cdot],
\end{align*}
and therefore,
\begin{align*}
[\Y+\X'+d_{1}\Delta+d_{0}\Delta'+d_{2}\Delta']={}&[\X+\Y'+d_{0}\Delta+d_{2}\Delta+d_{1}\Delta']\mod [\cdot,\cdot].
\end{align*}

The quotient of $\D{D}_{0}^{+}(\C{C}_{\bullet})$ by the commutator subgroup is the free abelian group with basis $\mathbb{S}$, hence there are objects $S_{1},\dots, S_{n}\in\mathbb{S}$ and a permutation $\sigma\in\operatorname{Sym}(n)$ with,
\begin{align*}
\Y+\X'+d_{1}\Delta+d_{0}\Delta'+d_{2}\Delta'&=S_{1}+\cdots+S_{n},\\
\X+\Y'+d_{0}\Delta+d_{2}\Delta+d_{1}\Delta'&=S_{\sigma_{1}}+\cdots+S_{\sigma_{n}}.
\end{align*}
In particular, there is an isomorphism,
$$\sigma_{S_{1},\dots,S_{n}}\colon \X+\Y'+d_{0}\Delta+d_{2}\Delta+d_{1}\Delta'\To  \Y+\X'+d_{1}\Delta+d_{0}\Delta'+d_{2}\Delta'.$$

By the isomorphism lifting property, there exists an isomorphism in $\C{C}_{2}$, $$\Phi\colon s_{0}\X+s_{0}\Y'+s_{0}d_{0}\Delta+s_{1}d_{2}\Delta+\Delta'\st{\sim}\To  \Delta_{2},$$
such that $d_{0}\Phi$ and $d_{2}\Phi$ are identity morphisms and $d_{1}\Phi=\sigma_{S_{1},\dots,S_{n}}$.

By Corollaries \ref{suma2} and \ref{simetria}, modulo the image of $\grupo{\cdot,\cdot}$,
\begin{align*}
x={}&-[f+1_{\X'}+1_{d_{0}\Delta+d_{0}\Delta'}]-[s_{0}\Y+s_{0}\X'+\Delta+s_{0}d_{0}\Delta'+s_{1}d_{2}\Delta']\\
&+[s_{0}\X+s_{0}\Y'+s_{0}d_{0}\Delta+s_{1}d_{2}\Delta+\Delta']+[1_{\X}+f'+1_{d_{0}\Delta+d_{0}\Delta'}] \\
={}&-[f+1_{\X'}+1_{d_{0}\Delta+d_{0}\Delta'}]-[s_{0}\Y+s_{0}\X'+\Delta+s_{0}d_{0}\Delta'+s_{1}d_{2}\Delta']\\
&+[\sigma_{S_{1},\dots,S_{n}}]\\
&+[s_{0}\X+s_{0}\Y'+s_{0}d_{0}\Delta+s_{1}d_{2}\Delta+\Delta']+[1_{\X}+f'+1_{d_{0}\Delta+d_{0}\Delta'}] \\
={}&-[f+1_{\X'}+1_{d_{0}\Delta+d_{0}\Delta'}]-[s_{0}\Y+s_{0}\X'+\Delta+s_{0}d_{0}\Delta'+s_{1}d_{2}\Delta']\\
&+[\Delta_{2}]+[1_{\X}+f'+1_{d_{0}\Delta+d_{0}\Delta'}] \\
={}&-[f+1_{\X'}+1_{d_{0}\Delta+d_{0}\Delta'}]^{[d_{2}\Delta+d_{2}\Delta']}-[s_{0}\Y+s_{0}\X'+\Delta+s_{0}d_{0}\Delta'+s_{1}d_{2}\Delta']\\
&+[\Delta_{2}]+[1_{\X}+f'+1_{d_{0}\Delta+d_{0}\Delta'}]^{[d_{2}\Delta+d_{2}\Delta']} \\
={}&[s_{0}\Y+s_{0}\X'+\Delta+s_{0}d_{0}\Delta'+s_{1}d_{2}\Delta',f+1_{\X'}+1_{d_{0}\Delta+d_{0}\Delta'};\\
&\Delta_{2},1_{\X}+f'+1_{d_{0}\Delta+d_{0}\Delta'}],
\end{align*}
i.e.~$x$ is represented by a pair of weak triangles modulo the image of $\grupo{\cdot,\cdot}$,
$$x=[\Delta_{1},f_{1};\Delta_{2},f_{2}]+y,\quad y\text{ in the image of }\grupo{\cdot,\cdot}.$$

Assume now that $\partial(x)=0$. Then $\partial(y)=0$ as well, therefore by \cite[Lemma 5.1]{k1wc} $y=\grupo{a,a}$ for some $a\in\D{D}_{0}(\C{C}_{\bullet})$, which is the free group of nilpotency class $2$ with basis $\mathbb{S}$. Since $y$ only depends on  $a$ mod $2$, we can suppose that $a=[S'_{1}]+\cdots+[S_{m}']=[M]$, $M=S_{1}'+\cdots +S_{m}'$, $S'_{i}\in\mathbb{S}$, therefore,
\begin{align*}
y&=\grupo{[M],[M]}=[s_{0}M+s_{1}M;s_{1}M+s_{0}M],
\end{align*}
is  a pair of triangles, in particular a pair of weak triangles, so $x$ is also a pair of weak triangles by Corollary \ref{sumalos}.
\end{proof}



\providecommand{\bysame}{\leavevmode\hbox to3em{\hrulefill}\thinspace}
\providecommand{\MR}{\relax\ifhmode\unskip\space\fi MR }
\providecommand{\MRhref}[2]{%
  \href{http://www.ams.org/mathscinet-getitem?mr=#1}{#2}
}
\providecommand{\href}[2]{#2}

\end{document}